\documentclass[11pt]{amsart}
\usepackage{amssymb}
\usepackage{amsmath}
\usepackage{verbatim}
\newtheorem{theorem}{Theorem}[section]
\newtheorem{lemma}[theorem]{Lemma}
\newtheorem{remark}[theorem]{Remark}
\newtheorem{definition}[theorem]{Definition}

\newtheorem{corollary}[theorem]{Corollary}
\newtheorem{proposition}[theorem]{Proposition}

\newtheorem{lem-def}[theorem]{Lemma-Definition}
\numberwithin{equation}{section}
\newcommand{\sst}[1]{\vskip 1mm \noindent\makebox[4mm][r]{\bf #1} \hspace{1mm}}
\newcommand{\stst}[1]{\vskip 1mm \noindent\makebox[4mm][r]{\bf #1} \hspace{6mm}}
\newcommand{\ststst}[1]{\vskip 1mm \noindent\makebox[4mm][r]{\bf #1} \hspace{11mm}}
\newcommand{\stststst}[1]{\vskip 1mm \noindent\makebox[4mm][r]{\bf #1} \hspace{16mm}}
\newcommand{\ststststst}[1]{\vskip 1mm \noindent\makebox[4mm][r]{\bf #1} \hspace{21mm}}

\def\as#1{\renewcommand\arraystretch{#1}}
\def\ay{A[y]}
\def\aiy#1{A_{#1}[y]}
\def\bb{\mathcal{B}}
\def\dsc{\operatorname{disc}}
\def\F{\mathbb F}
\def\fm{\F_\m}
\def\fmy{\fm[y]}
\def\ga{\gamma}

\def\gt{\ga_\ty}

\def\imp{\ \Longrightarrow\ }

\def\iso{\ \lower .6ex\hbox{$\stackrel{\lra}{\mbox{\tiny $\sim\,$}}$}\ }
\def\j{\mathbf{j}}
\def\la{\lambda}

\def\ll{\mathcal{L}}
\def\lra{\longrightarrow}

\def\m{\mathfrak{m}}
\def\md#1{\ \mbox{\rm(mod }{#1})}
\def\mm{\mathcal{M}}
\def\mma{\m\in\mx(A)}
\def\mpa{\mx_p(A)}
\def\mn{\operatorname{Min}}
\def\mx{\operatorname{Max}}
\def\N{\mathbb N}
\def\n{\mathfrak{n}}

\def\om{\omega}
\def\oo{\mathcal{O}}
\def\op{\operatorname}
\def\ord{\operatorname{ord}}
\def\p{\mathfrak{p}}
\def\ptm{\pp_{\ty_\m}}
\def\pp{\mathcal{P}}
\def\pt{\pp_\ty}

\def\Q{\mathbb Q}
\def\qb{\overline{\Q}}
\def\R{\mathbb R}
\def\rd{\operatorname{red}}
\def\rdm{\rd_\m}

\def\rdn{\rd_N}
\def\rdp{\rd_p}
\def\rep{\operatorname{Rep}}

\def\sg{\mathfrak{s}}
\def\sii{\ \Longleftrightarrow\ }
\def\ss{\mathcal{S}}
\def\st{\mathcal{S}_{\op{trm}}}
\def\t{\theta}
\def\tcal{\mathcal{T}}

\def\ty{\mathbf{t}}

\def\v2{v_2^{(2)}}
\def\Z{\mathbb Z}
\def\zkp{\Z_{K,p}}
\def\zp{\Z_{(p)}}

\newtheorem{alg}[theorem]{Algorithm}
\newlength{\alginputwidth}
\newlength{\algtmp}
\newcommand{\Algo}[5]
            {
            \begin{alg}[#1] \label{#2}{$\;$}\rm
                \\
\mbox{\enspace}
                \rlap{\rm Input: }\phantom{\rm Output: }
\parbox[t]{\alginputwidth}{#3}
                \\
\mbox{\enspace}
                {\rm Output: }
\parbox[t]{\alginputwidth}{#4}
\parskip0pt
\begin{list}{}{\setlength{\leftmargin}{0pt}}
\item                #5
\end{list}
            \end{alg}
            \goodbreak}

\title[Square-free OM computation of global integral bases]{Square-free OM computation of global integral bases}

\author[Gu\`ardia]{Jordi Gu\`ardia}
\address{Departament de Matem\`atiques, Universitat Politècnica de Catalunya, \ Escola Polit\`ecnica Superior d'En\-ginyeria de Vilanova i la Geltr\'u, Av. V\'\i ctor Balaguer s/n. E-08800 Vilanova i la Geltr\'u, Catalonia, Spain}
\email{jordi.guardia-rubies@upc.edu}

\author[Nart]{Enric Nart}

\address{Departament de Matem\`{a}tiques,
         Universitat Aut\`{o}noma de Barcelona,
         Edifici C\\ E-08193 Bellaterra, Barcelona, Catalonia, Spain}
\email{nart@mat.uab.cat}

\date{July, 2018}

\makeatletter
\@namedef{subjclassname@2010}{%
\textup{2010} Mathematics Subject Classification}


\thanks{Partially supported by grants MTM2015-66180-R  and MTM2016-75980-P from the Spanish MEC}

\thanks{To appear in {\em Algebra and Number Theory}}
\begin{document}
\begin{abstract}
For a prime $p$, the OM algorithm finds the $p$-adic factorization of an irreducible polynomial $f\in\Z[x]$ in polynomial time. This may be applied to construct $p$-integral bases in the number field $K$ defined by $f$. In this paper, we adapt the OM techniques to work with a positive integer $N$ instead of $p$. As an application, we obtain an algorithm to compute global integral bases in $K$, which does not require a previous factorization of the discriminant of $f$.
\end{abstract}

\maketitle


\section*{Introduction}

The OM algorithm is a $p$-adic polynomial factorization method developed by Montes \cite{m}, based on ideas of Ore and MacLane \cite{mcla,mclb,ore}. For a number field $K$ with defining polynomial $f\in\Z[x]$, the OM algorithm may be applied to compute $p$-integral bases for the different prime factors $p$ of the discriminant $\dsc(f)$ of $f$ \cite{newapp,bases,gen}. Hence, it facilitates the computation of a global integral basis, as long as a factorization of $\dsc(f)$ is available. Since integer factorization is a heavy task, this approach is unfeasible in practice if $f$ has a large degree or large coefficients.

It is well-known that the determination of the ring of integers of a number field is a computationally hard problem. Indeed, it is polynomial-time equivalent to  the square-free factorization of integers (\cite{C1},\cite{C2},\cite{BL}). Chistov  described in (\cite{C1}, \cite{C2})  a polynomial time algorithm to determine integral bases of number fields given an oracle to find square-free factorizations. It is based upon a previous algorithm of factorization of polynomials over complete fields, where Newton polygons of first order are crulcial.

Buchmann and Lenstra, inspired in a normalization criterion by Grauert-Remmert and Zassenhaus, introduced an algorithm which constructs successive augmentations of a given order in $K$ \cite{BL}. Along this process, some splittings of $\dsc(f)$ are obtained as a by-product. At a certain stage, one knows that the maximal order has been reached if certain factors of  $\dsc(f)$ are squarefree.
Therefore, combined with an integer squarefree decomposition routine, this algorithm  computes global integral bases of $K$. T he algorithm is efficient in cases where the splitting of $\dsc(f)$, caused by the order augmentation steps, yields factors which are sufficiently small to admit a feasible squarefree decomposition.

In this paper, we adapt the OM techniques to work with a positive integer $N$ instead of a prime $p$. This provides a method to compute a global integral basis of  $K$ which behaves as the Buchmann-Lenstra algorithm.

For a given prime $p$, the classical OM algorithm computes a tree of \emph{types} whose leaves are in 1-1 correspondence with the $p$-adic irreducible factors of $f$. These types support valuations on $\Q(x)$ extending the $p$-adic valuation, which encode intrinsic data of the $p$-adic irreducible factors.

The tree of types is computed by a branching process based essentially on two tasks: cons\-truction of higher order Newton polygons of $f$ with respect to the supported valuations, and polynomial factorization, over certain finite fields, of the \emph{residual polynomials} of $f$ with respect to certain sides of the Newton polygons.

For an integer $N>1$, we consider similar {\em squarefree (SF) types} supporting pseudo-valuations on $\Q(x)$ extending the $N$-adic pseudo-valua\-tion.

Newton polygons with respect to pseudo-valuations are easy to define, but the residual polynomials of $f$, with respect to the sides of these polygons, have coefficients in certain finite $(\Z/N\Z)$-algebras. Hence, instead of polynomial factorization, we must use a  polynomial squarefree decomposition, which makes sense for these artinian algebras.

The new {\em SF-OM} algorithm computes a tree of SF-types containing intrinsic data of the $p$-adic irreducible factors of $f$,  simultaneously for all prime factors $p$ of $N$. When these data cannot be coherently integrated in one single tree of SF-types, the method yields a splitting of $N$ as a by-product.

This leads to a computation of a global integral basis of $K$, similar in spirit to the Buchmann-Lenstra algorithm. Along the branching process of SF-types, a splitting of $\dsc(f)$ is obtained, and at a certain stage, the accumulated data yield a global integral basis if certain factors are squarefree.

This method speeds up the computations with respect to the classical OM algorithm,  because the number of types to be considered is much smaller in general. For a double reason: a squarefree integer may have many prime factors, and a squarefree factor of a residual polynomial may have many irreducible factors too.

As a consequence, the amount of integer linear algebra necessary to glue the local bases to build up a global basis is drastically reduced.

We made an implementation in {\tt Magma} of the SF-OM algorithm, which is available on request to the authors. The implementation includes a library to work with towers of artinian algebras. Even if this part is much slower than the internal pre-compiled routines of {\tt Magma} for towers of finite fields, the new program is faster, in many cases, than our implementation of the classical Montes algorithm.

The content of the paper is as follows. Section 1 contains basic algorithms for polynomials over artinian $(\Z/N\Z)$-algebras. Section 2 introduces SF-types and their basic properties. Section 3 describes the SF-OM algorithm. In section 4, we show that for any prime factor $p$ of $N$, an SF-type determines a tree of classical $p$-types. This connection is used in section 5 to analyze arithmetic properties of $K$ encoded by SF-types. Section 6, inspired in \cite{bases}, describes how to use the SF-OM algorithm to compute a global integral basis of $K$. In section 7, we discuss a few concrete examples.

\subsection*{Acknowledgements}Claus Fieker suggested to us that an adequate development of Montes' methods modulo $N$, based on polynomial squarefree decomposition over $(\Z/N\Z)$-algebras, could lead to this kind of results. We are indebted to him for his fine intuition.

\section{Polynomial squarefree decomposition over $(\Z/N\Z)$-algebras}\label{secSFD}
\subsection{Inductive artinian algebras}\label{subsecInductive}
For an integer $N>1$, let $A_0=\Z/N\Z$.

Let $A$ be an artinian $A_0$-algebra. That is, $\op{Spec}(A)=\mx(A)$ is finite and discrete. Then, $A$ is isomorphic to a product of local $A_0$-algebras:
\begin{equation}\label{local}
A\simeq \prod\nolimits_{\mma}A_\m.
\end{equation}
Since each $A_\m$ has a finite length as an $A_0$-module, $A$ is a finite set.


The isomorphism (\ref{local}) induces an analogous decomposition of the polynomial ring $A[y]$ in one indeterminate $y$:
$$
\ay\simeq \prod\nolimits_{\mma}\aiy{\m}.
$$
Irreducible polynomials are easily characterized in terms of their components. Also, non-unique factorization pathologies in $\ay$ are easily explained by this decomposition.

For each $\mma$, let $\fm=A/\m$ be the residue field and
$$
\rdm\colon A\lra \fm,\qquad \rdm\colon \ay\lra \fmy
$$
the reduction modulo $\m$ homomorphisms for the rings $A$ and $A[y]$.

\begin{definition}
Let  $t\in\ay$ and denote by $\op{lc}(t)\in A$ its leading coefficient.

We say that $t\in \ay$ is \emph{unitary} if $\op{lc}(t)$ is a unit in $A$.

We say that $t$ is \emph{strongly unitary} if all its non-zero coefficients are units.

We say that $t$ is \emph{squarefree} if it is unitary and $\rdm(t)\in\fmy$ is squarefree for all $\mma$.
\end{definition}

If $s,t\in \ay$ and $t$ is unitary, the natural routine $q,r=\op{Quotrem}(s,t)$
 computes $q,r\in \ay$ such that $s=tq+r$ and $\deg r<\deg t$. Clearly,
\begin{equation}\label{qr}
 s\ay+t\ay=r\ay+t\ay.
\end{equation}

Also, a unitary  $t$ is a \emph{minimal} polynomial:
\begin{equation}\label{minimal}
\deg t=\mn\{\deg s\mid s\in t\ay, s\ne0\}.
\end{equation}

In particular, for a unitary $t$ of positive degree, the chain of ideals generated by the powers of $t$ is strictly decreasing:
$$
\ay\supsetneq t\ay\supsetneq t^2\ay\supsetneq \cdots
$$
Hence, for all nonzero $s\in\ay$ we may define
$$\ord_t(s)=k \quad \mbox{ if }\quad s\in t^k\ay\setminus t^{k+1}\ay.
$$
We agree that $\ord_t(0)=\infty$.

\begin{lemma}\label{augmentation}
Let $t\in\ay$ be a unitary polynomial of degree $n>0$. Let $B=\ay/(t)$ and denote by $z\in B$ the class of $y$. Then,  $1,\,z,\,\dots\,,\,z^{n-1}$ is an $A$-basis of $B$, and the natural map $A\to B$ is injective.
\end{lemma}

\begin{proof}
Let $M\subset B$ be the sub-$A$-module generated by $1,z,\dots,z^{n-1}$.
Since $z^n\in M$, we have $M=B$. On the other hand, (\ref{minimal}) shows that $1,z,\dots,z^{n-1}$ are $A$-linearly independent and the map $A\to B$ is injective.
\end{proof}

\begin{definition}\label{inductive}
An \emph{inductive} $A_0$-algebra of length $r\ge0$ is an artinian algebra $A$ which may be obtained by a chain of augmentations:
\begin{equation}\label{chain}
A_0\subset A_1\subset \cdots \subset A_r=A,\qquad A_{i+1}=\aiy{i}/(t_i), \quad 0\le i<r,
\end{equation}
for some squarefree, strongly unitary $t_i\in A_i[y]$ such that $t_i(0)\ne0$ for $i>0$.

The \emph{moduli sequence} of $A$ is the list $[t_{-1},t_0,\dots,t_{r-1}]$ of the different ``moduli", starting formally with $t_{-1}:=N$.
\end{definition}


\subsection{Polynomial squarefree decomposition over inductive algebras}


Let $A$ be an inductive algebra of length $r$ as in (\ref{chain}).
In this section, we describe a squarefree decomposition routine (SFD) for polynomials with coefficients in $A$.
We start with a $gcd_A$ routine in $\ay$. Given $s,t\in\ay$, the idea is to mimic Euclid's algorithm to compute a monic polynomial $d=\gcd_A(s,t)$ in $A[y]$ such that
$$
s\ay + t\ay = d\ay.
$$
This is not always possible, but when the method crashes it outputs a factorization in $\ay$ of one of the moduli of $A$.

The $\gcd_A$ routine is defined recursively, assuming that $\gcd_{A_{r-1}}$ is well-defined. At the bottom of the recursion, for $a\in A_0$ it makes sense to define $\gcd_{A_{-1}}(a,N)\in\Z_{>0}$ as the usual $\gcd(\tilde{a},N)$, for any lifting  $\tilde{a}\in\Z$ of $a$.

\Algo{$\gcd_A(s,t)$}{gcdA}
{$s,t\in\ay$, $t\ne0$, where $A$ is an inductive $A_0$-algebra of length $r\ge0$ with moduli sequence
$[t_{-1},t_0,\dots,t_{r-1}]$
}
{Either a proper factor of a modulus of $A$, or a monic $d\in\ay$ such that $s\ay + t\ay = d\ay$. }
{
\begin{enumerate}
\item \quad while $t\ne 0$ do
\item \quad\qquad $a\gets \op{lc}(t)$, \ $b\gets\gcd_{A_{r-1}}(a,t_{r-1})$
\item \quad\qquad  if $b\ne1$ then return $[r-1,b]$  else $t\leftarrow a^{-1}t$
\item \qquad\quad $q,\,r=\op{Quotrem}(s,t)$
\item \qquad\quad $s\gets t$, \ $t\gets r$
\item \quad return $s$
\end{enumerate}
}
\noindent{\bf Convention. }{\it Throughout the paper, we shall simply write $\gcd_A(s,t)=d$ to indicate that the routine $\gcd_A(s,t)$ does not crash and outputs $d$}.\medskip

In step 2, we identify $a\in A$ with a polynomial in $A_{r-1}[y]$.
For the routine to be consistent, the condition $\gcd_{A_{r-1}}(a,t_{r-1})=1$ should imply that $a$ is a unit in $A$. This fact, and the fundamental properties of $\gcd_A$ are contained in the following result.

\begin{lemma}\label{gcd}
For $a\in A$, if $\gcd_{A_{r-1}}(a,t_{r-1})=1$, then $a$ is a unit in $A$.

Moreover, if  $\gcd_A(s,t)=d$ for certain $s,t\in \ay$, then

\begin{enumerate}
\item[(a)] $\ s\ay+t\ay=d\ay$,
\item[(b)] \ For all  $\mma$, $\ \rdm(d)=\gcd\left(\rdm(s),\rdm(t)\right)$ in $\fmy$.
\end{enumerate}
\end{lemma}

\begin{proof}
We prove these properties for the algebras $A_0,A_1,\dots,A_r=A$ in a recursive way. For $r=0$, the condition $\gcd_{A_{-1}}(a,N)=1$ implies trivially that $a$ is a unit in $A_0$. Hence, the routine $\gcd_{A_0}$ is consistent.

Once a routine $\gcd_{A_i}$ is consistent, the equality (a) is a consequence of (\ref{qr})
applied to each division with remainder in step (4). Then, (b) follows from (a) by the compatibility of all operations with reduction modulo $\m$.

Finally, suppose $\gcd_{A_i}(a,t_i)=1$. From $a\aiy{i}+t_i\aiy{i}=\aiy{i}$, we deduce a B\'ezout identity $au+t_iv=1$, proving that $a$ is a unit in $A_{i+1}$.
\end{proof}

With this $\gcd_A$ routine in hand, we can mimic the standard squarefree decomposition routine for polynomials with coefficients in a field of characteristic zero \cite[\S20.3]{shoup}. If our polynomial in $ \ay$ has a not too large degree, the output will be correct.

\Algo{SFD}{sfdA}
{A unitary $f\in\ay$ with $\deg f<p$ for all prime factors $p$ of $N$.
}
{A proper factor of a modulus of $A$, or pairs $(s_1,\ell_1),\dots,(s_k,\ell_k)$, where $s_1,\dots,s_k\in\ay$ are monic squarefree pairwise coprime polynomials, and $f=\op{lc}(f)\,s_1^{\ell_1}\cdots s_k^{\ell_k}$ with $1\le \ell_1<\cdots <\ell_k$. }
{
\begin{enumerate}
\item \quad $f\gets f/\op{lc}(f)$, \  $g\gets f/\gcd_A(f,f')$
\item \quad $\ell\gets 1$, \ $L\gets [\;]$
\item \quad while $f\ne 1$ do
\item \qquad\quad $f\gets f/g$, \ $h\gets\gcd_A(f,g)$, \ $s\gets g/h$
\item \qquad\quad if $s\ne 1$ then append $(s,\ell)$ to $L$
\item \qquad\quad $g\gets h$, \ $\ell\gets \ell+1$
\item \quad return $L$
\end{enumerate}
}

Of course, although not specifically indicated, after every call to $\gcd_A$ the routine ends if we find a proper factor of a modulus of $A$.

The next result is an immediate consequence of Lemma \ref{gcd}.

\begin{lemma}\label{sfd}
Suppose the SFD routine does not crash and outputs a list of pairs $(s_1,\ell_1),\dots,(s_k,\ell_k)$. Then, for any $\mma$, the list $$(\rdm(s_1),\ell_1),\dots,(\rdm(s_k),\ell_k)$$ is the canonical squarefree decomposition of $\rdm(f)$ in $\fmy$.
\end{lemma}

This result justifies that the output polynomials $s_1,\dots,s_k\in\ay$ are squarefree and pairwise coprime.

\begin{remark}\rm
The condition $\deg f<p$ for all $p\mid N$ fits well with our purpose of constructing a global integral basis in a number field of degree $n$. In this context, $\deg f\le n$ and $N$ will be a positive divisor of the discriminant from which  all prime factors $p\le n$ have been removed (cf. section \ref{subsecGlobalBases}).
\end{remark}

\section{Types with respect to pseudo-valuations}\label{secNtypes}

Let $v\colon \oo\to\Z\cup\{\infty\}$ be a discrete valuation on an integral domain $\oo$, and let $\oo_v$ be the completion of $\oo$ at $v$.

A \emph{type} over $(\oo,v)$ is a discrete object parameterizing a certain equivalence class of monic irreducible polynomials in $\oo_v[x]$
\cite{gen}.
Types were introduced by Montes \cite{m} as a tool to perform ``higher dissections", a procedure foreseen by Ore aiming at a polynomial factorization algorithm  in $\oo_v[x]$.

The papers \cite{GMN, gen} contain variant definitions of a type, with slight changes in the normalization of certain data.

Let $N>1$ be an integer. This section introduces SF-types over $(\Z,\ord_N)$, where $\ord_N$ is the \emph{$N$-adic pseudo-valuation} introduced in section \ref{subsecPseudo}.

For $N=p$ prime, $\ord_p$ is a valuation. However, SF-types over $(\Z,\ord_p)$ do not coincide with those introduced by Montes, which become \emph{irreducible} types in our terminology (see Definition \ref{unramified}).

\subsection{Pseudo-valuations}\label{subsecPseudo}

\begin{definition}\label{qv}
A \emph{pseudo-valuation} on $\oo$ is a mapping $v\colon \oo \to \Z\cup\{\infty\}$ satisfying the following conditions for all $a,b\in\oo$:

\begin{enumerate}
\item $v(a)=\infty$ if and only if $a=0$.
\item $v(-1)=0$.
\item $v(ab)\ge v(a)+v(b)$.
\item $v(a+b)\ge \mn\{v(a),v(b)\}$.
\end{enumerate}
\end{definition}

These axioms imply that $v(-a)=v(a)$ for all $a\in \oo$, and equality holds in (4) when $v(a)\ne v(b)$ \cite[\S1]{sarussi}.

\begin{definition}\label{equiv}
We say that $a,b\in\oo$ are $v$-\emph{equivalent} if either $a=b=0$, or $v(a-b)>v(a)$. We then write $a\sim_vb$. Note that this implies $v(a)=v(b)$.
\end{definition}

A pseudo-valuation may be extended to the ring of fractions $\ss_v^{-1}\oo$, where $\ss_v$ is the multiplicatively closed subset of all non-zero \emph{stable} elements:
$$
\ss_v=\left\{a\in\oo\setminus\{0\}\mid v(ab)=v(a)+v(b)\mbox{ for all }b\in\oo\right\}.
$$

Also, $v$ determines a pseudo-valuation on the polynomial ring $\oo[x]$ by:
$$
v\left(a_0+a_1x+\cdots+a_sx^s+\cdots\right)=\mn\{v(a_s)\mid 0\le s\}.
$$

We fix an integer $N>1$, and let $\oo$ be either $\Z$ or the $p$-adic ring $\Z_p$, the latter case only when $N=p$ is a prime number.

Let $A_0=\Z/N\Z$ and denote reduction modulo $N$ by $\rdn\colon \oo\to A_0$. Denote still by
$\rdn\colon \oo[x]\to A_0[y]$ the mapping that reduces modulo $N$ the coefficients of a polynomial and changes the variable $x$ to $y$.

We define the $N$-adic pseudo-valuation
$$\ord_N\colon \oo\to\Z\cup\{\infty\},\qquad \ord_N(a)=k \ \mbox{ if }\ a\in N^k\oo\setminus N^{k+1}\oo.$$

Consider the \emph{residual polynomial operator}
\begin{equation}\label{R0}
R_0\colon \oo[x]\to A_0[y],\quad f\mapsto \rdn\left(f/N^{\ord_N(f)}\right).
\end{equation}
We agree that $R_0(0)=0$. Clearly,
$$\ss_{\ord_N}=\{a\in\oo\mid  R_0(a)\mbox{ is a unit in }A_0\},
$$

\begin{definition}\label{robust0}
A non-zero polynomial $f\in\oo[x]$ is said to be \emph{$N$-robust} if all its non-zero coefficients are stable.
\end{definition}


The next results follow easily from the definitions.

\begin{lemma}\label{v0equiv}
Two polynomials $f,h\in\oo[x]$ are $\ord_N$-equivalent if and only if $\ord_N(f)=\ord_N(h)$ and $R_0(f)=R_0(h)$.
\end{lemma}

\begin{lemma}\label{v0good}
Let $f,h\in\oo[x]$ and suppose that $f$ is $N$-robust.
\begin{enumerate}
\item $\ord_p(f)=\ord_p(N)\ord_N(f)$ for all prime factors $p$ of $N$.
\item $\ord_N(fh)=\ord_N(f)+\ord_N(h)$ and $R_0(fh)=R_0(f)R_0(h)$.
\end{enumerate}
\end{lemma}

Thus, robust polynomials are stable for the extension of $\ord_N$ to $\oo[x]$. However, they do not form a multiplicatively closed set.

\subsection{Types over $(\oo,\ord_N)$, also called SF-types}\label{subsecNtypes}

An \emph{SF-type $\ty=(t_0)$ of order zero} over $(\oo,\ord_N)$ is determined by the choice of a monic squarefree, strongly unitary $t_0\in A_0[y]$.

Consider $A_1=A_0[y]/(t_0)$, and let $z_0\in A_1$ be the class of $y$.

A \emph{representative} of $\ty$ is a monic $N$-robust $g\in\oo[x]$ such that $R_0(g)=t_0$.\medskip

An \emph{SF-type of order} $r>0$ over $(\oo,\ord_N)$ is a collection of data:
$$
\ty=\left(t_0;(g_1,\la_1,t_1);\dots;(g_r,\la_r,t_r)\right)
$$
distributed into levels $0,1,\dots,r$, such that
\begin{itemize}
\item $\ty'=\left(t_0;(g_1,\la_1,t_1);\dots;(g_{r-1},\la_{r-1},t_{r-1})\right)$ is an SF-type of order $r-1$
\item $g_r\in\oo[x]$ is a representative of $\ty'$
\item $\la_r$ is a positive rational number
\item $t_r\in A_r[y]$ is monic squarefree, strongly unitary, with $t_r(0)\ne0$.
\end{itemize}

By the very definition, from a type of order $r$ we may deduce types of order $0\le i\le r$,  by an adequate truncation:
$$
\ty_i:=\op{Trunc}_i(\ty):=\left(t_0;(g_1,\la_1,t_1);\dots;(g_i,\la_i,t_i)\right).
$$

For the definition of a type to be complete we must define representatives of types of positive order, and describe the artinian inductive algebra $A_{r+1}$ attached to $\ty$.
To this end, we discuss some data and operators that $\ty$ carries at the $r$-th level.
Note that $\ty$ supports analogous objects at the levels $0,1,\dots,r-1$, associated with the truncated types $\ty_0,\dots,\ty_{r-1}$.\medskip

\noindent{\bf Numerical data}

$e_0=m_0=1,\ h_0=\la_0=0,\ f_0=\deg t_0,\ V_0=0$.

$m_r=\deg g_r=e_{r-1}f_{r-1}m_{r-1}=(e_0\cdots e_{r-1})(f_0\cdots f_{r-1})$,

$\la_r=h_r/e_r$, with $h_r,e_r$ positive coprime integers,

$\ell_rh_r+\ell'_re_r=1$, B\'ezout identity determined by $0\le \ell_r<e_r$,

$f_r=\deg t_r$,

$V_r=v_{r-1}(g_r)=e_{r-1}f_{r-1}(e_{r-1}V_{r-1}+h_{r-1})$.\medskip

From the latter recurrence it is easy to deduce the following identity:
\begin{equation}\label{recurrence}
\dfrac{V_r}{e_1\cdots e_{r-1}}=\sum_{1\le j<r}\dfrac{m_r}{m_j}\,\dfrac{h_j}{e_1\cdots e_j}.
\end{equation}

\noindent{\bf Inductive artinian algebra}
$$
A_0\subset A_1\subset\cdots\subset A_r\subset A_{r+1}, \qquad A_{r+1}=A_r[y]/(t_r).
$$
Let $z_r\in A_{r+1}$ be the class of $y$ in $A_{r+1}$; thus, $A_{r+1}=A_r[z_r]=A_0[z_0,\dots,z_r]$.
By Lemma \ref{augmentation}, $A_{r+1}$ is a free $A_r$-algebra with basis $1,z_r,\dots,z_r^{f_r-1}$.

Since $t_r(0)$ is a unit in $A_r$, $z_r$ is a unit in $A_{r+1}$. \medskip

\noindent{\bf Newton polygon operator of order $r$}

Let $v_0:=\ord_N$ be the $N$-adic pseudo-valuation, and let $2^{\R^2}$ be the set of subsets of $\R^2$. For $r>0$ the type $\ty$ determines an operator:
$$
N_r:=N_{v_{r-1},g_r}\colon \oo[x]\to 2^{\R^2}.
$$

The Newton polygon of the zero polynomial is the empty set. For a non-zero $f\in\oo[x]$ we consider its canonical $g_r$-expansion:
\begin{equation}\label{gexpansion}
f=\sum\nolimits_{0\le s}a_sg_r^s,\quad a_s\in\oo[x], \ \deg a_s<m_r.
\end{equation}
For each $s\in\Z_{\ge0}$ we compute $u_s=v_{r-1}(a_sg_r^s)\in\Z_{\ge0}$, and we define $N_r(f)$ as the lower convex hull of the set of points $\left\{(s,u_s)\in\R^2\mid s\ge0,a_s\ne0\right\}$.

The Newton polygon $N_r(f)$ is the union of different adjacent \emph{sides}, whose endpoints are called \emph{vertices} of the polygon.
The typical shape of this polygon is shown in Figure \ref{figNmodel}.

The \emph{length} of $N_r(f)$ is by definition the abscissa of the last vertex. We denote it by  $\ell(N_r(f))=\lfloor \deg(f)/m_r\rfloor$.

\begin{figure}
\caption{Newton polygon of a polynomial $f\in \oo[x]$}\label{figNmodel}
\begin{center}
\setlength{\unitlength}{5mm}
\begin{picture}(21,9.5)
\put(8.8,2.8){$\bullet$}\put(7.8,3.8){$\bullet$}\put(6.8,2.8){$\bullet$}\put(4.8,4.8){$\bullet$}
\put(3.8,3.8){$\bullet$}\put(2.8,5.8){$\bullet$}\put(1.8,7.8){$\bullet$}
\put(-1,1){\line(1,0){12}}\put(0,0){\line(0,1){9}}
\put(7,3){\line(-3,1){3}}\put(4,4){\line(-1,2){2}}\put(7,3.03){\line(-3,1){3}}
\put(4,4.03){\line(-1,2){2}}\put(9,3){\line(-1,0){2}}\put(9,3.02){\line(-1,0){2}}
\put(9.85,4.8){$\bullet$}\put(9,3){\line(1,2){1}}\put(9,3.02){\line(1,2){1}}
\multiput(2,.9)(0,.25){29}{\vrule height2pt}
\multiput(7,.9)(0,.25){9}{\vrule height2pt}
\multiput(10,.9)(0,.25){16}{\vrule height2pt}
\put(5.5,.1){\begin{footnotesize}$\ell(N^-_r(f))$\end{footnotesize}}
\put(1.1,.15){\begin{footnotesize}$\ord_{g_r}(f)$\end{footnotesize}}
\put(9,.15){\begin{footnotesize}$\ell(N_r(f))$\end{footnotesize}}
\multiput(-.1,3)(.25,0){41}{\hbox to 2pt{\hrulefill }}
\put(5.6,7.6){\begin{footnotesize}$N_r(f)$\end{footnotesize}}
\put(-2.4,2.85){\begin{footnotesize}$v_{r-1}(f)$\end{footnotesize}}
\put(-.45,.35){\begin{footnotesize}$0$\end{footnotesize}}

\put(20.8,2.8){$\bullet$}\put(18.8,4.8){$\bullet$}\put(17.8,3.8){$\bullet$}\put(16.8,5.8){$\bullet$}\put(15.8,7.8){$\bullet$}
\put(13,1){\line(1,0){9}}\put(14,0){\line(0,1){9}}
\put(21,3){\line(-3,1){3}}\put(18,4){\line(-1,2){2}}\put(21,3.03){\line(-3,1){3}}
\put(18,4.03){\line(-1,2){2}}
\multiput(16,.9)(0,.25){29}{\vrule height2pt}
\multiput(21,.9)(0,.25){9}{\vrule height2pt}
\put(19.5,.1){\begin{footnotesize}$\ell(N^-_r(f))$\end{footnotesize}}
\put(15.1,.15){\begin{footnotesize}$\ord_{g_r}(f)$\end{footnotesize}}
\multiput(13.9,3)(.25,0){29}{\hbox to 2pt{\hrulefill }}
\put(18,7.6){\begin{footnotesize}$N^-_r(f)$\end{footnotesize}}
\put(11.6,2.85){\begin{footnotesize}$v_{r-1}(f)$\end{footnotesize}}
\put(13.55,.35){\begin{footnotesize}$0$\end{footnotesize}}
\end{picture}
\end{center}
\end{figure}
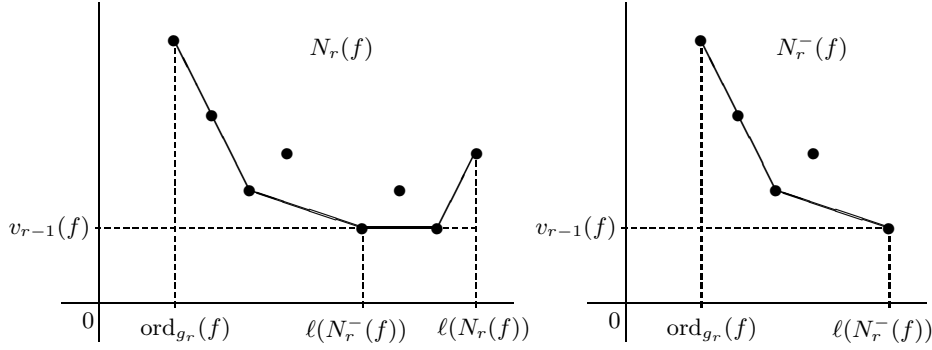

\begin{definition}\label{component}
The polygon $N_r^-(f)$ determined by the sides of negative slope of $N_r(f)$ is called the \emph{principal Newton polygon} of $f$.

If  $N_r(f)$ has no sides of negative slope, then $N_r^-(f)$ is the one-point set formed by the left endpoint of $N_r(f)$.

For any $\la\in\Q_{>0}$, the \emph{$\la$-component} of $N:=N_r^-(f)$ is the segment
$$S_\lambda(N)=\left\{(x,y)\in N\mid y+\la x\mbox{ is minimal\,}\right\}=N\cap L,
$$
where $L$ is the line of slope $-\la$ which first touches $N$ from below.

We denote by $(s_r(f),u_r(f))\in N_r^-(f)$ the coordinates of the left endpoint of the segment $S_r(f):=S_{\lambda_r}(N_r(f))$ (see Figure \ref{figRr}).
\end{definition}

If $N_r(f)$ has a side $S$ of slope $-\la_r$, then $S_r(f)=S$; otherwise, $S_r(f)=\left\{(s_r(f),u_r(f))\right\}$ is a vertex of $N_r(f)$.  \medskip

\noindent{\bf Pseudo-valuation of order $r$}

This is a pseudo-valuation $v_r\colon \oo[x]\to\Z\cup\{\infty\}$. For a non-zero $f\in\oo[x]$ having a $g_r$-expansion as in (\ref{gexpansion}), we define
\begin{equation}\label{posteriori}
v_r(f)=e_r\mn\left\{u_s+s\la_r\mid s\ge0\right\}=\mn\left\{v_r(a_sg_r^s)\mid s\ge0\right\}.
\end{equation}

Assuming that $v_{r-1}$ is a pseudo-valuation, it is easy to check that $v_r$ is a pseudo-valuation too.

Note that $v_r(f)/e_r$ can be reinterpreted as the ordinate where the line $L$ in Figure \ref{figRr} cuts the vertical axis.\medskip


\noindent{\bf Residual polynomial operator of order $r$}
$$
R_r:=R_{v_{r-1},g_r,\la_r}\colon \oo[x]\to A_r[y].
$$

The operator $R_r$ maps $0$ to $0$. For a non-zero $f\in \oo[x]$ with $g_r$-expansion as in (\ref{gexpansion}), we take the following residual coefficients $c_s\in A_r$, for $s\in \Z_{\ge0}$:
$$\as{1.2}
c_s=\left\{\begin{array}{ll}
0,&\mbox{if $(s,u_s)$ lies above }N_r^-(f), \mbox{ or }u_s=\infty,\\
z_{r-1}^{\nu_{r-1}(a_s)}R_{r-1}(a_s)(z_{r-1}),&\mbox{if $(s,u_s)$ lies on } N_r^-(f).
\end{array}
\right.
$$
For $r=1$, we agree that $z_0^{\nu_0(a)}=1$ for all $a\in\oo[x]$, even when $z_0=0$.

For $r>1$ and a nonzero $a \in\oo[x]$, we define:
\begin{equation}\label{nu}
\nu_{r-1}(a)=\ell'_{r-1}s_{r-1}(a)-\ell_{r-1}u_{r-1}(a)\in\Z,
\end{equation}
where $s_{r-1}(a)$, $u_{r-1}(a)$ were introduced in Definition \ref{component}.

Since $z_{r-1}$ is a unit in $A_r$, the element $z_{r-1}^{\nu_{r-1}(a_{s})}\in A_r$ is well defined regardless of the sign of the integer exponent.

Take $\la=h/e$, with $h,e$ positive coprime integers. Let $S$ be the $\la$-component of $N_r(f)$ (Definition \ref{component}), with endpoints having abscissas $s_0\le s'_0$. We define the residual polynomial attached to $\ty$ and $\la$ as:
$$
R_{v_{r-1},g_r,\la}(f)=c_{s_0}+c_{s_0+e}\,y+\cdots+c_{s_0+de}\,y^{d}\in A_r[y],
$$where  $d=(s'_0-s_0)/e$.
Note that $s_0,\,s_0+e,\dots,\,s'_0=s_0+de$ are the abscissas of the points in $S\cap\Z^2$.

Since $\deg a_s<m_r=e_{r-1}f_{r-1}m_{r-1}$, we have $\ell(N_{r-1}^-(a_s))<e_{r-1}f_{r-1}$, so that $\deg R_{r-1}(a_s)<f_{r-1}=\deg t_{r-1}$. Hence, the residual coefficient attached to a point lying on $N_r^-(f)$ is always nonzero. In particular, $c_{s_0}\ne0$ and $c_{s'_0}\ne0$. Thus, $\deg R_{v_{r-1},g_r,\la}(f)=d$.

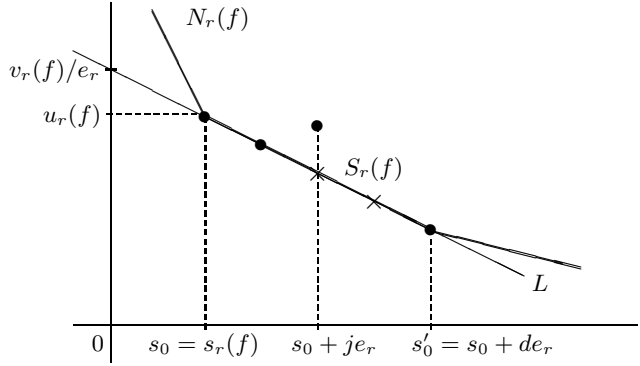
\begin{figure}
\caption{Residual polynomial $R_r(f)$. The line $L$ has slope $-\la_r$.}\label{figRr}
\begin{center}
\setlength{\unitlength}{5mm}
\begin{picture}(14,9)
\put(2.3,5.35){$\bullet$}\put(3.8,4.6){$\bullet$}\put(5.2,3.85){$\times$}\put(6.7,3.1){$\times$}\put(8.3,2.35){$\bullet$}\put(5.3,5.1){$\bullet$}
\put(-1,0){\line(1,0){15}}\put(0,-1){\line(0,1){9.6}}
\put(-1,7.3){\line(2,-1){12}}
\put(2.5,5.55){\line(-1,2){1.4}}\put(2.5,5.59){\line(-1,2){1.4}}
\put(2.5,5.55){\line(2,-1){6}}\put(2.5,5.59){\line(2,-1){6}}
\put(8.5,2.5){\line(4,-1){4}}\put(8.5,2.54){\line(4,-1){4}}
\multiput(2.5,-.1)(0,.25){23}{\vrule height2pt}
\multiput(8.5,-.1)(0,.25){11}{\vrule height2pt}
\multiput(5.5,-.1)(0,.25){22}{\vrule height2pt}
\multiput(-.1,5.6)(.25,0){10}{\hbox to 2pt{\hrulefill }}
\put(4.8,-.7){\begin{footnotesize}$s_0+je_r$\end{footnotesize}}
\put(8,-.7){\begin{footnotesize}$s'_0=s_0+de_r$\end{footnotesize}}
\put(1,-.7){\begin{footnotesize}$s_0=s_r(f)$\end{footnotesize}}
\put(11.2,.9){\begin{footnotesize}$L$\end{footnotesize}}
\put(6.2,4){\begin{footnotesize}$S_r(f)$\end{footnotesize}}
\put(2,8){\begin{footnotesize}$N_r(f)$\end{footnotesize}}
\put(-.5,-.7){\begin{footnotesize}$0$\end{footnotesize}}
\put(-.15,6.8){\line(1,0){.3}}
\put(-2.7,6.6){\begin{footnotesize}$v_r(f)/e_r$\end{footnotesize}}
\put(-1.8,5.4){\begin{footnotesize}$u_r(f)$\end{footnotesize}}
\end{picture}
\end{center}
\end{figure}

\begin{definition}\label{defrobust}
Let $\ty_0$ be a type of order zero.
A polynomial $f\in\oo[x]$ is said to be \emph{$\ty_0$-robust} if it is $N$-robust (Definition \ref{robust0}).

For our type $\ty$ of order $r>0$, we say that $f\in\oo[x]$ is \emph{$\ty$-robust} if
all coefficients $a_s\in\oo[x]$ of the canonical $g_r$-expansion of $f$ are $\ty_{r-1}$-robust and satisfy $\gcd\nolimits_{A_{r-1}}(R_{r-1}(a_{s}),t_{r-1})=1$.
\end{definition}

If $f$ is $\ty$-robust, then $R_{r-1}(a_s)(z_{r-1})$ is a unit in $A_r$
for all $s$, by Lemma \ref{gcd}. In particular, the polynomial $R_{v_{r-1},g_r,\la}(f)$ is strongly unitary for all slopes $-\la$ of $N^-_r(f)$.

Note that, for $r>0$, $\ty$-robustness does not depend on $\la_r$ nor $t_r$.

\begin{definition}\label{repr}
A \emph{representative} of $\ty$ is a $\ty$-robust monic $g\in\oo[x]$ of degree $m_{r+1}:=e_rf_rm_r$, such that $R_r(g)=t_r$.

We denote by $\rep(\ty)\subset \oo[x]$ the set of all representatives of $\ty$.
\end{definition}

The condition $R_r(g)=t_r$ implies that $N_r(g)$ contains a side of slope $-\la_r$ whose projection to the horizontal axis has length $e_rf_r$. Thus, $m_{r+1}$ is the minimal possible degree of a polynomial $g\in\oo[x]$ satisfying this condition.

Hence, if $g$ is a representative of $\ty$, the principal monomial of the $g_r$-expansion of $g$ is $g_r^{e_rf_r}$, and $N_r(g)$ must be one-sided of slope $-\la_r$ with endpoints $(0,e_rf_r(V_r+\la_r))$, $(f_r,e_rf_rV_r)$. In particular, $$v_r(g)=V_{r+1}:=e^2_rf_r(V_r+\la_r)=e_rf_r(e_rV_r+h_r).
$$

In order to enlarge $\ty$ to a type of order $r+1$, a crucial step is the construction of a representative. The construction for classical types (based on \cite[Prop. 3.4]{gen}) applies without changes to SF-types.

The condition of $g$ being $\ty$-robust is not guaranteed by this routine. In any case,  by introducing adequate hooks, we guarantee either a robust output or a non-trivial splitting of some modulus in the sequence $[t_{-1},t_0,\dots,t_r]$.

\subsection{Tree structure on the set of types}
Let us introduce a tree structure on the set $\tcal=\tcal(\oo,\ord_N)$ of all SF-types over $(\oo,\ord_N)$.

The root nodes are the types of order zero, and the previous node of a type $\ty\in\tcal$ of order $r>0$ is $\op{Trunc}_{r-1}(\ty)$. Thus, there is a unique path of length $r$ joining $\ty$ with its root node:

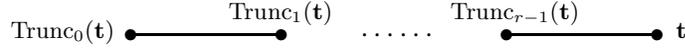
\setlength{\unitlength}{5mm}
\begin{picture}(16,1.5)
\put(4.8,0){$\bullet$}\put(8.8,0){$\bullet\qquad \ \cdots\cdots $}\put(14.8,0){$\bullet$}\put(18.8,0){$\bullet$}
\put(1.8,.1){\begin{footnotesize}$\op{Trunc}_{0}(\ty)$\end{footnotesize}}
\put(5,.2){\line(1,0){4}}\put(15,.2){\line(1,0){4}}
\put(7.6,0.6){\begin{footnotesize}$\op{Trunc}_{1}(\ty)$\end{footnotesize}}
\put(13.5,0.6){\begin{footnotesize}$\op{Trunc}_{r-1}(\ty)$\end{footnotesize}}
\put(19.5,0.1){\begin{footnotesize}$\ty$\end{footnotesize}}
\end{picture}\medskip

The connected components of $\tcal$ are the subtrees $\tcal_\varphi$ of all types $\ty$ with $\op{Trunc}_{0}(\ty)=(\varphi)$, for a monic squarefree strongly unitary $\varphi\in A_0[y]$.

The branches of a type $\ty$ of order $r$ are types of the form  $\ty^*=(\ty;(g,\la,t))$, obtained by enlarging $\ty$ with data $(g,\la,t)$ at the $(r+1)$-th level.

\subsection{Basic properties of types}\label{subsecBasic}
For the definition of the sum of plane segments see \cite[Sec. 1.1]{GMN}.

\begin{theorem}\label{basicprs}
 Let $\ty=\left(t_0;(g_1,\la_1,t_1);\dots;(g_r,\la_r,t_r)\right)$ be an SF-type of order $r\ge1$, and let $f,h\in\oo[x]$, $f\ne0$.\medskip

{\bf (A)} \ $f\sim_{v_r}h$ if and only if $S_r(f)=S_r(h)$ and $R_r(f)=R_r(h)$.\medskip

{\bf (B)} \ Suppose $v_r(f)=v_r(h)$, and let $H=f+h$. Then,
$$
y^{\nu_r(f)}R_r(f)+y^{\nu_r(h)}R_r(h)=
\begin{cases}
0,& \mbox{if }v_r(H)>v_r(f),\\
y^{\nu_r(H)}R_r(H),& \mbox{if }v_r(H)=v_r(f).
\end{cases}
$$

{\bf (C)} \ If $f$ is $\ty$-robust, then $v_r(fh)=v_r(f)+v_r(h)$ and
$$S_r(fh)=S_r(f)+S_r(h),\quad R_r(fh)=R_r(f)R_r(h).
$$

{\bf (D)} \ Let $g\in\oo[x]$ be a representative of $\ty$, and let $f=\sum_{0\le s}b_sg^s$ be the canonical $g$-expansion of $f$.
\begin{enumerate}
\item[(i)] $f\sim_{v_r}gh\ \imp\ R_r(f)\in t_rA_r[y]\ \imp\ \deg(f)\ge \deg(g)$.
\item[(ii)] $v_r(f)=\mn\{v_r(b_sg^s)\mid 0\le s\}$.
\end{enumerate}
\end{theorem}

\begin{proof}
We prove the four statements simultaneously by induction on the order $r$ of $\ty$. The analogous statements for types of order zero follow immediately from the definitions (cf. Lemmas \ref{v0equiv} and \ref{v0good}).

For each statement (X) of the theorem, we assume that the analogous statement (X$_{r-1}$) for the truncated type $\ty_{r-1}=\op{Trunc}_{r-1}( \ty)$ is true.

For any $q\in\oo[x]$ and any $s\in\Z_{\ge0}$, let $(s,u_s(q))$ be the point attached to the $s$-th term of the canonical $g_r$-expansion of $q$ and let $c_s(q)\in A_r$ be the corresponding residual coefficient determined by $N_r(q)$.

We denote $g=g_r$, $u_s=u_s(f)$, $u'_s=u_s(h)$ for simplicity. Let $$f=\sum_{0\le s}a_sg^s, \quad h=\sum_{0\le s}a'_sg^s, \quad f-h=\sum_{0\le s}(a_s-a'_s)g^s$$ be the canonical $g$-expansions of $f$, $h$ and $f-h$.

Since $g$ is $\ty_{r-1}$-robust, (C$_{r-1}$) shows that
\begin{equation}\label{robr-1}
u_s=v_{r-1}\left(a_sg^s\right)=v_{r-1}(a_s)+sv_{r-1}(g)=v_{r-1}(a_s)+sV_r.
\end{equation}

By (\ref{posteriori}) and the definition of $v_r$, we get
\begin{equation}\label{defs}
v_r(f)\le v_r(a_sg^s)=e_r(u_s+s\la_r),\ \,\forall s\ge0,
\end{equation}
and equality holds if and only if $(s,u_s)\in S_r(f)$. \medskip

Let us prove (A).
Suppose $f\sim_{v_r}h$; that is, $v_r(f-h)>v_r(f)$. Take a point $(s,u_s)\in S_r(f)$ for which equality holds in (\ref{defs}). By (\ref{defs}) applied to $h$ and $f-h$, we get $v_r(h)\le v_r(a'_sg^s)=e_r(u'_s+s\la_r)$, and
\begin{equation}\label{hyp}
v_r\left(a_sg^s\right)=v_r(f)<v_r(f-h)\le v_r\left((a_s-a'_s)g^s\right).
\end{equation}
This proves $v_r\left(a_sg^s\right)=v_r\left(a'_sg^s\right)$, leading to $u_s=u'_s$. Thus, $S_r(f)\subset S_r(h)$ and the symmetry of the argument implies  $S_r(f)=S_r(h)$.

Also, (\ref{robr-1}), (\ref{defs}) and (\ref{hyp}) imply $a_s\sim_{v_{r-1}}a'_s$. By (A$_{r-1}$), we deduce that $S_{r-1}(a_s)=S_{r-1}(a'_s)$ (if $r>1$) and $R_{r-1}(a_s)=R_{r-1}(a'_s)$. This implies in any case  $c_s(f)=c_s(h)$. Since this works for any abscissa $s$ for which $(s,u_s)$ lies on $S_r(f)=S_r(h)$, the polynomials $R_r(f)$ and $R_r(h)$ coincide.

The reciprocal implication follows from (B). In fact, $S_r(f)=S_r(h)$ implies $v_r(f)=v_r(h)$ and $\nu_r(f)=\nu_r(h)=\nu_r(-h)$. If moreover $R_r(f)=R_r(h)$, then  $R_r(f)+R_r(-h)=0$, and this implies $v_r(f-h)>v_r(f)$ by (B). Therefore, the proof of (B) will complete the proof of (A) as well.\medskip

Let us prove (B). If $v_r(f+h)>v_r(f)$, then $f\sim_{v_r}-h$ and we have seen above that this implies $\nu_r(f)=\nu_r(-h)=\nu_r(h)$ and $R_r(f)=R_r(-h)=-R_r(h)$. Hence, $R_r(f)+R_r(h)=0$, as predicted by (B).

Suppose  $v_r(f)=v_r(h)=v_r(f+h)$, and let $L$ be the line of slope $-\la_r$ cutting the vertical axis at the point $(0,v_r(f)/e_r)$. As Figure \ref{figRr} shows, this line $L$ contains the segments $S_r(f)$, $S_r(h)$ and $S_r(f+h)$.

Suppose $s_0:=s_r(f)\le s_r(h)$. We then have $s_0\le s_r(f+h)$ too. The points in $L\cap \Z^2$ have abscissa $s_j:=s_0+je_r$, for some $j\in\Z$. Let $q=f,h$, or $f+h$. The left endpoint of $S_r(q)$ lies on $L$, so that
$$s_r(q)=s_0+m_qe_r,\quad u_r(q)=u_r(f)-m_qh_r,
$$for some non-negative integer $m_q$. The formula (\ref{nu}) for the function $\nu_r$ yields $\nu_r(q)=\nu_r(f)+m_q$.

Consider the following residual coefficients:
$$
\tilde{c}_s(q)=
\begin{cases}
c_s(q),&\mbox{ if $(s,u_s(q))$ lies on }L,\\
0,&\mbox{ if $(s,u_s(q))$ lies above }L.
\end{cases}
$$
These residual coefficients satisfy:
$$
y^{\nu_r(q)}R_r(q)=y^{\nu_r(f)}y^{m_q}R_r(q)=y^{\nu_r(f)}\sum\nolimits_{0\le j}\tilde{c}_{s_j}(q)y^j.
$$
Hence, the identity predicted by (B) in this case amounts to:
\begin{equation}\label{tilde}
\tilde{c}_s(f)+\tilde{c}_s(h)=\tilde{c}_s(f+h),\quad \forall\,s\ge0.
\end{equation}

Suppose $\tilde{c}_s(f)=0$, so that $(s,u_s)$ lies above $L$. If $(s,u'_s)$ lies above $L$, then  $(s,u_s(f+h))$ lies above $L$ too and $\tilde{c}_s(h)=\tilde{c}_s(f+h)=0$.
If $(s,u'_s)$ lies on $L$, then $a_s+a'_s\sim_{v_{r-1}}a'_s$ and $\tilde{c}_s(h)=c_s(h)=c_s(f+h)=\tilde{c}_s(f+h)$, by (A$_{r-1}$).  The same arguments
prove (\ref{tilde}) when $\tilde{c}_s(h)=0$.

Suppose $\tilde{c}_s(f)\ne0$ and $\tilde{c}_s(h)\ne0$, so that $(s,u_s)=(s,u'_s)$ lies on $L$. By (\ref{robr-1}), we have $v_{r-1}(a_s)=v_{r-1}(a'_s)$.
By (B$_{r-1}$),
\begin{align*}
y^{\nu_{r-1}(a_s)}R_{r-1}(a_s)&+y^{\nu_{r-1}(a'_s)}R_{r-1}(a'_s)\\=&
\begin{cases}
0,& \mbox{if }v_{r-1}(a_s+a'_s)>v_{r-1}(a_s),\\
y^{\nu_{r-1}(a_s+a'_s)}R_r(a_s+a'_s),& \mbox{if }v_{r-1}(a_s+a'_s)=v_{r-1}(a_s).
\end{cases}
\end{align*}
Since $z_0^{\nu_0(a)}=1$ for all $a\in\oo[x]$, and $z_{r-1}$ is a unit in $A_r$ for $r>1$,  we may replace $y=z_{r-1}$ to derive the identity (\ref{tilde}). This ends the proof of (B).\medskip

Let us prove (C). Assume that $f$ is $\ty$-robust (Definition \ref{defrobust}).

Let $s_0=s_r(f)$,  $\ell_0=s_r(h)$ be the abscissas of the left endpoints of $S_r(f)$, $S_r(h)$, respectively. Let $d=\deg R_r(f)$, $d'=\deg R_r(h)$, and denote
$$
s_j=s_0+je_r,\ 0\le j\le d;\qquad \ell_i=\ell_0+ie_r,\ 0\le i\le d'.
$$
For an abscissa $s\not\in\{s_j\mid 0\le j\le d\}$ the point $(s,u_s)$ does not lie on $S_r(f)$ and (\ref{defs}) shows that $v_r(a_sg^s)>v_r(f)$. Analogously,  $v_r(a'_sg^s)>v_r(h)$ for the abscissas  $s\not\in\{\ell_i\mid 0\le i\le d'\}$.
Therefore, the polynomials $f_0=\sum_{j=0}^da_{s_j}g^{s_j}$,  $h_0=\sum_{i=0}^{d'}a'_{\ell_i}g^{\ell_i}$ satisfy:
\begin{equation}\label{f0h0}
f\sim_{v_r} f_0,\quad h\sim_{v_r} h_0,\quad v_r(fh-f_0h_0)>v_r(f)+v_r(h).
\end{equation}

\noindent{\bf Claim 1. }$v_r(f_0h_0)=v_r(f)+v_r(h)$.

Denote $\sg_k=s_0+\ell_0+ke_r,\ 0\le k\le d+d'$. Consider the product
$$
f_0h_0=\sum\nolimits_{0\le k\le d+d'}b_kg^{\sg_k},\qquad b_k=\sum\nolimits_{i+j=k}a_{s_j}a'_{\ell_i},
$$
and let the canonical $g$-expansion of each $b_k$ be:
$$
b_k=b_{k,0}+b_{k,1}g,\quad  0\le k\le d+d'.
$$
The lowest degree term of the canonical $g$-expansion of $f_0h_0$ is $b_{0,0}\,g^{s_0+\ell_0}$. By (\ref{f0h0}), $v_r(f_0h_0)\ge v_r(f)+v_r(h)$. By (\ref{posteriori}), equality holds if we show that
\begin{equation}\label{claim}
v_r\left(b_{0,0}g^{s_0+\ell_0}\right)=v_r(f)+v_r(h).
\end{equation}

By (\ref{robr-1}) and (\ref{defs}), this amounts to
\begin{equation}\label{toknow}
v_{r-1}(b_{0,0})=v_{r-1}(a_{s_0})+v_{r-1}(a'_{\ell_0}),
\end{equation}
which is equal to $v_{r-1}(a_{s_0}a'_{\ell_0})$ by (C$_{r-1}$), because $a_{s_0}$ is $\ty_{r-1}$-robust.
Since
$$a_{s_0}a'_{\ell_0}=b_0=b_{0,0}+b_{0,1}g$$ (D$_{r-1}$) shows that $v_{r-1}(b_{0,0})\ge v_{r-1}(a_{s_0}a'_{\ell_0})$. This inequality cannot be strict because  $a_{s_0}a'_{\ell_0}\sim_{v_{r-1}}b_{0,1}g$ implies, by (D$_{r-1}$):
$$
R_{r-1}(a_{s_0})R_{r-1}(a'_{\ell_0})=R_{r-1}(a_{s_0}a'_{\ell_0})\in t_{r-1}A_{r-1}[y].
$$
Since $f$ is $\ty$-robust, $\gcd_{A_{r-1}}(R_{r-1}(a_{s_0}),t_{r-1})=1$. Hence, $R_{r-1}(a'_{\ell_0})$ belongs to $t_{r-1}A_{r-1}[y]$, contradicting (D$_{r-1}$). This ends  the proof of Claim 1.\medskip

Claim 1 and (\ref{f0h0}), imply $fh\sim_{v_r} f_0h_0$ and $v_r(fh)=v_r(f)+v_r(h)$.

Therefore, by using (A), we may assume $f=f_0$, $h=h_0$ for the proof of the two last statements of (C) concerning the operators $S_r$ and $R_r$.

The segment $S_r(f)+S_r(h)$ has as left (resp. right) endpoint the vector sum of the two left (resp. right) endpoints of $S_r(f)$ and $S_r(h)$.

By (\ref{claim}) and (\ref{toknow}), $(s_0+\ell_0,u_{s_0}+u'_{\ell_0})$ is the left endpoint of $S_r(fh)$. Hence, in order to show that  $S_r(fh)=S_r(f)+S_r(h)$ we need only to check that $(s_d+\ell_{d'},u_{s_d}+u'_{\ell_{d'}})$ is the right endpoint of $S_r(fh)$.

For any term $ag^s$ of the canonical $g$-expansion of $fh$, (\ref{robr-1}) and (\ref{defs}) yield:
$$
v_{r-1}(a)\ge y_s:=(v_r(fh)/e_r)-s(V_r+\la_r),
$$
and equality holds if and only if $(s,v_{r-1}(ag^s))\in S_r(fh)$.

For $0 \le k\le d+d'$, the summand $b_k g^{\sg_k}$ contributes to the canonical $g$-expansion of $fh$ with two terms: $b_{k,0}g^{\sg_k}$ and  $b_{k,1}g^{\sg_k+1}$. By (D$_{r-1}$),
\begin{align*}
v_{r-1}(b_{k,1}g)\ge &\,v_{r-1}(b_k)\ge v_{r-1}(a_{s_j}a'_{\ell_i})=v_{r-1}(a_{s_j})+v_{r-1}(a'_{\ell_i})\\
\ge &\,(v_r(f)/e_r)-s_j(V_r+\la_r)+(v_r(h)/e_r)-\ell_i(V_r+\la_r) =y_{\sg_k},
\end{align*}
for some pair $(i,j)$ with $i+j=k$.
By (\ref{robr-1}), $v_{r-1}(b_{k,1})\ge y_{\sg_k}-V_r> y_{\sg_k+1}$. Thus, the contribution of the terms $b_{k,1}g^{\sg_k+1}$ is irrelevant in order to detect points $(s,v_{r-1}(ag^s))$ lying on $S_r(fh)$. More precisely, for any $s\in\Z_{\ge0}$,
\begin{equation}\label{bk0}
(s,v_{r-1}(ag^s))\in S_r(fh) \sii s=\sg_k \mbox{ and }v_{r-1}(b_{k,0})=y_s,
\end{equation}
for some $0\le k\le d+d'$.

Analogous arguments as those used in the proof of Claim 1 show that
$$v_{r-1}(b_{d+d',0})=v_{r-1}(a_{s_d}a'_{\ell_{d'}})=v_{r-1}(a_{s_d})+v_{r-1}(a'_{\ell_{d'}})=y_{\sg_{d+d'}}.
$$
Hence, (\ref{bk0}) shows that $(s_d+\ell_{d'},u_{s_d}+u'_{\ell_{d'}})$ is the right endpoint of $S_r(fh)$. This proves $S_r(fh)=S_r(f)+S_r(h)$ and, moreover:
\begin{equation}\label{ckfh}
c_{\sg_k}(fh)=c_{\sg_k}(b_{k,0}g^{\sg_k}),\quad 0\le k\le d+d'.
\end{equation}

\noindent{\bf Claim 2. }For each $0\le k\le d+d'$ we have
$$c_{\sg_k}(fh)=
\begin{cases}
0,&\mbox{if }v_{r-1}(b_k)>y_{\sg_k},\\
z_{r-1}^{\nu_{r-1}(b_k)}R_{r-1}(b_k)(z_{r-1}),&\mbox{if }v_{r-1}(b_k)=y_{\sg_k}.
\end{cases}
$$

By (D$_{r-1}$), $v_{r-1}(b_{k,0})\ge v_{r-1}(b_k)\ge y_{\sg_k}$.
If either $v_{r-1}(b_k)>y_{\sg_k}$ or $v_{r-1}(b_{k,0})>v_{r-1}(b_k)$, then  $c_{\sg_k}(b_{k,0}g^{\sg_k})=0$. In the latter case,
 $b_k\sim_{v_{r-1}}b_{k,1}g$ and (D$_{r-1}$) shows that $R_{r-1}(b_k)\in t_{r-1}A_{r-1}[y]$, so that $R_{r-1}(b_k)(z_{r-1})=0$. Thus, Claim 2 is correct in both cases, by (\ref{ckfh}).

Finally, suppose $v_{r-1}(b_k)=v_{r-1}(b_{k,0})=y_{\sg_k}$. By (B$_{r-1}$),
$$
y^{\nu_{r-1}(b_k)}R_{r-1}(b_k)-y^{\nu_{r-1}(b_{k,0})}R_{r-1}(b_{k,0})
$$
is either zero or equal to $y^{\nu_{r-1}(b_{k,1}g)}R_{r-1}(b_{k,1}g)$. In both cases, the polynomial vanishes when we replace $y=z_{r-1}$. Thus,
$$
c_{\sg_k}(b_{k,0}g^{\sg_k})=z_{r-1}^{\nu_{r-1}(b_{k,0})}R_{r-1}(b_{k,0})(z_{r-1})=z_{r-1}^{\nu_{r-1}(b_k)}R_{r-1}(b_k)(z_{r-1}).
$$
By (\ref{ckfh}), this ends the proof of Claim 2.\medskip

We are ready to prove $R_r(fh)=R_r(f)R_r(h)$, which amounts to
\begin{equation}\label{aimR}
\sum\nolimits_{i+j=k}c_{s_j}(f)c_{\ell_i}(h)=c_{\sg_k}(fh),\quad 0\le k\le d+d'.
\end{equation}

For any pair $(i,j)$ with $i+j=k$, we have
$v_{r-1}(a_{s_j}a'_{\ell_i})\ge y_{\sg_k}$, and equality holds if and only if $(s_j,u_{s_j})\in S_r(f)$ and
 $(\ell_i,u'_{\ell_i})\in S_r(h)$; or equivalently, $c_{s_j}(f)\ne0$ and $c_{\ell_i}(h)\ne0$. Hence, in the left-hand side of the equality (\ref{aimR}) we need only to consider pairs $(i,j)$ in the set:
$$I=\{(i,j)\mid i+j=k \mbox{ and }v_{r-1}(a_{s_j}a'_{\ell_i})=y_{\sg_k}\}.$$

If $I$ is the empty set, then both sides of (\ref{aimR}) vanish. The right-hand side vanishes by Claim 2, because $v_{r-1}(b_k)>y_{\sg_k}$.

Suppose $I\ne\emptyset$ and let $B_k=\sum_{(i,j)\in I}a_{s_j}a'_{\ell_i}$. Clearly,
$v_{r-1}(B_k)\ge y_{\sg_k}$ and
$$v_{r-1}(B_k)>y_{\sg_k} \sii v_{r-1}(b_k)>y_{\sg_k}.
$$
If $v_{r-1}(B_k)=v_{r-1}(b_k)=y_{\sg_k}$, then $B_k\sim_{v_{r-1}}b_k$, and by (C$_{r-1}$):
$$S_{r-1}(B_k)=S_{r-1}(b_k),\quad
R_{r-1}(B_k)=R_{r-1}(b_k).
$$
From the first equality we deduce $\nu_{r-1}(B_k)=\nu_{r-1}(b_k)$. Thus, by (B$_{r-1}$):
$$
\sum_{(i,j)\in I}y^{\nu_{r-1}(a_{s_j}a'_{\ell_i})}R_{r-1}(a_{s_j}a'_{\ell_i})=
\begin{cases}
0,& \mbox{if }v_{r-1}(b_k)>y_{\sg_k},\\
y^{\nu_{r-1}(b_k)}R_{r-1}(b_k),& \mbox{if }v_{r-1}(b_k)=y_{\sg_k}.
\end{cases}
$$

By taking $y=z_{r-1}$, we get a similar identity in $A_r$.
The right-hand side of this identity is equal to $c_{\sg_k}(fh)$ by Claim 2. Thus, we need only to show that the left-hand side is equal to $\sum_{(i,j)\in I}c_{s_j}(f)c_{\ell_i}(h)$.

Since $a_{s_j}$ is $\ty_{r-1}$-robust, (C$_{r-1}$) shows that
$$S_{r-1}(a_{s_j}a'_{\ell_i})=S_{r-1}(a_{s_j})+S_{r-1}(a'_{\ell_i}),\quad
R_{r-1}(a_{s_j}a'_{\ell_i})=R_{r-1}(a_{s_j})R_{r-1}(a'_{\ell_i}).
$$
Since the function $\nu_{r-1}(a)$ depends only on the left endpoint of $S_{r-1}(a)$, from the first equality we deduce $\nu_{r-1}(a_{s_j}a'_{\ell_i})=\nu_{r-1}(a_{s_j})+\nu_{r-1}(a'_{\ell_i})$. Hence,
$$
z_{r-1}^{\nu_{r-1}(a_{s_j}a'_{\ell_i})}R_{r-1}(a_{s_j}a'_{\ell_i})(z_{r-1})=c_{s_j}(f)c_{\ell_i}(h).
$$
This ends the proof of (C).\medskip

Let us prove (D). By (A) and (C), from $f\sim_{v_r}gh$ we deduce:
$$
R_r(f)=R_r(gh)=R_r(g)R_r(h)=t_rR_r(h)\in t_rA_r[y].
$$
This implies $f_r=\deg(t_r)\le \deg R_r(f)\le \deg(f)/m_re_r$, the last inequality by the definition of the operator $R_r$. This proves (i).

Write $f=b_0+gq$. By (i), $b_0\not \sim_{v_r}-gq$, so that
$$
v_r(f)=\mn\{v_r(b_0),v_r(gq)\}.
$$
By (C), $v_r(gq)=v_r(g)+v_r(q)$. A recursive argument proves (ii).
\end{proof}

\begin{corollary}\label{product}
Let $\ty$ be a type of order $r\ge 1$. Let $f,h\in\oo[x]$ with $f$ $\ty$-robust. Then, $N_r^-(fh)=N_r^-(f)+N_r^-(h)$.
\end{corollary}

\begin{proof}
The polygon $N^-_r(f)+N^-_r(h)$ is uniquely determined by the property $$S_\la(N^-_r(f)+N^-_r(h))=S_\la(N^-_r(f))+S_\la(N^-_r(h)),$$ for $-\la$ running on the slopes of $N^-_r(f)$ and $N^-_r(h)$ \cite[Sec. 1.1]{GMN}. This is satisfied by $N_r^-(fh)$, by (C) of Theorem \ref{basicprs} applied to the type obtained by replacing $\la_r$ with $\la$ in the data of the last level.
\end{proof}

\begin{definition}\label{ordt}
Let $\ty$ be a type of order $r\ge0$. For any $f\in\oo[x]$ we define $\ord_\ty(f)=\ord_{t_r}(R_r(f))$ in $A_r[y]$ (see section \ref{subsecInductive}).

If $\ord_\ty(f)>0$ we say that $\ty$ \emph{divides} $f$, and we write $\ty\mid f$.
\end{definition}

\begin{corollary}\label{length}
Let $\ty$ be a type of order $r\ge1$ and let $\ty_{r-1}=\op{Trunc}_{r-1}(\ty)$.
\begin{enumerate}
\item For any $f\in\oo[x]$ we have $\ell(N_r^-(f))=\ord_{\ty_{r-1}}(f)$.
\item If $f\in\oo[x]$ is $\ty$-robust, then $$\ord_{\ty_{r-1}}(fh)=\ord_{\ty_{r-1}}(f)+\ord_{\ty_{r-1}}(h),\quad \forall h\in\oo[x].$$
\end{enumerate}
\end{corollary}

\begin{proof}
Denote $g=g_r$ and let $f=\sum_{0\le s}a_sg^s$ be the canonical $g$-expansion of $f$. Let $I=\{s\in\Z_{\ge0}\mid v_{r-1}(a_sg^s)=v_{r-1}(f)\}$, and take
$$s_0=\mn(I),\quad f_0=\sum\nolimits_{s\in I}a_sg^s=g^{s_0}h.
$$

By (D) of Theorem \ref{basicprs}, $I\ne\emptyset$, $s_0=\ell(N_r^-(f))$, and $f\sim_{v_{r-1}}f_0$.

Since $g$ is $\ty_{r-1}$-robust, (A) and (C) of Theorem \ref{basicprs} show that
$$
R_{r-1}(f)=R_{r-1}(f_0)=R_{r-1}(g)^{s_0}R_{r-1}(h)=t_{r-1}^{s_0}R_{r-1}(h).
$$
Hence, the corollary will be proven if we show that $t_{r-1}\nmid R_{r-1}(h)$.

To this end, write $h=a_{s_0}+gq$. By (D) and (C) of Theorem \ref{basicprs}:
$$v_{r-1}(h)=v_{r-1}(a_{s_0})=v_{r-1}(gq).
$$
Hence, (B) of Theorem \ref{basicprs} shows that:
\begin{equation}\label{B}
y^{\nu_{r-1}(a_{s_0})}R_{r-1}(a_{s_0})+
y^{\nu_{r-1}(gq)}R_{r-1}(gq)=y^{\nu_{r-1}(h)}R_{r-1}(h).
\end{equation}
Now, $R_{r-1}(gq)=R_{r-1}(g)R_{r-1}(q)=t_{r-1}R_{r-1}(q)$. Suppose that $t_{r-1}$ divides $R_{r-1}(h)$. Since $z_0^{\nu_0(a)}=1$ for all $a\in\oo[x]$, and $z_{r-1}$ is a unit in $A_r$ for $r>1$, we may replace $y=z_{r-1}$ in (\ref{B}), leading to
$R_{r-1}(a_{s_0})(z_{r-1})=0$. This is impossible, because $\deg R_{r-1}(a_{s_0})<f_{r-1}=\deg t_{r-1}$. This proves (1).

Item (2) follows immediately from (1) and Corollary \ref{product}.
\end{proof}

\section{The SF-OM algorithm}\label{secOM}
Given $N$ and a monic irreducible $f\in\Z[x]$, the SF-OM algorithm either finds a proper factor of $N$, or constructs an \emph{SF-OM representation of $f$ with respect to $N$}.

That is, a finite tree $\tcal(f)$ of SF-types with the following properties:

\begin{enumerate}
\item The root nodes are the types of order zero attached to the squarefree factors of $\rd_N(f)$ in $A_0[y]$.
\item The leaves $\ty_1,\dots,\ty_m$ satisfy $\ord_{\ty_i}(f)=1$ and $\rep(\ty_i)\cap\rep(\ty_j)=\emptyset$ for all $1\le i\ne j\le m$.
\end{enumerate}

Initially, we compute a squarefree decomposition of $\rd_N(f)$ in $A_0[y]$:
$$\rd_N(f)=T_1^{\ell_1}\cdots T_k^{\ell_k},\qquad 1\le \ell_1<\cdots<\ell_k.$$ For each squarefree factor $T$ the type of order zero $\ty_0=(T)$ is taken as one of the root nodes of the tree. Let $\ell:=\ord_{\ty_0}(f)=\ord_T(\rd_N(f)$.

If $\ell=1$, then $\ty_0$ is a leaf too.
If $\ell>1$, then $\ty_0$ sprouts several branches $\ty_{\la,t}=(T;(g,\la,t))$, where $g$ is a representative of $\ty_0$, $-\la$ runs on the slopes of $N_{v_0,g}^-(f)$ and $t$ runs on the squarefree factors of $R_{v_0,g,\la}(f)\in A_1[y]$, provided by the SFD algorithm.

All types $\ty$ obtained along the process branch in a similar way, as long as $\ord_\ty(f)>1$. Let us briefly review the relevant subroutines which are used.\medskip

\noindent{\tt SFD($A$,\,$\varphi$)}

\noindent Algorithm \ref{sfdA}  applied to $\varphi\in A[y]$, plus a test to check if the squarefree factors are strongly unitary. \medskip

\noindent{\tt Newton($\ty$,\,$\ell$,\,$f$)}

\noindent The type $\ty$ of order $i$ is equipped with a representative $g$. The routine computes the first $\ell+1$ coefficients $a_0,\dots,a_\ell$ of the canonical $g$-expansion $f=\sum_{0\le s}a_sg^s$, and the Newton polygon of the cloud of points $(s,v_i(a_sg^s))$ for $0\le s\le\ell$. Then, it tests if $\gcd_{A_i}(R_i(a_s),t_i)=1$, and certificates that the input polynomial $f$ is $\ty$-robust.\medskip


\noindent{\tt ResidualPolynomial($\ty$,\,$\la$,\,$f$)}

\noindent The type $\ty$ of order $i$ is equipped with a representative $g$. The routine computes the residual polynomial $R_{v_i,g,\la}(f)\in A_{i+1}[y]$.\medskip

\noindent{\tt Representative($\ty$)}

\noindent Computation of a representative of $\ty$ based on \cite[Prop. 3.4]{gen}.\medskip

We now describe the SF-OM algorithm in pseudocode.
We emphasize the type to which a certain level data belongs as a superindex: $g_i^\ty$, $\la_i^\ty$, $t_i^\ty$, etc.
\bigskip

\noindent{\bf SF-OM ALGORITHM}\vskip 1mm

\noindent INPUT:

$-$ A monic irreducible polynomial $f\in \Z[x]$ of degree $n>1$.

$-$ An integer $N>1$ whose prime factors $p$ satisfy $p>n$.

\medskip

\sst{1}Call {\tt SFD($A_0$,$\rd_N(f)$)}, with output $L=[(T_1,\ell_1),\dots,(T_k,\ell_k)]$

\sst{2}If $\ell_1=1$ THEN build the type $\ty=(T_1)$, set {\tt Leaves} $\leftarrow [\ty]$ and delete

\sst{}$(T_1,\ell_1)$ from the list $L$ \quad ELSE\quad
{\tt Leaves} $\leftarrow [\ ]$

\sst{3}FOR each $(T,\ell)$ in $L$ DO

\stst{4}Take a robust monic lifting $g\in \Z[x]$ of $T$ and create a type $\ty$ with
\vskip -.4mm \stst{}$t_0^\ty\leftarrow T, \quad \om_1^\ty\leftarrow\ell,\quad g_1^\ty\leftarrow g,\quad A_1^\ty\leftarrow A_0[y]/(T)$

\stst{}Initialize a stack {\tt BranchNodes} $\leftarrow[\ty]$

\stst{}{\bf WHILE $\#${\tt BranchNodes}\;$>0$  DO}

\ststst{5}Extract a type $\ty_0$ from {\tt BranchNodes}. Let $i-1$ be its order

\ststst{6}FOR every slope $-\la$ of {\tt Newton($\ty_0$,\,$\om_i^{\ty_0}$,\,$f$)} DO

\stststst{7}$\lambda_i^{\ty_0}\leftarrow$ $\la$, \ $R_i(f)\leftarrow$ {\tt ResidualPolynomial($\ty_0$,\,$\la$,\,$f$)}

\stststst{8}Call {\tt SFD($A_i^{\ty_0}$,\,$R_i(f)$)}, with output $L'=[(T'_1,\ell'_1),\dots,(T'_j,\ell'_j)]$

\stststst{9}FOR each $(T',\ell')$ in $L'$ DO

\ststststst{10}$\ty\leftarrow\ty_0\quad t_i^{\ty}\leftarrow T',\quad A_{i+1}^\ty\leftarrow A_i^\ty[y]/(T'),\quad \om_{i+1}^\ty\leftarrow\ell'$

\ststststst{11}IF $\ell'=1$ THEN add $\ty$ to {\tt Leaves} and go to step {\bf 5}

\ststststst{12}$g_{i+1}^\ty\leftarrow$ {\tt Representative($\ty$)} and add
 $\ty$ to {\tt BranchNodes}




\stst{}{\bf END WHILE}\medskip

\noindent OUTPUT:

If no proper factor of $N$ is detected along the process, the types in {\tt Leaves} are the leaves of an SF-OM representation of $f$ with respect to $N$. \bigskip

This description of the SF-OM algorithm omits implementation details which are not relevant for the purposes of this paper, like an acceleration of the routine based on the use of \emph{optimal SF-types}.

Also, we omitted all hooks.
All subroutines crash if a proper factor of $N$, or of some modulus $t_i\in A_i[y]$, is detected. In the former case, the routine outputs the factor of $N$ and ends. In the latter case, the algorithm modifies some accumulated data and continues.

More precisely, if a factorization $t_i=\varphi\psi$ in $A_i[y]$ is detected, then all types $\ty$ in the stack {\tt BranchNodes} whose truncation at the $i$-th level is
$$
\op{Trunc}_i(\ty)=\left(t_0;(g_1,\la_1,t_1);\dots;(g_i,\la_i,t_i)\right)
$$
are replaced with two types of order $i$:
$$
\ty'=\left(t_0;(g_1,\la_1,t_1);\dots;(g_i,\la_i,\varphi)\right),\quad
\ty''=\left(t_0;(g_1,\la_1,t_1);\dots;(g_i,\la_i,\psi)\right).
$$

\section{Tree of irreducible types associated with an SF-type}

We fix an SF-type $\ty=\left(t_0;(g_1,\la_1,t_1);\dots;(g_r,\la_r,t_r)\right)$.
Let $A$ be the inductive artinian algebra attached to $\ty$ in section \ref{subsecNtypes}.

\begin{definition}\label{unramified}
We say that $\ty$ is \emph{unramified} if $e_0=\cdots=e_r=1$.
We say that $\ty$ is \emph{irreducible} if $N$ is a prime and all $t_i\in A_i[y]$ are irreducible.
\end{definition}

Irreducible types coincide with the types introduced by Montes \cite{m,GMN}, except for a different normalization of the data and operators they support, which has been taken from \cite{gen}.

We fix a prime factor $p$ of $N$, and denote by $\mx_p(A)$ the set of maximal ideals containing $p$. For instance,  $\mx_p(A_0)=\{\m_0\}$, with $\m_0=pA_0$.


\subsection{Tree of maximal ideals of an inductive artinian algebra}\label{subsecTreeM}

Consider the formal disjoint union of all the sets $\mx_p(A_i)$:
$$
\mm_p=\coprod\nolimits_{i=0}^{r+1}\mx_p(A_i)
$$
Note that $\mx_p(A_i)$ and $\mx_p(A_{i+1})$ are disjoint subsets of  $\mm_p$ even if
$A_i=A_{i+1}$, which may occur for several indices $0\le i\le r$.

The set $\mm_p$ may be given the structure of a finite connected tree with root node $\m_0$, by defining the previous node of any $\n\in\mx_p(A_i)$, with $i>0$, as the prime ideal $\n\cap A_{i-1}$. The leaves are the elements in $\mpa$.

The path joining any $\m\in\mpa$ with the root node is:
$$
\m_0\mapsto\m_1\mapsto\cdots\mapsto\m_r\mapsto\m_{r+1}=\m,\qquad \m_i:=\m\cap A_i,\  0\le i\le r+1.
$$

The finite field $\F_\m=A/\m$ is an inductive artinian algebra with intermediate fields  $\F_{\m_i}=A_i/\m_i$, for $0\le i\le r+1$. The canonical mapping $\rdm\colon A\to \F_\m$
respects the inductive structures:
\begin{equation}\label{chainF}
\begin{array}{rcccccccl}
 A_0&\subset&A_1&\subset&\cdots&\subset&A_r&\subset&A\\
\downarrow\ &&\downarrow&&&&\downarrow&&\downarrow\rdm\\
\Z/p\Z=\F_{\m_0}\!\!&\subset &\ \F_{\m_1}&\subset&\cdots&\subset&\ \F_{\m_r}&\subset& \F_\m.
\end{array}
\end{equation}

The branches of any node $\n\in\mx_p(A_i)$ are parameterized by the irreducible factors of $\op{red}_\n(t_i)$ in $\F_\n[y]$. Choose  monic polynomials $\psi\in A_i[y]$ such that $\op{red}_\n(\psi)$ are these irreducible factors. Each pair $(\n,\psi)$ determines a maximal ideal of $A_i[y]$, and the class modulo $t_iA_i[y]$ of this ideal is a maximal ideal of $A_{i+1}=A_i[y]/(t_i)$ whose intersection with $A_i$ is $\n$.

\subsection{Admissible $\phi$-expansions with respect to irreducible types}\label{subsecAdm}
Consider an irreducible type of order $r\ge0$ over $(\Z_p,\ord_p)$:
$$\ty^\star=\left(\psi_0;(\phi_1,\la_1,\psi_1);\dots;(\phi_r,\la_r,\psi_r)\right).$$
The inductive Artin algebra $\F$ associated with $\ty^\star$ is a chain of finite fields:
$$
\Z/p\Z=\F_0\subset \F_1\subset\cdots\subset \F_r\subset \F_{r+1}=\F.
$$
All polynomials in $\Z_p[x]$ are $\ty^\star$-robust. Also, a representative of $\ty^\star$ is necessarily irreducible in $\Z_p[x]$ \cite[Thm. 2.11]{GMN}.

Denote by $N^\star_i$, $v_i^\star$, $R^\star_i$ the Newton polygon operators, valuations and residual polynomial operators associated with $\ty^\star$, respectively.

Denote $\phi=\phi_r$ and consider an arbitrary $\phi$-expansion of a non-zero polynomial $f\in\Z_p[x]$, not necessarily the canonical one:
\begin{equation}\label{phdev}
f=b_0+b_1\phi+\cdots+b_s\phi^s+ \cdots,\quad b_s\in\Z_p[x].
\end{equation}

Take $u'_s=v^\star_{r-1}(b_s\phi^s)$ for all $s\ge0$, and let $M$ be the Newton polygon obtained as the lower convex hull of the set of points $(s,u'_s)$ with $b_s\ne0$.

To any integer abscissa $0\le s$ we attach a residual coefficient as usual:
$$\as{1.2}
c'_s=\left\{\begin{array}{ll}
0,&\mbox{ if $(s,u'_s)$ lies above }M, \mbox{ or }u'_s=\infty,\\
(z^\star_{r-1})^{\nu^\star_{r-1}(b_s)}R^\star_{r-1}(b_s)(z^\star_{r-1}),&\mbox{ if $(s,u'_s)$ lies on }M.
\end{array}
\right.
$$
For the points $(s,u'_s)$ lying on $M$ we now can have $c'_s=0$, because $R^\star_{r-1}(b_s)$ could be divisible by $\psi_{r-1}$ in $\F_{r-1}[y]$.

Let $S$ be the $\la_r$-component of $M$ (Definition \ref{component}), with endpoints having abscissas $s_0\le s'_0=s_0+de_r$. We can define the residual polynomial
$$
R'_r(f)=c'_{s_0}+c'_{s_0+e_r}\,y+\cdots+c'_{s_0+de_r}\,y^{d}\in\F_r[y].
$$

\begin{definition}\label{adm}
We say that the $\phi$-expansion (\ref{phdev}) is \emph{admissible} if $c'_s\ne0$
for each abs\-cissa $s$ of a vertex of $M^-$.
\end{definition}

Admissible $\phi$-expansions yield the same principal Newton polygon and the same residual polynomials as the canonical $\phi$-expansion.

\begin{definition}\label{polygons}
Let $N\subset \R_{\ge0}\times\R$ be a Newton polygon, and let $i_0$ be the abscissa of the left endpoint of $N$.
For any  $i_0\le s\le\ell(N)$, let $(s,y_s(N))$ be the unique point on $N$ of abscissa $s$.
For $0\le s<i_0$ we take $y_s(N)=\infty$.

We say that $N$ lies on or above a Newton polygon $M$, and we write $N\ge M$, if $y_s(N)\ge y_s(M)$ for all $0\le s\le \mn\{\ell(N),\ell(M)\}$.
\end{definition}

\begin{lemma}{\cite[Lem. 1.12]{GMN}}\label{admissible}
For any $\phi$-expansion one has $(N^\star_r)^-(f)\ge M^-$.
If the $\phi$-expansion is admissible, then $(N^\star_r)^-(f)=M^-$ and $R^\star_r(f)=R'_r(f)$.
\end{lemma}

\subsection{Tree of irreducible $p$-types attached to an SF-type}\label{subsecTree}
Recall that $\ty_{r-1}=\op{Trunc}_{r-1}(\ty)$. From now on, we make the following\medskip

\noindent{\bf Assumption: }Either $\rho:=\ord_p(N)=1$ or $\ty_{r-1}$ is unramified.\medskip

The next result reveals, under certain conditions, some arithmetic information about the prime $p$ encoded by the SF-type $\ty$.

\begin{theorem}\label{comparison}
For each $\m\in \mpa$ there exists an irreducible type
$$\ty_\m=\left(\psi_{\m,0};(\phi_{\m,1},\la_{\m,1},\psi_{\m,1});\dots;(\phi_{\m,r},\la_{\m,r},\psi_{\m,r})\right)$$
over $(\Z_p,\ord_p)$,
uniquely determined by the following conditions, where all data and operators of $\ty_\m$ are marked with a subscript $\m$:
\medskip

{\bf (A)} \ Each $\phi_{\m,i}\in\Z_p[x]$ is a $p$-adic irreducible factor of $g_i$, and $\la_{\m,i}=\rho\la_i$, for $1\le i\le r$. In particular, $e_{\m,i}=e_i$ for $i<r$, and $e_{\m,r}=e_r/\gcd(\rho,e_r)$.\medskip

{\bf (B)} \ Denote $G_{\m,i}=g_i/\phi_{\m,i}\in\Z_p[x]$ for $1\le i\le r$. The following constants, which depend only on $\ty$ and $\m$, are non-zero:
$$
\xi_i=
z_{\m,i-1}^{\nu_{\m,i-1}(G_{\m,i})}R_{\m,i-1}(G_{\m,i})(z_{\m,i-1})\in\F_{\m,i}^*.
$$

Hence, starting with the initial values $\chi_0=\rdp(N/p^\rho)$, $\sigma_0=1$, we may consider non-zero constants in $\F_{\m,i}$ defined by the recurrent formulas:
$$
\chi_i=\left(\chi_{i-1}\right)^{\ell'_i-\ell_iV_i}\xi_{i}^{\ell_i},\quad
\sigma_i=\left(\chi_{i-1}\right)^{-(e_iV_i+h_i)}\xi_{i}^{e_i},\quad 1\le i\le r.
$$

{\bf (C)} \ There is a commutative diagram of vertical isomorphisms
$$
\begin{array}{ccccccccc}
\F_{\m_0}\ &\subset& \F_{\m_1}\ &\subset&\cdots &\subset &\F_{\m_r}&\subset&\F_{\m_{r+1}}=\F_\m \\
\ \,\|\iota_0 &&\ \downarrow\iota_1&&\cdots&&\ \downarrow\iota_r&&\downarrow\iota_{r+1}\quad\\
\F_{\m,0}&\subset& \F_{\m,1}\ &\subset&\cdots &\subset &\F_{\m,r}&\subset&\F_{\m,r+1}\quad\
\end{array}
$$
determined by $\iota_{i+1}\left(\rdm(z_i)\right)=\sigma_i z_{\m,i}$, \,for $0\le i\le r$.

We shall consider these isomorphisms as identities. Thus, the mapping $\rdm$ in (\ref{chainF}) may be considered as a homomorphism between the inductive Artin algebras of $\ty$ and $\ty_\m$.
Under these identifications, we have  $$\rdm(z_i)=\sigma_i z_{\m,i},\quad \psi_{\m,i}(y)=\sigma_i^{-\deg(\psi)} \rdm(\psi)(\sigma_iy),\quad 0\le i\le r,$$
where $\m_{i+1}=(\m_i,\psi)$ is the parameterization described in section \ref{subsecTreeM}.\medskip

{\bf (D)} \ The assignment $\n\mapsto\ty_\n$ yields a tree isomorphism between the subtree
$\mm_p^0\subset\mm_p$  obtained by deleting the root node $\m_0$, and the full finite subtree of $\tcal(\Z_p,\ord_p)$ having as leaves the types $\ty_\m$ for all $\m\in\mpa$.\medskip

{\bf (E)} \ The sets $\rep(\ty_\m)$, for $\m$ running on $\mpa$, are pairwise disjoint.
Moreover, $\sum\nolimits_{\m\in\mpa}m_{\m,r+1}=m_{r+1}/\gcd(\rho,e_r)$.\medskip

{\bf (F)} \ Let $f\in\oo[x]$ be a $\ty$-robust polynomial.

\begin{enumerate}
\item[(i)\;] If $r>0$, we have $N_{\m,r}^-(f)=E_\rho(N_r^-(f))$, where $E_\rho\colon \R^2\to \R^2$ is the affinity $E_\rho(x,y)=(x,\rho y)$.
\item[(ii)] $R_{\m,r}(f)(y)=\tau_r(f)\rdm\left(R_r(f)\right)(\sigma_ry^{\gcd(\rho,e_r)})$,
where $\tau_0(f)=\chi_0^{v_0(f)}$, and $\tau_r(f)=\left(\chi_{r-1}\right)^{u_r(f)-s_r(f)V_r}\xi_{r}^{s_r(f)}$ for $r>0$.
\end{enumerate}\medskip

{\bf (G)} \ For any $f\in\oo[x]$, $v_{\m,r}(f)\ge(\rho/\gcd(\rho,e_r))\, v_r(f)$.
If $f$ is $\ty$-ro\-bust, then equality holds.\medskip

{\bf (H)} \ Suppose that $\gcd(\rho,e_r)=1$. Every $g\in\rep(\ty)$ factors in $\Z_p[x]$ as a product of monic irreducible polynomials:
$$
g=\prod\nolimits_{\m\in\mpa} \phi_\m,\qquad\phi_\m\in\rep(\ty_\m),
$$
such that $\ty_\m\nmid \phi_\n$ for all $\m,\n\in\mpa$, $\m\ne \n$.
\end{theorem}

\begin{proof}
With the notation of section \ref{subsecTreeM}, let $\m=(\m_r,\psi)$, for some monic  $\psi\in A_r[y]$ such that $\rdm(\psi)$ is an irreducible factor of $\rdm(t_r)$ in $\F_{\m_r}[y]$.

We proceed by induction on the order $r$ of $\ty$. First, suppose $r=0$.

We take $\ty_\m=(\psi_{\m,0})$, for $\psi_{\m,0}=\rdm(\psi)$. Thus, $\F_{\m_0}=\Z/p\Z=\F_{\m,0}$ and
$$
\F_\m=A_1/\m\simeq \F_{\m_0}[y]/(\psi_{\m,0})=\F_{\m,1}.
$$
This isomorphism sends $\rdm(z_0)$ to $z_{\m,0}$.
Since $t_0$ is squarefree, the factors $\psi_{\m,0}$ are pairwise different. This proves (C) and (D).

Clearly, $v_{\m,0}=\ord_p\ge\rho\, v_0$. Equality holds for $\ty$-robust polynomials by Lemma \ref{v0good}. This proves  (G).

For a non-zero $f\in\oo[x]$, $R_{\m,0}(f)=\rdp\left(f/p^{v_{\m,0}(f)}\right)$ and $R_0(f)=\rdn\left(f/N^{v_0(f)}\right)$. Since $\rdp=\rdm\circ\rdn$, for a robust $f$ we get by (G):
\begin{equation}\label{RR}
R_{\m,0}(f)=\rdp\left(N/p^\rho\right)^{v_0(f)}\rdm\left(R_0(f)\right)=\tau_0(f)\rdm\left(R_0(f)\right).
\end{equation}
This proves (F).

Let $g\in\rep(\ty)$. Since $v_0(g)=0$, we have $\tau_0(g)=1$ and (\ref{RR}) shows that
$$
\rdp(g)=R_{\m,0}(g)=\rdm\left(R_0(g)\right)=\rdm(t_0)=\prod\nolimits_{\m\in\mx_p(A)}\psi_{\m,0}.
$$
By Hensels' lemma, $g=\prod\nolimits_{\m\in\mx_p(A)}\phi_{\m}$, for monic irreducible $\phi_\m\in\Z_p[x]$ with $\rdp(\phi_\m)=\psi_{\m,0}$. Each $\phi_\m$ is a representative of $\ty_\m$ and the sets $\rep\left(\ty_\m\right)$ are clearly pairwise disjoint. Also,
$$
\sum\nolimits_\m m_{\m,1}=\sum\nolimits_\m \deg \phi_\m=\deg g=m_1.
$$
This proves (E) and (H). The proof for types of order zero is complete.

We now assume $r>0$ and the existence of an irreducible $p$-type
$$
\ty_{\m_r}=\left(\psi_{\m,0};(\phi_{\m,1},\la_{\m,1},\psi_{\m,1});\dots;(\phi_{\m,r-1},\la_{\m,r-1},\psi_{\m,r-1})\right),
$$
such that for each statement (X) of the theorem, the analogous statement (X$_{r-1}$) concerning $\ty_{r-1}$ is true.


By condition (H$_{r-1}$), the $p$-adic factorization of $g_r$ takes the form:
$$
g_r=\prod\nolimits_{\n\in\mx_p(A_r)}\phi_\n.
$$
Take $\phi_{\m,r}=\phi_{\m_r}\in\rep(\ty_{\m_r})$.  Consider the irreducible $p$-type of order $r$:
$$
\ty_\m:=\left(\ty_{\m_r};(\phi_{\m,r},\la_{\m,r},\psi_{\m,r})\right),
$$
where $\la_{\m,r}\in\Q_{>0}$ and $\psi_{\m,r}\in\F_{\m,r}[y]$ monic irreducible, with $\psi_{\m,r}(0)\ne0$, are uniquely determined by (A), (B) and (C):
$$
\la_{\m,r}=\rho\la_r,\quad \psi_{\m,r}(y)=\sigma_r^{-\deg(\psi)}\rdm(\psi)(\sigma_r y).
$$
Since $t_r(0)$ is a unit in $A_r$, we have $\rdm(\psi)(0)\ne0$. Hence, we need only to prove (B); that is, $\xi_r\ne0$, to ensure that $\sigma_r\ne0$ and $\psi_{\m,r}$ is well defined.

By Theorem \ref{basicprs}, the operator $R_{\m,r-1}$ is multiplicative; hence,
\begin{equation}\label{delta}
\delta:=R_{\m,r-1}(g_r/\phi_{\m,r})(z_{\m,r-1})=\prod\nolimits_{\n\ne \m_r}R_{\m,r-1}(\phi_{\n})(z_{\m,r-1}).
\end{equation}
For every $\n\in\mx_p(A_r)$, $\n\ne\m_r$, we have $\ty_{\m_r}\nmid \phi_\n$ by (H$_{r-1}$). Hence, $\psi_{\m,r-1}\nmid R_{\m,r-1}(\phi_\n)$, and $R_{\m,r-1}(\phi_{\n})(z_{\m,r-1})\ne0$. Thus, the constant $\delta$ of (\ref{delta}) is nonzero, and this implies $\xi_r\ne0$. In fact, $\xi_1=\delta$ for $r=1$, whereas for $r>1$, $\xi_r$ is equal to $\delta$ times a power of the unit $z_{\m,r-1}$.

Therefore, our type $\ty_\m$ is well defined and satisfies (A), (B) and (C).

By construction, $\n$ is the previous node of $\m$ in $\mm_p^0$ if and only if $\ty_{\n}$ is the previous node of $\ty_\m$ in $\tcal$. This proves (D).\medskip

Let us prove (E). Since any $g\in\rep(\ty_\m)$ is irreducible in $\Z_p[x]$, \cite[Lem. 2.4]{GMN} shows that $R_{\m,i}(g)$ is a power of the irreducible polynomial $\psi_{\m,i}$ in $\F_{\m,i}[y]$, for all $0\le i\le r$. Now, for $\n\in\mpa$, $\n\ne\m$, let $\m_i=\m\cap A_i$ be the greatest common node in the paths joining $\m$ and $\n$ with $\m_0$. By the construction of $\ty_\m$ and $\ty_{\n}$, we have $\psi_{\m,i}\ne\psi_{\n,i}$, so that $g$ cannot be a representative of $\ty_{\n}$. Thus, the sets $\rep(\ty_\m)$ and $\rep(\ty_{\n})$ are disjoint.

The branches $\m=\n_1,\dots,\n_k$ of $\m_r$ are determined by monic liftings $\psi=\psi_1,\dots,\psi_k\in A_r[y]$ of the irreducible factors of the squarefree polynomial $\rdm(t_r)$ in $\F_{\m,r}[y]$. Thus,
\begin{equation}\label{branches}
\sum\nolimits_{j=1}^kf_{\n_j,r}=\deg t_r=f_r.
\end{equation}
Let $d=\gcd(\rho,e_r)$. By (A), $e_{\m,r}=e_r/d$ for all $\m\in\mpa$. By (\ref{branches}),
\begin{align*}
\sum\nolimits_{j=1}^km_{\n_j,r+1}=&\;
\sum\nolimits_{j=1}^ke_{\n_j,r}f_{\n_j,r}m_{\n_j,r}\\=&\;\sum\nolimits_{j=1}^k(e_r/d)f_{\n_j,r}m_{\m_r,r}=e_rf_rm_{\m_r,r}/d.
\end{align*}
By our assumptions, $\gcd(\rho,e_i)=1$ for all $i<r$. Thus, by (E$_{r-1}$):
$$
\sum\nolimits_{\m\in\mpa}m_{\m,r+1}=
\sum\nolimits_{\n\in\mx_p(A_r)}e_rf_rm_{\n,r}/d=e_rf_rm_r/d=m_{r+1}/d.
$$
This ends the proof of (E).\medskip

Let us prove (F). Assume that $f\in\oo[x]$ is $\ty$-robust. Denote $g=g_r$, $\phi=\phi_{\m,r}$ and $G=g/\phi$. The canonical $g$-expansion of $f$ induces in a natural way a $\phi$-expansion:
\begin{equation}\label{phiexp}
f=\sum\nolimits_{0\le s}a_sg^s=\sum\nolimits_{0\le s}b_s\phi^s,\quad b_s=a_sG^s.
\end{equation}

Let $u_s=v_{r-1}(a_sg^s)$, $u'_s=v_{\m,r-1}(b_s\phi^s)=v_{\m,r-1}(a_sg^s)$.
Let $M$ be the Newton polygon determined by the $\phi$-expansion (\ref{phiexp}); that is, $M$ is the lower convex hull of the set of points $\left\{(s,u'_s)\mid s\ge0\right\}$.

Since $a_s$ and $g$ are $\ty_{r-1}$-robust, (G$_{r-1}$) and Theorem \ref{basicprs},(C) show that
\begin{equation}\label{uu}
u'_s=v_{\m,r-1}(a_s)+s\,v_{\m,r-1}(g)= \rho\left(v_{r-1}(a_s)+s\,v_{r-1}(g)\right)=\rho\,u_s,
\end{equation}
for all $s\ge0$. Hence, $M=E_\rho(N_r(f))$.

We now proceed to compare $M^-$ with $N_{\m,r}^-(f)$ by using Lemma \ref{admissible}.

By item (C) of Theorem \ref{basicprs}, the functions $s_{\m,r-1}$ and $u_{\m,r-1}$ convert a product of polynomials into a sum of integers. Therefore, the function $\nu_{\m,r-1}=\ell'_{\m,r-1}s_{\m,r-1}-\ell_{\m,r-1}u_{\m,r-1}$
has the same property.

Hence, by using (F$_{r-1}$),(ii) and (C$_{r-1}$), for any abscissa $s$ such that $(s,u'_s)$ lies on $M$, the residual coefficient $c'_s$ may be expressed as:
\begin{equation}\label{cs}
\begin{array}{rl}
c'_s=&\!\!z_{\m,r-1}^{\nu_{\m,r-1}(b_s)}R_{\m,r-1}(b_s)(z_{\m,r-1})\\=&\!\!z_{\m,r-1}^{\nu_{\m,r-1}(a_s)+s\nu_{\m,r-1}(G)}R_{\m,r-1}(a_s)(z_{\m,r-1})R_{\m,r-1}(G)^s(z_{\m,r-1})\\=&\!\xi_r^s\,
z_{\m,r-1}^{\nu_{\m,r-1}(a_s)}R_{\m,r-1}(a_s)(z_{\m,r-1})\\
=&\!\!\xi_r^s\,
z_{\m,r-1}^{\nu_{\m,r-1}(a_s)}\tau_{r-1}(a_s)\rdm\left(R_{r-1}(a_s)\right)(\sigma_{r-1}z_{r-1}),\\
=&\!\!\xi_r^s\,
z_{\m,r-1}^{\nu_{\m,r-1}(a_s)}\tau_{r-1}(a_s)\rdm\left(R_{r-1}(a_s)(z_{r-1})\right)\ne0,
\end{array}
\end{equation}
because, being $f$ $\ty$-robust, $R_{r-1}(a_s)(z_{r-1})$ is a unit in $A_r$ by Lemma \ref{gcd}.

Therefore, the $\phi$-expansion (\ref{phiexp}) is admissible with respect to $\ty_\m$, and Lemma \ref{admissible} shows that $M^-=N_{\m,r}^-(f)$ and $R'_r(f)=R_{\m,r}(f)$.

The equality between Newton polygons implies $N_{\m,r}^-(f)=E_\rho(N_r^-(f))$.

For the proof of item (ii) we need the following fact.\medskip

\noindent{\bf Claim. } $\nu_{\m,r-1}(a_s)=\nu_{r-1}(a_s)$ for all $s\ge0$.
\medskip


By (F$_{r-1}$), $N_{\m,r-1}^-(a_s)=E_\rho(N_{r-1}^-(a_s))$.
The affinity $E_\rho$ maps the $\la_{r-1}$-component of $N_{r-1}^-(a_s)$ into the $\rho\la_{r-1}$-compo\-nent of $N_{\m,r-1}(a_s)$; thus,
\begin{equation}\label{su}
s_{\m,r-1}(a_s)=s_{r-1}(a_s),\quad u_{\m,r-1}=\rho\, u_{r-1}(a_s).
\end{equation}

By (A), $e_{\m,r-1}=e_{r-1}$. By our assumptions, either $e_{r-1}=1$ or $\rho=1$.

If $e_{r-1}=e_{\m,r-1}=1$, then $\ell_{\m,r-1}=\ell_{r-1}=0$ and $\ell'_{\m,r-1}=\ell'_{r-1}=1$.

If $\rho=1$, then $\la_{\m,r-1}=\la_{r-1}$, so that $h_{\m,r-1}=h_{r-1}$. Hence, $\ell_{\m,r-1}=\ell_{r-1}$ and $\ell'_{\m,r-1}=\ell'_{r-1}$.
The Claim follows in both cases from (\ref{su}).\medskip


We now compare $R_r(f)$ with $R'_r(f)=R_{\m,r}(f)$. Let $s_0=s_r(f)\le s'_0$ be the abscissas of the endpoints of the $\la_r$-component $S$ of $N_r^-(f)$. The affinity $E_\rho$ maps $S$ into the $\rho\la_r$-component of $M$; thus, the latter component has endpoints with the same abscissas $s_0\le s'_0$.

For any integer $s_0\le s\le s'_0$ we have $u'_s=\rho u_s$ by (\ref{uu}). Hence, the point  $(s,u_s)$ lies above $N_r^-(f)$ if and only if  the point  $(s,u'_s)$ lies above $M^-=E_\rho(N_r^-(f))$. In this situation, we have $c_s=0=c'_s$.

Let now $s=s_j=s_0+je_r$ be such that $(s,u_s)$ lies on $S$. The identity
\begin{equation}\label{end}
\xi_r^{s_j}\,\sigma_{r-1}^{-\nu_{r-1}(a_{s_j})}
\tau_{r-1}(a_{s_j})=\tau_r(f)\,\sigma_r^j,
\end{equation}
is straightforward to deduce from the definition of each constant. We need only to have in mind two identities:
$$
\begin{array}{l}
u_{r-1}(a_{s_j})+s_{r-1}(a_{s_j})\la_{r-1}=v_{r-1}(a_{s_j})/e_{r-1},\\
v_{r-1}(a_{s_j})+s_jV_r+jh_r=u_{s_j}+(je_r)\la_r=u_r(f),
\end{array}
$$
which follow from the definition of $v_{r-1}$ and the fact that $(s_j,u_{s_j})$ lies on $S$.

By (\ref{cs}), the Claim, (C$_{r-1}$) and (\ref{end}), we have:
$$
\begin{array}{rl}
c'_{s_j}=&\!\!\xi_r^{s_j}\,
z_{\m,r-1}^{\nu_{r-1}(a_{s_j})}\tau_{r-1}(a_{s_j})\rdm\left(R_{r-1}(a_{s_j})(z_{r-1})\right)\\
=&\!\!\xi_r^{s_j}\,\sigma_{r-1}^{-\nu_{r-1}(a_{s_j})}
\tau_{r-1}(a_{s_j})\rdm\left(c_{s_j}\right)\,=\,\tau_r(f)\sigma_r^j\rdm(c_{s_j}).
\end{array}
$$

Item (ii) follows from this identity. In fact, $c_{s_j}$ is the coefficient of $y^j$ in $R_r(f)$, whereas $c'_{s_j}$ is the coefficient of $y^{jd}$ in $R'_r(f)$, where $d=\gcd(\rho,e_r)$, because  $s_j=s_0+je_r=s_0+jde_{\m,r}$.
This ends the proof of (F).\medskip

Let us prove (G). By Lemma \ref{admissible}, $N_{\m,r}^-(f)\ge M^-$. Also, $u'_s\ge \rho\,u_s$ for all $s$, by (G$_{r-1}$). Thus, $M\ge E_\rho(N_r(f))$, leading to $N_{\m,r}^-(f)\ge E_\rho(N_r^-(f))$.

By definition, $v_{\m,r}(f)/e_{\m,r}$ is the ordinate where the line of slope $-\rho\la_r$ first touching $N_{\m,r}^-(f)$ from below cuts the vertical axis. Since $v_r(f)/e_r$ admits the same interpretation with respect to the polygon $N_r^-(f)$ and the similar line of slope $-\la_r$, we deduce that $v_{\m,r}(f)/e_{\m,r}\ge \rho\, v_r(f)/e_r$. By (A), this implies
$v_{\m,r}(f)\ge (\rho/\gcd(\rho,e_r))\, v_r(f)$.

Also, this argument shows that equality holds if $N_{\m,r}^-(f)=E_\rho(N_r(f))$, which follows from item (i) of (F) for a $\ty$-robust $f$.  \medskip

Finally, let us prove (H). Let $g$ be a representative of $\ty$, so that $R_r(g)=t_r$. Since $g$ is $\ty$-robust, item (ii) of (F) shows that
$$R_{\m,r}(g)(y)=\tau_r(g)\rdm(R_r(g))(\sigma_ry)=\tau_r(g)\rdm(t_r)(\sigma_ry).$$
Since $t_r$ is squarefree, the irreducible factor $\psi_{\m,r}$ of $\rdm(t_r)(\sigma_ry)$ divides $R_{\m,r}(g)$ only once; in other words, $\ord_{\ty_\m}(g)=1$.

Let $g=\phi_1\cdots \phi_k$ the factorization of $g$ into a product of monic irreducible polynomials in $\Z_p[x]$. Since $\ty_\m$ is irreducible, $$1=\ord_{\ty_\m}g=\ord_{\ty_\m}\phi_1+\cdots +\ord_{\ty_\m}\phi_k.$$ Hence, there exists an index $1\le i\le k$ such that $\ord_{\ty_\m}\phi_i=1$ and $\ord_{\ty_\m}\phi_j=0$ for all $j\ne i$. Let us denote $\phi_\m=\phi_i$.

An irreducible $\phi_\m\in\Z_p[x]$ with $\ord_{\ty_\m}\phi_\m=1$ satisfies $R_{\m,r}(\phi_\m)=\psi_{\m,r}$ and has degree $m_{\m,r+1}$ \cite[Lem. 2.4]{GMN}. Hence, $\phi_\m$ is a representative of $\ty_\m$.

By (E), these factors $\phi_\m$ are pairwise different, and: $$\deg g=m_{r+1}=\sum\nolimits_{\m\in\mpa}m_{\m,r+1}=\sum\nolimits_{\m\in\mpa}\deg\phi_\m.$$
Hence, $g=\prod_{\m\in\mpa}\phi_\m$.
\end{proof}





\begin{corollary}\label{suGr}
For $r>1$ and any $\m\in\mpa$ we have:
$$
s_{\m,r-1}(G_{\m,r})=0,\quad
u_{\m,r-1}(G_{\m,r})=\left(\rho V_r-V_{\m,r}\right)/e_{r-1}.
$$
\end{corollary}

\begin{proof}
As mentioned after Definition \ref{repr}, $N_{r-1}(g_r)$ is one-sided of slope $-\la_{r-1}$ and has left endpoint $(0,V_r/e_{r-1})$, whereas   $N_{\m,r-1}(\phi_{\m,r})$ is one-sided of slope $-\rho\la_{r-1}$ and has left endpoint $(0,V_{\m,r}/e_{r-1})$. Hence,
$$
\begin{array}{ll}
s_{r-1}(g_r)=0,& u_{r-1}(g_r)=V_r/e_{r-1},\\
s_{\m,r-1}(\phi_{\m,r})=0,& u_{\m,r-1}(\phi_{\m,r})=V_{\m,r}/e_{r-1}.
\end{array}
$$
On the other hand, (F) of Theorem \ref{comparison} shows that
$$s_{\m,r-1}(g_r)=s_{r-1}(g_r)=0,\quad u_{\m,r-1}(g_r)=\rho\, u_{r-1}(g_r)=\rho\, V_r/e_{r-1}.$$

By Corollary \ref{product}, $s_{\m,r-1}(G_{\m,r})=s_{\m,r-1}(g_r)-s_{\m,r-1}(\phi_{\m,r})=0$, and
$u_{\m,r-1}(G_{\m,r})=u_{\m,r-1}(g_r)-u_{\m,r-1}(\phi_{\m,r})=\left(\rho V_r-V_{\m,r}\right)/e_{r-1}$.
\end{proof}

\begin{corollary}\label{squarefree}
Let $f\in\oo[x]$ be a $\ty$-robust polynomial and let $\m\in\mpa$.

\begin{enumerate}
\item If $r\ge1$, then  $\ord_{\ty_{r-1}}(f)=\ord_{\ty_{\m,r-1}}(f)$.
\item If $\gcd(\rho,e_r)=1$, then $\ord_{\ty}(f)\le \ord_{\ty_\m}(f)$. Equality holds if moreover $R_r(f)=t_r^{\ord_\ty(f)}q$ in $A_r[y]$, with $\gcd_{A_r}(q,t_r)=1$.
\end{enumerate}
\end{corollary}

\begin{proof}

Item (1) follows from Theorem \ref{comparison},(F) and Corollary \ref{length}.

Let $\rdm(\psi)$ be the irreducible factor of $\rdm(t_r)$ corresponding to $\m$, for some $\psi\in A_r[y]$. By (C) and (F) of Theorem \ref{comparison},
\begin{align*}
\ord_\ty f=&\ \ord_{t_r}R_r(f)\le\ord_{\rdm(t_r)}\rdm(R_r(f))\le\ord_{\rdm(\psi)}\rdm(R_r(f))\\
=&\ \ord_{\rdm(\psi)(\sigma_ry)}\rdm(R_r(f))(\sigma_ry)=\ord_{\psi_{\m,r}}R_{\m,r}(f)=\ord_{\ty_\m}f.
\end{align*}
If $R_r(f)=t_r^{\ord_\ty(f)}q$, with $\gcd_{A_r}(q,t_r)=1$, the two inequalities in this chain become equalities, by Lemma \ref{gcd}.
\end{proof}

\section{Arithmetic properties of number fields encoded by SF-types}\label{secArit}

Let $f\in\Z[x]$ be a monic irreducible polynomial of degree $n>1$. Consider the number field $K=\Q(\t)$ generated by some root $\t\in\qb$ of $f$. Let $\Z_K$ be the ring of integers of $K$.

For a given prime number $p$, let $\ord_p\colon \qb_p^{\,*}\longrightarrow \Q$ be the canonical extension of the $p$-adic valuation to an algebraic closure of $\Q_p$.

Let $\pp$ be the set of prime ideals in $\Z_K$ lying above $p$.
For any $\p\in\pp$, denote $\F_\p=\Z_K/\p$. Let $e_\p$ be the ramification index of $\p$, and $f_\p=\dim_{\Z/p\Z}\F_\p$ its residual degree.  Let $v_{\p}$ be the discrete valuation on $K$ induced by $\p$ and let $\Z_\p\subset K$ be its valuation ring. The canonical isomorphism $\F_\p\simeq  \Z_\p/\p\Z_\p$ will be considered as an identity.

Consider the normalized valuation:
$$
w_\p\colon K\lra \Q\cup\{\infty\}, \quad w_\p(\alpha)=v_\p(\alpha)/e_\p,
$$
which extends $\ord_p$ to $K$.
Endow $K$ with the $\p$-adic topology and fix a topological embedding $\iota_\p\colon K\hookrightarrow \qb_p$. Then,
$$
w_\p(\alpha)=\ord_p(\iota_\p(\alpha)),\quad\forall \,\alpha\in K.
$$

\noindent{\bf Notation. }For any $h\in\Z_p[x]$ we abuse of language and write $w_\p(h(\t))$ instead of $\ord_p(h(\iota_\p(\t)))$.\medskip

The polynomial $f$ factorizes in $\Z_p[x]$ as  $f=\prod_{\p\in\pp}F_\p$, where $F_\p$ is the minimal polynomial of $\iota_\p(\t)$ over $\Q_p$.

Consider an integer $N>1$ and an SF-type over $(\Z,\ord_N)$ of order $r>0$,
$$
\ty=\left(t_0;(g_1,\la_1,t_1);\dots;(g_r,\la_r,t_r)\right),
$$
with truncates $\ty_i=\op{Trunc}_i(\ty)$ for $0\le i\le r$. Suppose that  $\rho:=\ord_p(N)>0$.

Throughout this section we make the following\medskip

\noindent{\bf Assumptions: }(1) \ $f$ is $\ty_i$-robust for all $0\le i\le r$,

 (2) \ Either $\rho=1$, or $\ty$ is unramified.

\subsection{Prime ideals attached to an SF-type}
Let $A$ be the artinian algebra attached to $\ty$, and let $\{\ty_\m\mid \m\in\mpa\}$ be the  tree of irreducible $p$-types associated with $\ty$ in Theorem \ref{comparison}. We define
$$
\pt=\bigcup\nolimits_{\m\in\mpa}\ptm,\qquad \ptm:=\left\{\p\in\pp\mid\; \ty_\m\mid F_\p\right\}.
$$
As we saw in the proof of item (E) of Theorem \ref{comparison}, an irreducible polynomial in $\Z_p[x]$ cannot be divided by two different $\ty_\m$'s; hence, these sets $\ptm$ are pairwise disjoint.

If $\ty\mid f$, then $\ty_\m\mid f$ for all $\m$, by Corollary \ref{squarefree}. Since $$0<\ord_{\ty_\m}(f)=\sum\nolimits_{\p\in\pp}\ord_{\ty_\m}(F_\p),$$ the sets $\ptm$ are all non-empty in this case.

Let us fix $g\in\rep(\ty)$ and denote $N_{r+1}=N_{v_r,g}$. For each $\m\in\mpa$, let $\phi_\m\in\Z_p[x]$ be the irreducible factor of $g$ attached to $\m$ in Theorem \ref{comparison}.

\begin{lemma}\label{Pslope}
Let $\p\in\pp_{\ty_\m}$ for some $\m\in\mpa$. Then, there is a unique slope $-\mu$ of $N_{r+1}^-(f)$ and a unique monic irreducible factor $\psi$ of $R_{v_{\m,r},\phi_\m,\rho\mu}(f)$ such that the irreducible type $(\ty_\m;(\phi_\m,\rho\mu,\psi))$ divides $F_\p$.
\end{lemma}

\begin{proof}
By Theorem \ref{comparison}, $N_{v_{\m,r},\phi_\m}(f)=E_\rho(N_{r+1}(f))$, so that the slopes of this polygon are $-\rho\la$, for $-\la$ running on the slopes of $N_{r+1}(f)$.

The result follows from \cite[Thms. 3.1, 3.7]{GMN}.
\end{proof}

\begin{proposition}\label{splitting}
Suppose that $\ty\mid f$ and for each slope $-\la$ of $N_{r+1}^-(f)$ Algorithm \ref{sfdA} outputs a squarefree decomposition of $R_{v_r,g,\la}(f)$ in $A[y]$ with strongly unitary squarefree factors.  Consider the types $\ty_{\la,t}=(\ty;(g,\la,t))$ for $-\la$ a slope of $N_{r+1}^-(f)$ and $t\in A[y]$ a squarefree factor of $R_{v_r,g,\la}(f)$.

If the least positive denominator $e_\la$ of every slope $-\la$ satisfies  $\gcd(\rho,e_\la)=1$, then we have a splitting of $\pt$ into a union of pairwise disjoint subsets:
$$\pt=\bigcup\nolimits_{(\la,t)}\pp_{\ty_{\la,t}}$$
\end{proposition}

\begin{proof}
It is easy to check that these $\pp_{\ty_{\la,t}}$ are pairwise disjoint subsets of  $\pp_\ty$. Let us show that they cover $\pp_\ty$.

Take $\p\in\pp_\ty$. There is a unique $\m\in\mpa$ such that $\p\in\pp_{\ty_\m}$. By Lemma \ref{Pslope}, there is a type $\ty'_\m=(\ty_\m;(\phi_\m,\rho\mu,\psi))$ dividing $F_\p$, for a certain slope $-\mu$ of $N_{r+1}^-(f)$.

By Theorem \ref{comparison},(F) applied to the type $\ty_{\mu,t}$ (for any choice of $t$), we have $$R_{v_{\m,r},\phi_\m,\rho\mu}(f)=\tau_{r+1}(f)\rdm(R_{v_r,g,\mu}(f))(\sigma_{r+1}y),$$for some non-zero constants $\tau_{r+1}(f),\sigma_{r+1}\in\F_{\m}$. Take any monic $\varphi\in A[y]$ such that $\psi(y)=\sigma_{r+1}^{-\deg\varphi}\rdm(\varphi)(\sigma_{r+1}y)$. Clearly, there is a unique squarefree factor $t$ of $R_{v_r,g,\mu}(f)$ such that $\rdm(t)$ is divisible by the irreducible factor $\rdm(\varphi)$ of $\rdm(R_{v_r,g,\mu}(f))$ in $\F_\m[y]$. For this choice of $t$, let $A'=A[t]/(t)$ be the artinian inductive algebra associated with the type $\ty_{\mu,t}$. Clearly, $\m'=(\m,\varphi)$ determines a maximal ideal of $A'$ for which $\ty'_\m=(\ty_{\mu,t})_{\m'}$. Since, $\ty'_\m\mid F_\p$, the prime ideal $\p$ belongs to $\pp_{\ty_{\mu,t}}$.
\end{proof}

\subsection{Computation of $w_\p$ in terms of data of the SF-type}

\begin{proposition}\label{value}
For any $h\in\Z_p[x]$ we have
\begin{equation}\label{wpvalue}
w_\p(h(\t))\ge \rho\, v_r(h)/e_1\cdots e_r,\quad \forall\,\p\in\pt.
\end{equation}
If $h$ is $\ty$-robust and $\gcd_{A_r}(R_r(h),t_r)=1$, then equality holds.
\end{proposition}

\begin{proof}
Take $\p\in\ptm$ for some $\m\in\mpa$. By \cite[Prop. 2.9]{GMN},
\begin{equation}\label{wpvaluem}
w_\p(h(\t))\ge v_{\m,r}(h)/e_{\m,1}\cdots e_{\m,r},
\end{equation}
and equality holds if and only if $\ty_\m\nmid h$.
By Theorem \ref{comparison}, $e_{\m,i}=e_i$ for all $1\le i\le r$, and  $v_{\m,r}(h)\ge \rho\, v_r(h)$. This proves (\ref{wpvalue}).

If $h$ is $\ty$-robust, then $v_{\m,r}(h)=\rho\, v_r(h)$, by Theorem \ref{comparison}. If moreover $\gcd_{A_r}(R_r(h),t_r)=1$, then $\ty\nmid h$. By Corollary \ref{squarefree}, $\ty_\m\nmid h$, so that equality holds in (\ref{wpvaluem}). Thus, equality holds in (\ref{wpvalue}).
\end{proof}

\begin{corollary}\label{previous}For all $\p\in\pt$,
$w_\p(g_r(\t))=\rho(V_r+\la_r)/(e_1\cdots e_{r-1})$.
\end{corollary}

\begin{proof}
The Newton polygon $N_r(g_r)$ contains a single point $(1,V_r)$, so that $v_r(g_r)=e_r(V_r+\la_r)$ and $R_r(g_r)=1$. The desired equality follows from Proposition \ref{value}, since $g_r$ is $\ty$-robust.
\end{proof}

The next result completes the computation of $w_\p(g_r(\t))$ for all $\p\in\pp$.

\begin{lemma}\label{allps}
Let $\p\in\pp$ such that $\p\not\in\pp_{\ty_\ell}$ for some minimal $0\le \ell\le r$. If $\ell>0$, let $-\mu$ be the slope of  $N_\ell^-(f)$ attached to $\p\in\pp_{\ty_{\ell-1}}$, as explained in Lemma \ref{Pslope}. If $\ell=0$ take $\mu=0$. Then,
$$
w_\p(g_r(\t))=
\rho\, (m_r/m_\ell)\,(V_\ell+\delta)/(e_1\cdots e_{\ell-1}),
$$
where $\delta=\mu$ if $\ell=r$, and $\delta=\mn\{\la_\ell,\mu\}$ if $\ell<r$.
\end{lemma}

\begin{proof}
By Theorem \ref{comparison}, $g_r=\prod_{\n\in\mx_p(A_r)}\phi_\n$, so that
\begin{equation}\label{sum}
w_\p(g_r(\t))=\sum\nolimits_{\n\in\mx_p(A_r)}w_\p(\phi_\n(\t)).
\end{equation}

If $\ell=0$, then $\phi_\n$ and $F_\p$ are congruent modulo $p$ to a power of two different monic irreducible polynomials in $\Z/p\Z$; hence $w_\p(\phi_\n(\t))=0$ for all $\n$, so that $w_\p(g_r(\t))=0$.

Suppose $\ell>0$. Choose any $\m\in\mx_p(A_r)$ such that $\m_\ell=\m\cap A_\ell$ is the unique ideal in $\mx_p(A_\ell)$ such that $\p\in\pp_{(\ty_{\ell-1})_{\m_\ell}}$.
The computation of $w_\p(\phi_\n(\t))$ for $\n\in\mx_p(A_r)$ depends on the relative position of $(\ty_{r-1})_\n$ and $(\ty_{\ell-1})_{\m_\ell}$ in the tree of irreducible $p$-types attached to $\ty$. By (D) of Theorem \ref{comparison}, we may use instead the tree $\mm_p(A_r)$ of maximal ideals of the inductive artinian algebra $A_r$ containing $p$, which is easier to handle.

For any $\n\in\mx_p(A_r)$ we define the intersection index $i:=i(\m,\n)$
as the maximal index $0\le i\le r$ with $\m_i=\n_i$. Then, \cite[Prop. 4.7]{newapp} shows that
$$
w_\p(\phi_\n(\t))=
\begin{cases}
0, &\mbox{ if }i=0,\\
(m_{\n,r}/m_{\m,i})(V_{\m,i}+\rho\la_i)/(e_1\cdots e_{i-1}), &\mbox{ if }0<i<\ell,\\
(m_{\n,r}/m_{\m,\ell})(V_{\m,\ell}+\rho\,\delta)/(e_1\cdots e_{\ell-1}), &\mbox{ if }i\ge \ell.
\end{cases}
$$

For $0\le j\le k\le r$, denote:
$$
\begin{array}{l}
\ll_{j,k}(\m)=\{\n\in\mx_p(A_k)\mid i(\m,\n)=j\},\\
\mm_{j,k}(\m)=\{\n\in\mx_p(A_k)\mid \m_j=\n_j\}=\bigcup_{i=j}^k\ll_{i,k}(\m).
\end{array}
$$
We claim that
\begin{equation}\label{summ}
\sum\nolimits_{\n\in\mm_{j,k}(\m)}m_{\n,k}/m_{\m,j}=m_k/m_j.
\end{equation}
In fact, if we sum
$m_{\n,k}/m_{\m,j}$ over all $\n\in\mx_p(A_k)$ having the same previous node $\n'\in\mx_p(A_{k-1})$, we get $e_{k-1}f_{k-1}m_{\n',k-1}$ by (\ref{branches}). Hence,
$$
\sum\nolimits_{\n\in\mm_{j,k}(\m)}m_{\n,k}/m_{\m,j}=e_{k-1}f_{k-1}\sum\nolimits_{\n'\in\mm_{j,k-1}(\m)}m_{\n',k-1}/m_{\m,j}.
$$
An iteration of this argument proves (\ref{summ}), since $\mm_{j,j}(\m)=\{\m_j\}$.

By (\ref{sum}) and the explicit formulas for $w_\p(\phi_\n(\t))$, we have:
$$
w_\p(g_r(\t))=\sum_{k=1}^{\ell-1}\sum_{\n\in\ll_{k,r}(\m)}\dfrac{m_{\n,r}(V_{\m,k}+\rho\la_k)}{m_{\m,k}\,e_1\cdots e_{k-1}}+\sum_{k=\ell}^r\sum_{\n\in\ll_{k,r}(\m)}\dfrac{m_{\n,r}(V_{\m,\ell}+\rho\delta)}{m_{\m,\ell}\,e_1\cdots e_{\ell-1}}.
$$
By using $h_{\m,k}=\rho h_k$ and the identity (\ref{recurrence}) for $V_{\m,k}/e_1\cdots e_{k-1}$, we get
\begin{align*}
w_\p(g_r(\t))=&\;\sum_{k=1}^{\ell-1}\left(\sum_{\n\in\ll_{k,r}(\m)}\dfrac{m_{\n,r}}{m_{\m,k}}\right)\sum_{1\le j\le k}\dfrac{m_{\m,k}}{m_{\m,j}}\dfrac{\rho\,h_j}{\,e_1\cdots e_j}\\+&\;\sum_{k=\ell}^r\left(\sum_{\n\in\ll_{k,r}(\m)}\dfrac{m_{\n,r}}{m_{\m,\ell}}\right)\left(\sum_{1\le j< \ell}\dfrac{m_{\m,\ell}}{m_{\m,j}}\dfrac{\rho\,h_j}{e_1\cdots e_j}+\dfrac{\rho\delta}{e_1\cdots e_{\ell-1}}\right).
\end{align*}

If we group all terms involving $(\rho\,h_j)/(e_1\cdots e_j)$ for each $1\le j<\ell$, and we use (\ref{summ}), we see that $w_\p(g_r(\t))$ is equal to

\begin{align*}
\sum_{1\le j< \ell}&\left(\sum_{\n\in\mm_{j,r}(\m)}\dfrac{m_{\n,r}}{m_{\m,j}}\right)\dfrac{\rho\,h_j}{e_1\cdots e_j}+\left(\sum_{\n\in\mm_{\ell,r}(\m)}\dfrac{m_{\n,r}}{m_{\m,\ell}}\right)\dfrac{\rho\,\delta}{e_1\cdots e_{\ell-1}}\\
=&\;\sum_{1\le j< \ell}\dfrac{m_r}{m_j}\dfrac{\rho\,h_j}{\,e_1\cdots e_j}+\dfrac{m_r}{m_\ell}\,\dfrac{\rho\,\delta}{e_1\cdots e_{\ell-1}}=\dfrac{\rho\,m_r}{m_\ell}\,
\dfrac{V_\ell+\delta}{e_1\cdots e_{\ell-1}},
\end{align*}
the last equality again by (\ref{recurrence}).
\end{proof}

\subsection{Computation of the residue fields attached to prime ideals}\label{subsecRatfs}
Let $A=A_0[z_0,\dots,z_r]$ be the artinian algebra associated with $\ty$, and denote by $\rdp\colon A\rightarrow A/pA$ the homomorphism of reduction modulo $p$.

The isomorphism (\ref{local}) induces an isomorphism of $(\Z/p\Z)$-algebras:$$\rd_p(A)=A/pA\simeq \prod\nolimits_{\m\in\mx_p(A)}\fm.$$

In this section, we relate this algebra $\rd_p(A)$ with the residue fields $\F_\p$ of the prime ideals $\p\in\pp_\ty$.
To this end, we introduce some rational functions in $\Q(x)$.
We agree that $g_0=x$, $\pi_0=N$, and we define
\begin{equation}\label{ratfs}
\phi_i=g_i\,\pi_i^{-V_i},\quad \ga_i=\phi_i^{e_i}\pi_i^{-h_i},\quad \pi_{i+1}=\phi_i^{\ell_i}\pi_i^{\ell'_i},\quad 0\le i\le r.
\end{equation}
For $i>0$, let $h=\phi_i,\ga_i$, or $\pi_{i+1}$. We may express $h=N^{n_0}\,g_1^{n_1}\dots g_i^{n_i}\in \ss(v_r)^{-1}\oo[x]$, for certain exponents $n_0,\dots,n_r\in\Z$. In particular,
\begin{equation}\label{wrphii}
v_r(h)=e_{i+1}\cdots e_r\,v_i(h),\quad w_\p(h(\t))=\rho\,v_r(h)/(e_1\cdots e_r),
\end{equation}
since this holds for the $\ty$-robust polynomials $N,g_1,\dots,g_i$. The second identity follows from Proposition \ref{value}, because $R_r(g_j)$ is a unit in $A_r$ for $j\le r$.

\begin{lemma}\label{wrphir}
\begin{enumerate}
\item $v_r(\phi_r)=h_r$,\quad$v_r(\ga_r)=0$,\quad$v_r(\pi_{r+1})=1$.
\item For all $\p\in\pt$ we have
$$w_\p(\phi_r(\t))=\rho\,h_r/e_1\cdots e_r,\quad w_\p(\ga_r(\t))=0,\quad w_\p(\pi_{r+1}(\t))=\rho/e_1\cdots e_r.
$$
\end{enumerate}
\end{lemma}

\begin{proof}
Let us prove the equalities of (1) simultaneously by induction on $r$. For $r=0$ they amount to $v_0(x)=0$ and $v_0(N)=1$, because $\phi_0=\ga_0=x$ and $\pi_1=N$.
Suppose that (1) holds for all $i<r$. Then by (\ref{wrphii}), we have
\begin{equation}\label{vrpir}
v_r(\pi_r)=e_r\,v_{r-1}(\pi_r)=e_r.
\end{equation}
Hence, by the recurrent definition of the functions,
$$
\begin{array}{l}
v_r(\phi_r)=v_r(g_r)-V_rv_r(\pi_r)=e_r(V_r+\la_r)-V_re_r=e_r\la_r=h_r,\\
v_r(\ga_r)=e_rv_r(\phi_r)-h_rv_r(\pi_r)=0,\\
v_r(\pi_{r+1})=\ell_rv_r(\phi_r)+\ell'_rv_r(\pi_r)=\ell_rh_r+\ell'_re_r=1.
\end{array}
$$

Item (2) follows from (\ref{wrphii}) and item (1).
\end{proof}

\begin{lemma}\label{gammas}
Consider a rational function  $h=N^{n_0}g_1^{n_1}\cdots g_r^{n_r}\in \Q(x)$ with $v_r(h)=0$. Then, there exist integers $a_1,\dots,a_r$ such that $h=\ga_1^{a_1}\cdots\ga_r^{a_r}$.
\end{lemma}

\begin{proof}
By induction on $r$. For $r=0$ the statement is obvious because necessarily $h=1$. For $r>0$, we have
$$
v_r\left(N^{n_0}g_1^{n_1}\cdots g_{r-1}^{n_{r-1}}\right)\equiv0\md{e_r},\qquad v_r(g_r)\equiv h_r\md{e_r}.
$$
Hence, the condition $v_r(h)=0$ implies $n_r=e_r m$ for some $m\in\Z$. Now, $h'=h\ga_r^{-m}=N^{n_0}g_1^{n_1}\cdots g_{r-1}^{n_{r-1}}\pi_r^{m(e_rV_r+h_r)}$ satisfies $h'=N^{\nu_0}g_1^{\nu_1}\cdots g_{r-1}^{\nu_{r-1}}$ for some integers $\nu_0,\dots,\nu_{r-1}$, and  $v_r(h')=v_r(h)-mv_r(\ga_r)=0$, by Lemma \ref{wrphir}. Hence, $h'=\ga_1^{a_1}\cdots\ga_{r-1}^{a_{r-1}}$, by the induction hypothesis.
\end{proof}

\begin{theorem}\label{AFp}\mbox{\null}

{\bf (A)} \ There is an injective homomorphism of $(\Z/p\Z)$-algebras:
$$
\gt\colon \rdp(A)\hookrightarrow\prod\nolimits_{\p\in\pt}\F_\p,\quad x\mapsto \ga_\ty(x)=\left(\ga_{\ty,\p}(x)\right)_{\p\in\pt}
$$
determined by $\ga_{\ty,\p}(\rdp(z_i))=\ga_i(\t)+\p$ for all $0\le i\le r$.\medskip

{\bf (B)} \ For all $\m\in\mpa$, $\p\in\ptm$ we have  $\ga_{\ty,\p}\circ\rdp=\ga_{\ty_\m,\p}\circ\rdm$.

Moreover, if $\phi_{\m,r}$, $\ga_{\m,r}$, $\pi_{\m,r+1}\in\Q_p(x)$ are the rational functions associated with $\ty_\m$ by the recurrent formulas of (\ref{ratfs}), we have
\begin{equation}\label{compratfs}
\begin{array}{rcl}
(\phi_r/\phi_{\m,r})(\t)+\p&=&\ga_{\ty,\p}\left(\xi_r\chi_{r-1}^{-V_r}\right),\\
(\ga_r/\ga_{\m,r})(\t)+\p&=&\ga_{\ty,\p}\left(\sigma_r\right),\\
(\pi_{r+1}/\pi_{\m,r+1}^\rho)(\t)+\p&=&\ga_{\ty,\p}\left(\chi_r\right),
\end{array}
\end{equation}
where $\xi_i,\chi_i,\sigma_i\in\F_{\m,i}\subset \F_\m\subset \rdp(A)$ are defined in Theorem \ref{comparison}.
\medskip

{\bf (C)} \ Let $h\in\Z[x]$. If $r=0$, take $s_0=0$, $u_0=v_0(h)$. If $r>0$, take $(s_0,u_0)=(s_r(h),u_r(h))$, the left endpoint of $S_r(h)$. Then, for each $\p\in\pt$,
\begin{equation}\label{C}
h(\t)\phi_r(\t)^{-s_0}\pi_r(\t)^{-u_0}+\p=\ga_{\ty,\p}\left(\rdp(R_r(h)(z_r))\right).
\end{equation}
\end{theorem}

\begin{proof}
For an irreducible type, statement (B) is trivial and statements (A), (C) were proved in \cite[Sec. 3.1]{GMN}.

Thus, for $r\ge1$, $\m\in\mpa$ and $\p\in\ptm$, we may apply (C) to the polynomial $G_{\m,r}=g_r/\phi_{\m,r}\in\Z_p[x]$ and the type $\ty_{\m,r-1}:=(\ty_\m)_{r-1}$.

Denote $u=(\rho V_r-V_{\m,r})/e_{r-1}$. For $r>1$, Corollary \ref{suGr} shows that $s_{r-1}(G_{\m,r})=0$ and $u_{r-1}(G_{\m,r})=u$.
For the artinian algebra associated with a $p$-type, the mapping $\rdp$ is the identity and (C) yields:
\begin{equation}\label{Gmr}
\begin{array}{rcl}
G_{\m,r}(\t)\pi_{\m,r-1}(\t)^{-u}+\p&=&\ga_{\ty_{\m,r-1},\p}\left(R_{\m,r-1}(G_{\m,r})(z_{\m,r-1}))\right)\\&=&\;\ga_{\ty_\m,\p}\left(\xi_rz_{\m,r-1}^{\ell_{\m,r-1}u}\right),
\end{array}
\end{equation}
the last equality following from $\ga_{\ty_{\m,r-1}}=(\ga_{\ty_\m})_{\mid \F_{\m,r}}$ and the definition of the constant $\xi_r\in \F_{\m,r}^*$ in Theorem \ref{comparison}.

Note that (\ref{Gmr}) holds for $r=1$ too, because $u=0=v_0(G_{\m,1})$.

Let us prove the theorem by induction on $r$.

For $r=0$ we have $\ty=(t_0)$ and $A=A_0[y]/(t_0)=A_0[z_0]$.
Each monic irreducible factor $\psi$ of $\rdp(t_0)$ in $(\Z/p\Z)[y]$ determines a maximal ideal $\m=(\m_0,\psi)\in\mpa$ with $\F_\m=(\Z/p\Z)[y]/(\psi)$ and $\ty_\m=(\psi)$.

For every $\p\in\ptm$ the polynomial $F_\p$ is congruent to a power of $\psi$ modulo $p$. Hence,  the field  $\F_\m$ is embedded into $\F_\p$ by sending the class of $y$ modulo $\psi$ to $\t+\p=\ga_0(\t)+\p$. Since $\rdp(A)\simeq\prod_{\m\in\mpa}\F_\m$ and the sets $\ptm$ are pairwise disjoint, this proves (A).

Since $\phi_0=\phi_{\m,0}=\ga_0=\ga_{\m,0}=x$ and $\pi_1=N$, $\pi_{\m,1}=p$, statement (B) follows immediately. Although $\chi_{-1}$ is not defined, we agree that  $\chi_{-1}^0=1$.

Finally, the two sides of (\ref{C}) coincide with the element in $\F_\p$ obtained by reducing modulo $p$ the polynomial $h(x)/N^{v_0(h)}\in\Z[x]$ and then replacing $x$ with $\t+\p$. This ends the proof of the theorem in the case $r=0$.

Suppose that $r>0$ and the theorem holds for types of order less than $r$.

As a consequence of Theorem \ref{comparison} and our general assumptions:
$$
e_{\m,i}=e_i,\quad h_{\m,i}=\rho h_i,\quad \ell_{\m,i}=\ell_i,\quad  \ell'_{\m,i}=\ell'_i,\quad 1\le i\le r.
$$

 Let us first prove (B), whose first statement contains an independent definition of $\ga_\ty$. For any $\p\in\pt$ and any $a\in A$, let us \emph{define} $\ga_{\ty,\p}(\rdp(a))$ to be $\ga_{\ty_\m,\p}(\rdm(a))$ for the unique $\m \in \mpa$ for which $\p\in \ptm$.

For the proof of (\ref{compratfs}), we abuse of language and omit evaluation at $\t$ and classes modulo $\p$. From
 $ \phi_r=g_r\pi_r^{-V_r}$ and $\phi_{\m,r}=\phi_{\m,_r}\pi_{\m,r}^{-V_{\m,r}}$,
 we deduce
 $$
 \phi_r/\phi_{\m,r}=G_{\m,r}\left(\pi_r/\pi_{\m,r}^\rho\right)^{-V_r}\pi_{\m,r}^{V_{\m,r}-\rho V_r}.
 $$
 If $r=1$, then $u=0=V_1$, and (\ref{Gmr}) yields $ \phi_r/\phi_{\m,r}=\ga_{\ty,\p}(\xi_1)$, as desired.

 If $r>1$, then (\ref{Gmr}) and the induction hypothesis lead to
 \begin{align*}
\phi_r/\phi_{\m,r}=&\;G_{\m,r}\ga_{\ty,\p}\left(\chi_{r-1}^{-V_r}\right)\pi_{\m,r}^{V_{\m,r}-\rho V_r}\\
=&\;\ga_{\ty,\p}\left(\xi_rz_{\m,r-1}^{\ell_{r-1}u}\right)\ga_{\ty,\p}\left(\chi_{r-1}^{-V_r}\right)\pi_{\m,r-1}^{u}\pi_{\m,r}^{-e_{r-1}u}\\=&\;\ga_{\ty,\p}\left(\xi_r\chi_{r-1}^{-V_r}\right)\left(\ga_{\m,r-1}^{\ell_{r-1}}\pi_{\m,r-1}\pi_{\m,r}^{-e_{r-1}}\right)^{u}=\ga_{\ty,\p}\left(\xi_r\chi_{r-1}^{-V_r}\right),
\end{align*}
because $\ga_{\m,r-1}^{\ell_{r-1}}\pi_{\m,r-1}\pi_{\m,r}^{-e_{r-1}}=1$ in $\Q_p(x)$ by a direct application of (\ref{ratfs}).

This ends the proof of the first identity in (\ref{compratfs}). The other two identities follow easily from the first one and the induction hypothesis.

Let us prove (A). Since $\ga_{\ty_\m}(z_{\m,i})=\ga_{\m,i}(\t)+\p$, for all $0\le i\le r$, the homomorphism $\ga_\ty$ defined in (B) satisfies:
\begin{align*}
\ga_{\ty,\p}(\rdp(z_i))=&\;\ga_{\ty_\m,\p}(\rdm(z_i))=\ga_{\ty_\m,\p}(\sigma_iz_{\m,i})\\=&\;\ga_{\ty_\m,\p}(\sigma_i)\left[\ga_{\m,i}(\t)+\p\right]=
\ga_i(\t)+\p=\ga_i(\t)+\p,
\end{align*}
by Theorem \ref{comparison} and the second identity in (\ref{compratfs}). Hence, the mapping $\ga_\ty$ defined in (A) is well defined and coincides with the mapping $\ga_\ty$ defined in (B). Finally, for any $\m\in\mpa$, the homomorphism $\ga_\ty$ embeds the field $\F_\m$ diagonally into $\prod_{\p\in\ptm}\F_\p$. Since the sets $\ptm$ are pairwise disjoint, $\ga_\ty$ is injective. This ends the proof of (A).

Let us prove (C). By Lemma \ref{wrphir} and (\ref{vrpir}), for any $\p\in\pt$ we have
$$
v_r\left(\phi_r^{s_0}\pi_r^{u_0}\right)=s_0h_r+u_0e_r=v_r(h),\quad
w_\p\left(\phi_r^{s_0}(\t)\pi_r^{u_0}(\t)\right)=\rho \,v_r(h)/(e_1\cdots e_r).
$$

Let $d=\deg R_r(h)$ and denote $s_j=s_0+je_r$ for $0\le j\le d$. Let $h=\sum_{0\le s}a_sg_r^s$ be the canonical $g_r$-expansion of $h$ and consider $h_0=\sum_{j=0}^da_{s_j}g_r^{s_j}$. For any integer abscissa $s\ne s_j$ we have $v_r(a_sg_r^s)>v_r(h)$; hence,
$$
w_\p(a_s(\t)g_r(\t)^s)>\rho\,v_r(h)(e_1\cdots e_r)=w_\p\left(\phi_r^{s_0}(\t)\pi_r^{u_0}(\t)\right)
$$
by Proposition \ref{value}. Therefore,
\begin{equation}\label{init}
h(\t)\phi_r(\t)^{-s_0}\pi_r(\t)^{-u_0}+\p=h_0(\t)\phi_r(\t)^{-s_0}\pi_r(\t)^{-u_0}+\p.
\end{equation}

Now, for any $0\le j\le d$, we have the following identity in $\Q(x)$, which is easy to deduce from the recurrent formulas of (\ref{ratfs}) (cf. \cite[Lem. 3.4]{GMN}):
$$
a_{s_j}\,g_r^{s_j}\phi_r^{-s_0}\pi_r^{-u_0}=\ga_{r-1}^{\nu_{r-1}(a_{s_j})}a_{s_j}\,\phi_{r-1}^{-s_{r-1}(a_{s_j})}\pi_{r-1}^{-u_{r-1}(a_{s_j})}\,\ga_r^j,
$$

If we evaluate the left-hand side at $\t$, take classes modulo $\p$  and sum over $0\le j\le d$, we get the left-hand side of (\ref{C}), thanks to (\ref{init}). If we do the same operation with the right-hand side, we get the right-hand side of (\ref{C}), thanks to the induction hypothesis and the definition of $R_r(h)$.
\end{proof}

\begin{corollary}\label{epfp}
If $\ord_\ty(f)=1$ and $R_r(f)=t_rq$ with $\gcd_{A_r}(q,t_r)=1$, then:
\begin{enumerate}
\item $e_\p=e_1\cdots e_r$ for all $\p\in\pt$, \ and \quad $\sum_{\p\in\pt}f_\p=f_0f_1\cdots f_r$.
\item $\gt$ is an isomorphism.
\end{enumerate}
\end{corollary}

\begin{proof}
By Corollary \ref{squarefree}, $\ord_{\ty_\m}(f)=1$ for all $\m\in\mpa$. This implies that every $\ptm$ is a one-element set, say  $\ptm=\{\p_\m\}$. By \cite[Cor. 3.8]{GMN},
$$
e_{\p_\m}=e_{\m,1}\cdots e_{\m,r}=e_1\cdots e_r,\quad
f_{\p_\m}=f_{\m,0}f_{\m,1}\cdots f_{\m,r}.
$$
By a recurrent application of (\ref{branches}),
$$
\sum_{\p\in\pt}f_\p=\sum_{\m\in\mpa}f_{\m,0}\cdots f_{\m,r}=f_r\sum_{\n\in\mx_p(A_r)}f_{\n,0}\cdots f_{\n,r-1}=f_r\cdots f_0.
$$
In particular, $\dim_{\Z/p\Z}(\rdp(A))=f_0\cdots f_r=\dim_{\Z/p\Z}\left(\oplus_{\p\in\pt}\F_\p\right)$. Since $\ga_\ty$ is injective, it is an isomorphism.
\end{proof}

\section{Computation of integral bases}\label{secGlobalBasis}
We keep the notation and assumptions from the last section.

 \subsection{Reduced $p$-integral bases}\label{subsecRed}
Let $\zp$ be the local ring  of $\Z$ at the prime ideal $p\Z$. Let $\zkp\subset K$ be the integral closure of $\zp$ in the number field $K$. This ring is a free $\zp$-module of rank $n$. A $\zp$-basis of $\zkp$ is called a \emph{$p$-integral basis} of $K$.

Consider the following pseudo-valuation extending $\ord_p$ to $K$:
$$
w:=w_p\colon K\lra e^{-1}\Z\cup\{\infty\},\quad w(\alpha)=\mn\nolimits_{\p\in\pp}\{w_\p(\alpha)\},
$$
where $e=\op{lcm}\left\{e_\p\mid\p\in\pp\right\}$.
The $p$-integral elements $\alpha\in\zkp$ are characterized by the condition $w(\alpha)\ge0$. Clearly $\Q^*\subset \ss_w$ (cf. Definition \ref{qv}).

\begin{definition}
A subset $\bb=\{\alpha_1,\dots,\alpha_m\}\subset K$ is called \emph{$p$-reduced} if for all families $a_1,\dots,a_m\in \zp$, one has:
\begin{equation}\label{reduceness}
w\left(\sum\nolimits_{1\le i\le m}a_i\alpha_i\right)=\mn\{w(a_i\alpha_i)\mid 1\le i\le m\}.
\end{equation}
\end{definition}


\begin{lemma}\cite[Lem. 5.6]{bases}\label{reducedbasis}
A $p$-reduced set $\bb=\{\alpha_1,\dots,\alpha_n\}\subset K$ such that $w(\bb)\subset [0,1)$ is a $p$-integral basis of $K$.
\end{lemma}

The reduceness criterion of Theorem \ref{criterion} below plays an essential role.

\begin{definition}
For each $\p\in\pp$ let us fix some $\pi_\p\in\Z_\p$ with $w_\p(\pi_\p)=1/e_\p$.
For any value $\delta\in w(K)$, consider the $\zp$-modules: $$K_\delta=\{\alpha\in K\mid w(\alpha)\ge \delta\}\supset K_\delta^+=\{\alpha\in K\mid w(\alpha)>\delta\},$$
and the following homomorphism of $\zp$-modules with kernel $K_\delta^+$:
$$
\rd_\delta\colon K_\delta\lra V=\prod\nolimits_{\p\in\pp}\F_\p,\quad \rd_\delta(\alpha)=\left(\alpha\pi_\p^{-\lfloor e_\p\delta\rfloor}+\p\Z_\p\right)_{\p\in\pp}.
$$
\end{definition}

\begin{theorem}\cite[Lem. 5.7]{bases}\label{criterion}
Let $\bb\subset K$ with  $w(\bb)\subset [0,1)$. Then, $\bb$ is $p$-reduced if and only if $\rd_\delta(\bb_\delta)\subset V$ is a $(\Z/p\Z)$-linearly independent family for all $\delta\in w(\bb)$, where $\bb_\delta=\{\alpha\in \bb\mid w(\alpha)=\delta\}$.
\end{theorem}

Clearly, the following diagram commutes:
$$
\begin{array}{ccc}
K_\delta\times K_\epsilon&\stackrel{\rd_\delta\times \rd_\epsilon}\lra&V\times V\\
\downarrow&&\downarrow\\
K_{\delta+\epsilon}&\stackrel{\rd_{\delta+\epsilon}}\lra&V
\end{array}
$$
where the vertical mappings are ordinary multiplication.

\subsection{Quotients of $g_r$-expansions}\label{subsecQuots}

\begin{definition}
Let $f=a_0+a_1g_r+\cdots+a_mg_r^m$ be the canonical $g_r$-expansion of $f$.
The \emph{$g_r$-quotients of $f$} are the quotients  $q_1,\dots,q_m\in \Z[x]$ of the divisions with remainder involved in the computation of the coefficients $a_s$:
$$
f=g_r\,q_1+a_0,\quad q_1=g_r\,q_2+a_1,\quad \cdots,\quad q_m=g_r\cdot 0+a_m=a_m.
$$
\end{definition}

Clearly, for any $1\le s\le m$, the canonical $g_r$-expansion of $q_s$ is:
\begin{equation}\label{residue}
 q_s=a_s+a_{s+1}g_r+\cdots+a_mg_r^{m-s}.
\end{equation}
In particular, if $f$ is  $\ty$-robust, all its $g_r$-quotients are $\ty$-robust.


\begin{lemma}\label{Rquot}
Let $R_r(f)=c_{s_0}+c_{s_1}y+\cdots+c_{s_d}y^d\in A_r[y]$, where $s_0\le s_d$ are the abscissas of the endpoints of $S_r(f)$ and $s_j=s_0+je_r$ for $0\le j\le d$.

Take $\ell=\mn\{j\mid s_j \ge s, \ c_{s_j}\ne0 \}$. Then,
$$R_r(q_s)=c_{s_\ell}+c_{s_{\ell+1}}y+\cdots+c_{s_d}y^{d-\ell}.$$
\end{lemma}

\begin{proof}
By (\ref{residue}), the $g_r$-expansions of $q_sg_r^s$ and $f$ coincide except for the first $s$ coefficients. Hence, in the region $[s,\infty)\times \R$ both polynomials provide the same cloud of points $(s,v_{r-1}(a_sg_r^s))$. Thus, the Newton polygons $N_r(f)$ and $N_r(q_sg_r^s)$ coincide in the region $[s_\ell,\infty)\times \R$ (cf. Figure \ref{figSplit}).

In particular, the $\la_r$-component $S_r(q_sg_r^s)$ is contained in $S_r(f)$ and has endpoints with abscissas $s_\ell\le s_d$. Also, the residual coefficients of $f$ and $q_sg_r^s$ coincide for all integer abscissas $s_j\ge s_\ell$. Since $R_r(q_s)=R_r(q_sg_r^s)$, this ends the proof of the lemma.
\end{proof}

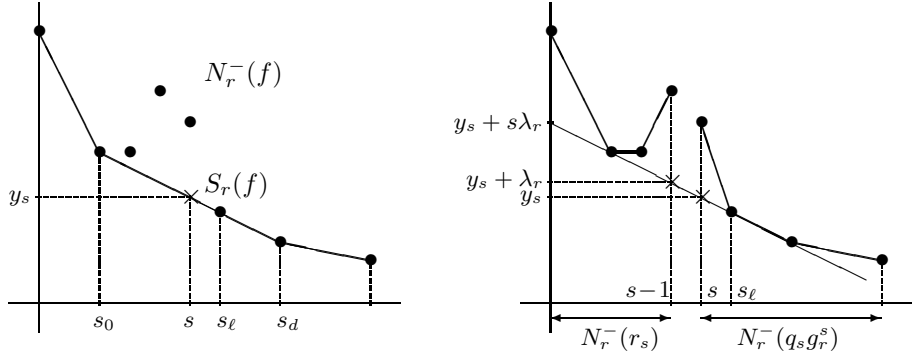
\begin{figure}\caption{Newton polygons of $f$, $q_sg_r^s$ and $r_s=f-q_sg_r^s$}\label{figSplit}
\setlength{\unitlength}{4.mm}
\begin{picture}(28,12)
\put(-.2,9.8){$\bullet$}\put(1.8,5.8){$\bullet$}
\put(5.8,3.8){$\bullet$}\put(7.8,2.8){$\bullet$}\put(10.8,2.2){$\bullet$}
\put(3.8,7.8){$\bullet$}\put(2.8,5.8){$\bullet$}\put(4.8,6.8){$\bullet$}
\put(0,0){\line(0,1){11}}\put(-1,1){\line(1,0){13}}
\put(2,6.03){\line(-1,2){2}}\put(2,6){\line(-1,2){2}}
\put(8,3){\line(-2,1){6}}\put(8,3.03){\line(-2,1){6}}
\put(11,2.4){\line(-5,1){3}}\put(11,2.43){\line(-5,1){3}}
\put(4.7,4.3){$\times$}
\put(5.5,8.4){\begin{small}$N_{r}^-(f)$\end{small}}
\put(5.5,4.7){\begin{small}$S_r(f)$\end{small}}
\multiput(2,.9)(0,.25){20}{\vrule height2pt}
\multiput(6,.9)(0,.25){13}{\vrule height2pt}
\multiput(8,.9)(0,.25){9}{\vrule height2pt}
\multiput(5,.9)(0,.25){15}{\vrule height2pt}
\multiput(11,.9)(0,.25){7}{\vrule height2pt}
\multiput(-.1,4.5)(.25,0){21}{\hbox to 2pt{\hrulefill }}
\put(-1,4.4){\begin{footnotesize}$y_s$\end{footnotesize}}
\put(1.8,.2){\begin{footnotesize}$s_0$\end{footnotesize}}
\put(4.8,.2){\begin{footnotesize}$s$\end{footnotesize}}
\put(5.9,.2){\begin{footnotesize}$s_\ell$\end{footnotesize}}
\put(7.9,.2){\begin{footnotesize}$s_d$\end{footnotesize}}

\put(16.8,9.8){$\bullet$}\put(18.8,5.8){$\bullet$}\put(22.8,3.8){$\bullet$}
\put(20.8,7.8){$\bullet$}\put(19.8,5.8){$\bullet$}\put(21.8,6.8){$\bullet$}
\put(24.8,2.8){$\bullet$}\put(27.8,2.2){$\bullet$}
\put(17,0){\line(0,1){11}}\put(16,1){\line(1,0){13}}
\put(19,6.03){\line(-1,2){2}}\put(19,6){\line(-1,2){2}}
\put(28,2.4){\line(-5,1){3}}\put(28,2.43){\line(-5,1){3}}
\put(19,6){\line(1,0){1}}\put(19,6.03){\line(1,0){1}}
\put(20,6){\line(1,2){1}}\put(20,6.03){\line(1,2){1}}
\put(23,4){\line(-1,3){1}}\put(23,4.03){\line(-1,3){1}}
\put(23,4){\line(2,-1){2}}\put(23,4.03){\line(2,-1){2}}
\put(20.7,4.8){$\times$}\put(21.7,4.3){$\times$}
\multiput(21,.9)(0,.25){28}{\vrule height2pt}
\multiput(22,.9)(0,.25){25}{\vrule height2pt}
\multiput(28,.9)(0,.25){7}{\vrule height2pt}
\multiput(23,.9)(0,.25){13}{\vrule height2pt}
\put(19.5,1.2){\begin{footnotesize}$s\!-\!1$\end{footnotesize}}
\put(22.2,1.2){\begin{footnotesize}$s$\end{footnotesize}}
\put(23.2,1.2){\begin{footnotesize}$s_\ell$\end{footnotesize}}
\multiput(16.9,4.5)(.25,0){21}{\hbox to 2pt{\hrulefill }}
\multiput(16.9,5)(.25,0){17}{\hbox to 2pt{\hrulefill }}
\put(16,4.4){\begin{footnotesize}$y_s$\end{footnotesize}}
\put(14.2,4.9){\begin{footnotesize}$y_s+\la_{r}$\end{footnotesize}}
\put(19,.5){\vector(-1,0){2}}\put(19,.5){\vector(1,0){2}}
\put(24,.5){\vector(-1,0){2}}\put(24,.5){\vector(1,0){4}}
\put(18,-.4){\begin{footnotesize}$N_{r}^-(r_s)$\end{footnotesize}}
\put(23.2,-.4){\begin{footnotesize}$N_{r}^-(q_sg_{r}^s)$\end{footnotesize}}
\put(17,7){\line(2,-1){10.5}}
\put(13.8,6.8){\begin{footnotesize}$y_s+s\la_{r}$\end{footnotesize}}
\put(16.9,7){\line(1,0){.2}}
\end{picture}
\end{figure}

\begin{theorem}\label{denquot}
Suppose that $\ty_{r-1}\mid f$ and for each slope $-\la$ of $N_r^-(f)$ the least positive denominator $e_\la$ of $\la$ satisfies  $\gcd(\rho,e_\la)=1$. Also, suppose that Algorithm \ref{sfdA} outputs a squarefree decomposition of $R_{v_{r-1},g_r,\la}(f)$ in $A_r[y]$ with strongly unitary squarefree factors.

Suppose that $-\la_r$ is one of the slopes of $N_r^-(f)$, and let $s_0< s_d$ be the abscissas of the endpoints of $S_r(f)$. For any integer
$s_0< s\le s_d$, let $q_s$ be the $s$-th $g_{r}$-quotient of $f$, and denote $H_s=(y_s-sV_{r})/(e_0\cdots e_{r-1})$, where $y_s\in\Q$ is determined by the condition $(s,y_s)\in S_r(f)$. Then,
$$
w(q_s(\t))\ge \rho\, v_r(q_s)/(e_0\cdots e_r)=\rho\,H_s.
$$
\end{theorem}

\begin{proof}
From $S_r(q_sg_r^s)\subset S_r(f)$ we deduce $v_r(q_sg_r^s)=v_r(f)=e_r(y_s+s\la_r)$, as Figure \ref{figSplit} shows. Since $v_r(g_r)=e_r(V_r+\la_r)$, we get
$$
v_r(q_s)=e_r(y_s-sV_r)=e_0\cdots e_r\,H_s.
$$


We want to check that $w_\p(q_s(\t))\ge\rho\,H_s$ for all $\p\in\pp$. The proof mimics that of \cite[Thm. 3.3]{bases}. Denote $e=e_0\cdots e_{r-1}$.\medskip

\noindent{\bf Case $\p\in\pp_{\ty_{r-1}}$.}
By Proposition \ref{splitting}, $\p\in \pp_{\ty_{\mu,t}}$ for some slope $-\mu$ of $N_r(f)$ and squarefree factor $t$ of $R_{v_{r-1},g_r,\mu}(f)$.

Let $v_{\mu,r}$ be the $r$-th pseudo-valuation of $\ty_{\mu,t}$.
For any polynomial $h\in \Z[x]$, $v_{\mu,r}(h)/e_\mu$ is the ordinate of the intersection point of the vertical axis with the line of slope $-\mu$ first touching $N_r(h)$ from below.  Hence, a look at Figure \ref{figSplit} justifies the following arguments.

If $\mu\ge\la_r$, Proposition \ref{value} applied to $q_sg_r^s$ yields:
$$
w_\p(q_s(\t)g_r(\t)^s)\ge \rho\, v_{\mu,r}(q_sg_r^s)/(ee_\mu)\ge \rho(y_s+s\mu)/e,
$$
so that $w_\p(q_s(\t))\ge (\rho(y_s+s\mu)/e) - sw_\p(g_r(\t))=\rho\,H_s$, by Corollary \ref{previous}.

If $\mu<\la_r$, Proposition \ref{value} applied to $r_s=f-q_sg_r^s$ yields:
$$
w_\p(r_s(\t))\ge \rho\, v_{\mu,r}(r_s)/(ee_\mu)\ge \rho\,(y_s+\la_r+(s-1)\mu)/e>\rho\,(y_s+s\mu)/e.
$$

Since $q_s(\t)g_r(\t)^s=-r_s(\t)$, we deduce $w_\p(q_s(\t))> \rho H_s$ in this case.\medskip

\noindent{\bf Case $\p\not\in\pp_{\ty_{r-1}}$. }Take the minimal index $0\le \ell< r$ for which $\p\not\in\pp_{\ty_{\ell}}$.

Since $-q_s(\t)=r_s(\t)g_r(\t)^{-s}=\sum_{0\le i<s}a_i(\t)g_r(\t)^{i-s}$, it suffices to show:
\begin{equation}\label{subaim}
w_\p(a_i(\t))-w_\p\left(g_r(\t)^{s-i}\right)> \rho\,H_s,\quad 0\le i<s.
\end{equation}

Take $g_0=x$. For any $0\le i<s$, consider the multiadic expansion in $\Z[x]$:
$$
a_i=\sum\nolimits_{\j\in J} b_\j\, G^\j,\quad \deg(b_\j)<\deg(g_\ell),\quad G^\j=g_{\ell}^{j_\ell}\cdots g_{r-1}^{j_{r-1}},
$$
where $J=\left\{(j_\ell,\dots,j_{r-1})\mid 0\le j_k<e_kf_k \mbox{ for all }\ell\le k<r\right\}$.
An iterative application of (\ref{posteriori}) yields $v_{r-1}(a_i)=\mn\{v_{r-1}\left(b_\j\, G^\j\right)\mid \j\in J\}$.

There is an index $\j\in J$ such that $w_\p(a_i(\t)\ge w_\p(b_\j(\t)G(\t)^\j)$. If $\ell>0$, we have $\p\in\pp_{\ty_{\ell-1}}$, and Proposition \ref{value} shows that
\begin{align*}
w_\p(b_\j(\t))\ge&\; \rho\,v_{\ell-1}(b_\j)/(e_0\cdots e_{\ell-1})=
\rho\,v_{r-1}(b_\j)/e\\=&\;\rho(v_{r-1}(b_\j\,G^\j)-\rho\,v_{r-1}(G^\j))/e\ge\rho\,v_{r-1}(a_i)/e-\rho\,v_{r-1}(G^\j)/e.
\end{align*}
If $\ell=0$, $w_\p(b_\j(\t))\ge \rho\,v_0(b_\j)=\rho\,v_{r-1}(b_\j)/e$, and the same inequalities hold.

Finally, by the convexity of the Newton polygon:
$$
v_{r-1}(a_ig_r^i)\ge y_i\ge y_s+(s-i)\la_r.
$$
Hence, $v_{r-1}(a_i)\ge y_s-sV_r+(s-i)(V_r+\la_r)=y_s-sV_r+v_r(g_r^{s-i})/e_r$.

For any polynomial $h\in\Z[x]$ and any integer $k\ge0$, denote
$$\epsilon_k(h):=\rho\,v_k(h)/(e_0\cdots e_k)-w_\p(h(\t)).
$$The above inequalities yield:
\begin{equation}\label{subsubaim}
w_\p(a_i)-w_\p(g_r(\t)^{s-i})\ge \rho\,H_s-\epsilon_{r-1}(G^\j)+\epsilon_r(g_r^{s-i}).
\end{equation}


For any $\ell\le k<r$, we have $v_{r-1}(g_k)/e=(V_k+\la_k)/(e_0\cdots e_{k-1})$. Take $\delta=\mu$ if $\ell<k$, and $\delta=\mn\left\{\la_\ell,\mu\right\}$ if $\ell=k$. By using (\ref{recurrence}) and the explicit formulas of Lemma \ref{allps}, we obtain:
\begin{align*}
\rho^{-1}\epsilon_{r-1}(g_k)=&\;\dfrac{m_k}{m_\ell}\dfrac{\la_\ell-\delta}{e_0\cdots e_{\ell-1}}+\sum_{\ell< u\le k}\dfrac{m_k}{m_u}\dfrac{h_u}{e_0\cdots e_u}\\
\le&\;\dfrac{m_k}{m_\ell}\dfrac{\la_\ell-\mn\left\{\la_\ell,\mu\right\}}{e_0\cdots e_{\ell-1}}+\sum_{\ell< u\le k}\dfrac{m_k}{m_u}\dfrac{h_u}{e_0\cdots e_u}.
\end{align*}
$$
\rho^{-1}\epsilon_{r}(g_r)=\dfrac{m_r}{m_\ell}\dfrac{\la_\ell-\mn\left\{\la_\ell,\mu\right\}}{e_0\cdots e_{\ell-1}}+\sum_{\ell< u\le r}\dfrac{m_r}{m_u}\dfrac{h_u}{e_0\cdots e_u}>0.
$$
We may deduce $\epsilon_{r-1}(G^\j)< \epsilon_r(g_r)$ from the obvious inequality:
$$
j_um_u+j_{u+1}m_{u+1}+\cdots+j_{r-1}m_{r-1}<m_r,\quad \ell\le u<r.
$$
This proves (\ref{subaim}), having in mind (\ref{subsubaim}) and $\epsilon_{r}(g_r)>0$.
\end{proof}

\subsection{Computation of reduced local integral bases}\label{subsecLocalBases}

Suppose that, for the input $(N,f)$, the SF-OM algorithm of section \ref{secOM} does not detect a proper factor of $N$ and outputs a tree $\tcal(f)$ of SF-types.

\begin{definition}\label{terminalS}
Let $S$ be a side of a Newton polygon derived, along the execution of the SF-OM algorithm, from a type $\ty_{r-1}$ of order $r-1$, with representative $g_r$. Let $-\la_S$ be the slope of $S$ and let $R_{v_{r-1},g_r,\la_S}(f)=c\,T_1^{\ell_1}\cdots T_k^{\ell_k}$ be the squarefree factorization of the residual polynomial attached to $S$.

We say that $S$ is a \emph{terminal side of order $r$} if $\ell_1=1$. In this case, $S$ determines a leaf of the tree $\tcal(f)$:
$$
\ty_S=(\ty_{r-1};(g_r,\la_r,t_r)),\qquad \la_r:=\la_S,\quad t_r:=T_1.
$$

Let $\st$ be the set of all terminal sides provided by the OM algorithm.
\end{definition}

Let $S\in\st$ be a terminal side of order $r$.
For each $1\le i\le r$, let $d_i=\deg R_i(f)$ and denote by $s_{d_i}$ the abscissa of the right endpoint of $S_i(f)$. Moreover, for each $0\le j<s_{d_i}$ we
introduce the following notation:
$$
q_{i,j}=(s_{d_i}-j)-\mbox{th  $g_i$-quotient of $f$},\qquad
H_{i,j}=v_i\left(q_{i,j}\right)/e_0\cdots e_i.
$$


Let $J_S=\left\{(j_0,\dots,j_r)\in\N^{r+1}\mid 0\le j_i<e_if_i, \ 0\le i\le r\right\}$,
and define:
$$
\bb_S=\left\{q_\j\mid \j\in J_S\right\},\qquad q_\j=\t^{j_0} q_{1,j_1}(\t)\cdots q_{r,j_r}(\t)N^{-\lfloor H_{1,j_1}+\cdots +H_{r,j_r} \rfloor}.
$$

Also, we consider some more integral elements derived from the eventual leaf of $\tcal(f)$ of order zero.
If there is a squarefree factor $t\in A_0[y]$ dividing $\rd_N(f)$ with exponent one, the type of order zero $\ty^0=(t)$ is both a root node and a leaf of $\tcal(f)$. Choose a monic $g\in\Z[x]$ lifting $t$, and consider the division with remainder
$f=q\,g+a$, $\deg(a)<\deg (g)$.
Consider the set:
$$
\bb^0=\{q(\t),\t\, q(\t),\dots,\t^{\deg(t)-1}q(\t)\}.
$$

\begin{definition}\label{Nintegral}
A set $\bb\subset \Z_K$ is a (reduced) \emph{$N$-integral basis} of $K$ if it is a (reduced) $p$-integral basis simultaneously for all prime factors $p$ of $N$.
\end{definition}

The next theorem is the main result of the paper.

\begin{theorem}\label{pBasis}
Suppose that either $N$ is squarefree or all types in $\tcal(f)$ are unramified. Then,
the following set is a reduced $N$-integral basis of $K$:
$$
\bb=\bb^0\cup\left(\bigcup\nolimits_{S\in \st} \bb_S\right).
$$
\end{theorem}

The proof of this theorem requires some auxiliary results.

By a recurrent application of Proposition \ref{splitting}, $\pp$ splits into the disjoint union of the subsets $\pp_\ty$, for $\ty$ running on the leaves of $\tcal(f)$. In other words,
$$
\pp=\pp^0\cup \left(\bigcup\nolimits_{S\in \st} \pp_S\right),\qquad \pp^0:=\pp_{\ty^0},\quad\pp_S:=\pp_{\ty_S}.
$$


\begin{lemma}\label{Hequal}
Let $S$ be a terminal side. For any $\,\j=(j_0,\dots,j_r) \in J_S$, there exists $\p\in \pp_S$ such that
\begin{equation}\label{Hequaleq}
w_\p(\t^{j_0})=0;\qquad w_\p(q_{i,j_i}(\t))=\rho\,H_{i,j_i},\quad 1\le i\le r.
\end{equation}

\end{lemma}

\begin{proof}
Let $\ty=\ty_S=(t_0;(g_1,\la_1,t_1);\dots;(g_r,\la_r,t_r))$.

For any $\p\in \pp_S$, let $\m\in\mpa$ such that $\p\in\pp_{\ty_\m}$. Then,
$$w_\p(\t)>0  \sii y\mid t_0\ \mbox{ and }\ \m\cap A_1=(\m_0,y).
$$
Hence, there exists $\m'\in\mx_p(A_1)$ such that $w_\p(\t^{j_0})=0$ for all $\p\in\pp_{\ty_\m}$, for all $\m\in\mpa$ with $\m\cap A_1=\m'$.
In fact, if $t_0=y$, then $j_0=0$, and any $\m'\in\mx_p(A_1)$ does the job. If $t_0\ne y$, there are irreducible factors in $\rd_p(t_0)$ leading to maximal ideals in $\mx_p(A_1)$ different from $(\m_0,y)$.

Now starting with this $\m'\in\mx_p(A_1)$, we use a recurrent argument to show the existence of $\m\in\mpa$ such that $\m\cap A_1=\m'$ and $$w_\p(q_{i,j_i}(\t))=\rho\,H_{i,j_i}, \quad 1\le i\le r,\qquad \forall\,\p\in\pp_{\ty_\m}.$$

Suppose $\m'\in\mx_p(A_r)$ satisfies this condition for all $\p\in\pp_{\ty_{\m'}}$ and all $i<r$. Denote $q=q_{r,j_r}$.
By Theorem \ref{comparison}, there is a unique $p$-adic irreducible factor $\phi_{\m'}$ of $g_r$ which is a representative of $\ty_{\m'}$, and satisfies
$$
R_{\m',r}(q)(y)=\tau_r(q)\rdm(R_r(q))(\sigma_ry),
$$
for some non-zero constants $\tau_r(q),\sigma_r\in \F_{\m'}$.
By Lemma \ref{Rquot},
$$
\deg R_r(q)\le(s_{d_r}-(s_{d_r}-j_r))/e_r\le j_r/e_r< f_r.
$$
Since $\rdm(t_r)$ is squarefree, it has a monic irreducible factor not dividing $\rdm(R_r(q))$. If we write this irreducible factor as $\rdm(\varphi)$ for some monic $\varphi\in A_r[y]$, we see that $\psi(y):=\sigma_r^{-\deg\varphi}\rdm(\varphi)(\sigma_ry)$ does not divide $R_{\m',r}(q)$. Hence, for the maximal ideal $\m=(\m',\varphi)$ in $\mx_p(A_{r+1})$, the type $\ty_\m=(\ty_{\m'};(\phi_\m,\rho\la_r,\psi))$ divides $f$ and it  does not divide $q$. Hence,
$$
w_\p(q(\t))=v_{\m,r}(q)/(e_0\cdots e_r)=\rho\, v_r(q)/(e_0\cdots e_r)=\rho\,H_s.
$$
for every $\p\in \pp_{\ty_\m}$, by \cite[Prop. 2.9]{GMN}.
\end{proof}

\begin{lemma}\label{w<1}
We have $\#\bb=n$ and $0\le w(\alpha)<1$, for all $\alpha\in\bb$.
\end{lemma}

\begin{proof}
By Corollary \ref{epfp}, $\#\bb^0=\sum_{\p\in\pp^0}e_\p f_\p$ and
$$
\#\bb_S=(e_0f_0)\cdots(e_rf_r)=\sum\nolimits_{\p\in\pp_S}e_\p f_\p,\quad \forall S\in\st.
$$
Hence, $\#\bb=\sum_{\p\in\pp}e_\p f_\p=n$.

For any terminal side $S$ and any $\alpha\in\bb_S$, we have $0\le w(\alpha)<1$ by Lemma \ref{Hequal}. Let us show that $w(\alpha)=0$ for all $\alpha\in\bb^0$.

Let $\ty^0=(t)$ be the leaf of order zero of $\tcal(f)$. Let $g\in\Z[x]$ be a monic lifting of $t$, and consider the division with remainder  $f=qg+a$. We have $\rd_N(a)=0$, and $\rd_N(f)=\rd_N(q)\rd_N(g)=\rd_N(q)t$, so that $\rd_N(q)$ is coprime with $t$. Hence, $\rd_p(q)$ is coprime with $\rd_p(t)$, and this implies $w_\p(q(\t))=0$ for all $\p\in\pp^0$.

If $t=y$, then $\bb^0=\left\{q(\t)\right\}$ and the lemma is proven. If $t\ne y $, then $\rd_p(t)$ has irreducible factors different from $y$ and there are prime ideals $\p\in\pp^0$ such that $w_\p(\t)=0$. For them, we have
$w_\p(\t^jq(\t))=0$ for all $j$. 
\end{proof}

\begin{definition}
With the notation of Definition \ref{terminalS}, let $S$ be a terminal side and consider the splitting $\pp_{\ty_{r-1}}=\bigcup\nolimits_{\la,t}\pp_{\ty_{\la,t}}$ of Proposition \ref{splitting}:

A terminal side $T$ is said to \emph{dominate} $S$ if there exists a pair $(\la,t)$ with $\la\ge \la_S$, such that $\ty_{\la,t}$ is a truncation of the leaf $\ty_T$ of $\tcal(f)$ (or equivalently, $\pp_T\subset\pp_{\ty_{\la,t}}$). In this case, we write $T\ge S$.
\end{definition}

\begin{lemma}\label{dom}
\begin{enumerate}
\item Domination is a partial ordering on $\st$.
\item For any $S,T\in \st$ such that $T\not\ge S$, we have
$$
\alpha\in\bb_S,\quad \p\in \pp_T\ \imp\ w_\p(\alpha)>w(\alpha).
$$
\end{enumerate}
\end{lemma}

\begin{proof}
Let $S\in\st$, and let $(\la_S,t_S)$ be the unique pair such that $\ty_S=\ty_{\la_S,t_S}$.
The reflexive property $S\ge S$ is obvious.

Let $T\in\st$ such that $T\ge S$. Let $(\la,t)$ be the unique pair such that $\la\ge\la_S$ and $\ty_{\la,t}$ is a truncation of $\ty_T$.

If $T\ne S$, then $(\la_S,t_S)\ne(\la,t)$ and $\ty_S=\ty_{\la_S,t_S}$ cannot have  $\ty_{\la,t}$ as one of its truncations. This shows that domination is antisymmetric.

Now, suppose $R\ge T$. If $\ty_T=\ty_{\la,t}$, then there is a pair $(\la',t')$ with $\la'\ge\la$ such that $\ty_{\la',t'}$ is a truncation of $\ty_R$. If $\ty_T\ne \ty_{\la,t}$, then the previous node of $\ty_T$ is a truncation of $\ty_R$. In both cases, $R\ge S$. This shows that domination is transitive and ends the proof of item (1).

Let $\alpha=q_\j\in \bb_S$ for some $\j=(j_0,\dots,j_r)\in J_S$.
If  $T\not\ge S$, then a prime ideal $\p\in\pp_T$ satisfies either $\p\not\in\pp_{\ty_{r-1}}$, or $\p\in\pp_{\mu,t}$ with $\mu<\la_S$. In both cases, we saw  along the proof of Theorem \ref{denquot} that $w_\p(q_{r,j_r})>\rho\,H_{r,j_r}$. Since  $w_\p(q_{i,j_i})\ge\rho\,H_{i,j_i}$ for all $1\le i<r$, again by Theorem \ref{denquot}, we deduce
$$
w_\p(\alpha)>\rho\,(H_{1,j_1}+\cdots+H_{r,j_r})=w(\alpha),
$$
the last equality by Lemma \ref{Hequal}.
\end{proof}

\begin{lemma}\label{alphabeta}
Let $S\in\st$. For $0\le i\le r$ and any $\epsilon\in\Q$ we denote
$$
\bb_{S,\epsilon}^{(i)}=\left\{\dfrac{\t^{j_0} q_{1,j_1}(\t)\cdots q_{i,j_i}(\t)}{N^{\lfloor H_{1,j_1}+\cdots +H_{i,j_i} \rfloor}}\ \Big|\  0\le j_\ell<e_\ell f_\ell,\ 0\le \ell\le i\right\}\cap \bb_\epsilon.
$$

For any $\delta \in w(\bb_S)$, there exists a unique integer $0\le a<e_r$ such that
\begin{equation}\label{alpha}
\bb_{S,\delta}:=\bb_S\cap\bb_\delta=\bigcup\nolimits_{0\le k<f_r}q_{r,ke_r+a}(\t)N^{-m_k}\,\bb_{S,\delta_k}^{(r-1)},
\end{equation}
where $m_k\in\Z$ and $\delta_k\in\Q$ depend only on $S$, $\delta$ and $k$.
\end{lemma}

\begin{proof}
Let $\ty_S=(t_0;(g_1,\la_1,t_1);\dots;(g_r,\la_r,t_r))$.
By Corollary \ref{epfp}, all $\p\in\pp_S$ have ramification index $e_\p=e_0\cdots e_r$. Thus, there is an integer $b$ for which
$$\delta=b/e_0\cdots e_r,\qquad 0\le b<e_0\cdots e_r.
$$
Take $\alpha=q_\j\in\bb_{S, \delta}$ for $\j=(j_0\dots,j_r)\in J_S$. Write $\alpha=q_{r,j_r}(\t)N^{-m_k}\beta$ with
$$
\begin{array}{c}
\beta=\t^{j_0}q_{1,j_1}(\t)\cdots q_{r-1,j_{r-1}}(\t)N^{-\lfloor H_{1,j_1}+\cdots +H_{r-1,j_{r-1}}\rfloor},\\
m_k=\lfloor H_{1,j_1}+\cdots +H_{r,j_{r}}\rfloor- \lfloor H_{1,j_1}+\cdots +H_{r-1,j_{r-1}}\rfloor.
\end{array}
$$

Denote $s=s_{d_r}-j_r$, $q_s=q_{r,j_r}$, and let $j_r=ke_r+a$, $0\le a<e_r$.

The Newton polygon $N_r(q_s)$, displayed in Figure \ref{figNQ}, is easy to deduce from $N_r(q_sg_r^s)$, which was described in Figure \ref{figSplit}.

We denote by $(s_0,u_0)$ the left endpoint of $S_r(q_s)$, and write $s_0=\ell e_r+a$ for some integer $0\le \ell\le k$. A look at Figure \ref{figNQ} shows that
\begin{equation}\label{Hrjr}
H_{r,j_r}=(u_0+s_0\la_r))/e_0\cdots e_{r-1}=\left((u_0+\ell h_r)e_r+ah_r\right)/e_0\cdots e_r.
\end{equation}

\begin{figure}\caption{Newton polygon of $q_{r,j_r}$. The line $L$ has slope $-\la_r$.}\label{figNQ}\setlength{\unitlength}{4.mm}
\begin{picture}(15,10.5)
\put(-.2,8.7){$\bullet$}\put(.4,5.9){$\times$}\put(1.6,5.3){$\times$}
\put(2.9,4.7){$\bullet$}\put(4,4.1){$\times$}\put(5.2,3.6){$\bullet$}
\put(6.4,2.9){$\times$}\put(7.7,2.3){$\bullet$}
\put(0,-.5){\line(0,1){10.5}}\put(-1,0.5){\line(1,0){17}}
\put(3.2,4.8){\line(-3,4){3.2}}\put(3.2,4.86){\line(-3,4){3.2}}
\put(7.8,2.55){\line(5,-1){4.5}}\put(7.8,2.58){\line(5,-1){4.5}}
\put(3,5.03){\line(2,-1){5}}
\multiput(.75,.3)(0,.25){24}{\vrule height2pt}
\multiput(3.1,.3)(0,.25){19}{\vrule height2pt}
\multiput(7.9,.3)(0,.25){10}{\vrule height2pt}
\put(.6,-.4){\begin{footnotesize}$a$\end{footnotesize}}
\put(2.4,-.4){\begin{footnotesize}$s_{0}\!=\!a\!+\!\ell e_r$\end{footnotesize}}
\put(7.4,-.4){\begin{footnotesize}$j_r\!=\!a\!+\!ke_r$\end{footnotesize}}
\multiput(-.1,4.9)(.25,0){14}{\hbox to 2pt{\hrulefill }}
\put(-1.2,4.7){\begin{footnotesize}$u_{0}$\end{footnotesize}}
\put(-.6,-.4){\begin{footnotesize}$0$\end{footnotesize}}
\put(2.5,8){\begin{footnotesize}$N_{r}(q_{r,j_r})$\end{footnotesize}}
\put(0,6.5){\line(2,-1){10}}
\put(-5.5,6.3){\begin{footnotesize}$e_0\cdots e_{r-1}H_{r,j_r}$\end{footnotesize}}
\put(-.1,6.5){\line(1,0){.2}}\put(10,1){\begin{footnotesize}$L$\end{footnotesize}}
\end{picture}\end{figure}
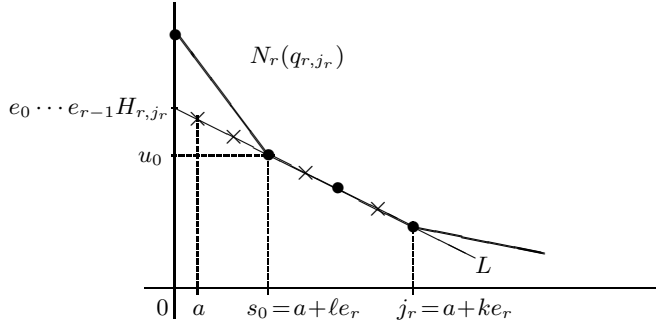

Take $\p\in \pp_S$ satisfying (\ref{Hequaleq}). By Theorem \ref{denquot},
$$
\begin{array}{l}
w(\alpha)=w_\p(\alpha)=\rho\left(H_{r,j_r}-m_k\right)+w_\p(\beta),\\
w(\beta)=w_\p(\beta)=\rho\left(H_{1,j_1}+\cdots +H_{r-1,j_{r-1}}-\lfloor H_{1,j_1}+\cdots +H_{r-1,j_{r-1}}\rfloor\right).
\end{array}
$$
Recall that $e_0\cdots e_iH_{i,j_i}=v_i(q_{i,j_i})\in\Z$ for all $i$. By our general assumptions, $\rho>1$ occurs only if $e_0=\cdots=e_r=1$, in which case all $H_{i,j_i}$ are integers and $w(\alpha)=w(\beta)=0$. Hence, we may take $\rho=1$ in the above equations. From these equalities and (\ref{Hrjr}) we deduce the existence of $b_k\in\Z$ such that
\begin{equation}\label{bprima}
\begin{array}{c}
w(\beta)=w_\p(\beta)=b_k/e_0\cdots e_{r-1},\qquad 0\le b_k<e_0\cdots e_{r-1},\\
b/e_0\cdots e_r=w(\alpha)=-m_k+\left((u_0+\ell h_r+b_k)e_r+ah_r)/(e_0\cdots e_r\right).
\end{array}
\end{equation}

Hence, the integer $a$ depends only on $S$ and $\delta$, because it is uniquely determined by the conditions:
$$
ah_r\equiv b\md{e_r},\qquad 0\le a <e_r.
$$
In particular, there are only $f_r$ possible values of $j_r=ke_r+a$, determined by the choice of $0\le k<f_r$.
Also, the integers $m_k,b_k$ depend only on $S$, $\delta$ and $k$. In fact, the integers $u_0$ and $\ell$ depend only on $S$ and $k$, and $b_k$ is uniquely determined by the conditions:
$$
b_k\equiv \dfrac{b-ah_r}{e_r}-u_0-\ell h_r\md{(e_0\cdots e_{r-1})},\qquad 0\le b_k<e_0\cdots e_{r-1}.
$$
The integer $m_k$ is then determined by (\ref{bprima}).
Thus, the proof of (\ref{alpha}) is complete, by taking $\delta_k=w(\beta)=b_k/e_0\cdots e_{r-1}$.
\end{proof}

\begin{lemma}\label{bsreduced}
For $S\in\st$ let $V_S=\prod_{\p\in\pp_S}\F_\p$, and denote $\op{pr}_S\colon V\to V_S$ the canonical projection. Then, $\op{pr}_S(\rd_\delta(\bb_{S,\delta}))$ is a $(\Z/p\Z)$-basis of $V_S$ for all $\delta\in w(\bb_S)$.
\end{lemma}

\begin{proof}
Take $0\le i\le r$ and $\epsilon=c/e_0\cdots e_i$, for some integer $0\le c<e_0\cdots e_i$.
Consider the mapping
$$
\rd_{S,\epsilon}^{(i)}\colon \,\bb_{S,\epsilon}^{(i)}\lra V_S,\qquad
\alpha\mapsto\left(\alpha/\pi_{i+1}(\t)^c+\p\Z_\p\right)_{\p\in\pp_S}
$$
By Lemma \ref{wrphir}, $w_\p(\pi_{i+1}(\t))=\rho/e_0\cdots e_i$ for all $\p\in\pp_S$. By our general assumptions, $\rho>1$ implies $e_0=\cdots= e_i=1$, in which case $c=0$. Thus,  $w_\p(\pi_{i+1}(\t)^c)=c/e_0\cdots e_i=\epsilon$ for all $\p\in\pp_S$, and $\rd_{S,\epsilon}^{(i)}$ is well defined.

Let $A$ be the artinian algebra attached to the type $\ty_S$.
By Corollary \ref{epfp}, the mapping $\ga_{\ty_S}\colon \rd_p(A)\to V_S$ is an isomorphism. Since $\bb_{S,\delta}^{(r)}=\bb_{S,\delta}$ and $\rd_{S,\delta}^{(r)}=\op{pr}_S\circ\rd_\delta$, the lemma is a consequence of the following:\medskip

\noindent{\bf Claim. }The set $\ga_{\ty_S}^{-1}\left(\rd_{S,\epsilon}^{(i)}\left(\bb_{S,\epsilon}^{(i)}\right)\right)$ is a $(\Z/p\Z)$-basis of $\rd_p(A_{i+1})$.\medskip

If $i=0$, then $\epsilon=0$ and $\bb_{S,0}^{(0)}=\{1,\t,\dots,\t^{f_0-1}\}$. The image of this set under $\ga_{\ty_S}^{-1}$ is $\{1,\rd_p(z_0),\dots,\rd_p(z_0)^{f_0-1}\}$, which is a $(\Z/p\Z)$-basis of $\rd_p(A_1)$ by Lemma \ref{augmentation}.

Assuming the Claim for some $0\le i<r$ and all $\epsilon\in (e_0\cdots e_i)^{-1}\Z\cap[0,1)$, let us show that it holds for $i+1$ and all $\epsilon\in (e_0\cdots e_{i+1})^{-1}\Z\cap[0,1)$.

For commodity we work out the case $i=r-1$.

Let $\delta=b/e_0\cdots e_r$, with $0\le b<e_0\cdots e_r$. By Lemma \ref{alphabeta}, the elements in the set $\bb_{S,\delta}^{(r)}=\bb_{S,\delta}$ may be parameterized as:
$$
\alpha=q_{r,ke_r+a}(\t)N^{-m_k}\beta,\qquad 0\le k<f_r,\ \beta\in \bb_{S,\delta_k}^{(r-1)}.
$$

With the notation of Lemma \ref{alphabeta}, for each $\p\in\pp_S$ we may express $\alpha/\pi_{r+1}(\t)^b+\p$ as the product $(\alpha_1+\p)(\alpha_2+\p)(\alpha_3+\p)$, with:
$$
\alpha_1=
\dfrac{q_{r,j_r}(\t)}{\phi_r(\t)^{s_0}\pi_r(\t)^{u_0}},\qquad
\alpha_2=
\dfrac{\phi_r(\t)^{s_0}\pi_r(\t)^{u_0}\pi_r(\t)^{b_k}}{N^{m_k}\pi_{r+1}(\t)^b},\qquad\alpha_3=
\dfrac{\beta}{\pi_r(\t)^{b_k}},
$$
where $(s_0,u_0)$ are taken from Theorem \ref{AFp}. If $R_r(f)=c_0+c_1y+\cdots+c_dy^d$, Lemma \ref{Rquot} and Theorem \ref{AFp},(C) show that (see Figure \ref{figNQ})
$$
\ga_{\ty_S}^{-1}\left(\alpha_1+\p\right)_{\p\in\pp_S}=\rd_p(R_r(q_{r,j_r})(z_r))=\rd_p(c_{d-k+\ell}+\cdots +c_d z_r^{k-\ell}).
$$

From the identities $b=ah_r-\nu e_r$ (for some integer $\nu$), $s_0=\ell e_r+a$, and $\ell_rh_r+\ell'_re_r=1$, we deduce:
\begin{equation}\label{exponents}
s_0-\ell_rb=\ell e_r+a-\ell_r(ah_r-\nu e_r)=(\ell+\ell'_ra+\ell_r\nu)e_r=(\ell+c)e_r,
\end{equation}
where the integer $c:=\ell'_ra+\ell_r\nu$ depends only on $S$ and $\delta$.

Let  $h\in\Q(x)$ such that $\alpha_2=h(\t)$. By (\ref{exponents}) and  (\ref{ratfs}), we can write
\begin{align*}
h=&\;\phi_r^{s_0}\pi_r^{u_0+b_k}N^{-m_k}\pi_{r+1}^{-b}=\phi_r^{s_0-\ell_rb}\pi_r^{u_0+b_k-\ell'_rb}N^{-m_k}\\=&\;\ga_r^{\ell+c}\,\pi_r^{u_0+b_k-\ell'_rb+(\ell+c)h_r}N^{-m_k}=\ga_r^{\ell+c}\,N^{n_0}g_1^{n_1}\cdots g_{r-1}^{n_{r-1}},
\end{align*}
for some integers $n_0,\dots,n_{r-1}$.
By Lemmas \ref{wrphir} and \ref{gammas},  $v_r(h)=0$ and
$h=\ga_r^{\ell+c}\ga_1^{a_1}\cdots \ga_{r-1}^{a_{r-1}}$,
for some integers $a_1,\dots,a_{r-1}$. Hence,
$$
\ga_{\ty_S}^{-1}\left(\alpha_2+\p\right)_{\p\in\pp_S}=\rd_p(z_r)^{\ell+c}\,\tau_k,
$$
where $\tau_k=\rd_p(z_1)^{a_1}\cdots \rd_p(z_{r-1})^{a_{r-1}}$ is a unit in $\rd_p(A_r)$ which depends only on $S$, $\delta$ and $k$.

Finally, for $\beta$ running on $\bb_{S,\delta_k}^{(r-1)}$, the elements
$$
u_\beta:=\ga_{\ty_S}^{-1}\left(\alpha_3+\p\right)_{\p\in\pp_S}\in\rd_p(A_r)
$$
form a $(\Z/p\Z)$-basis of $\rd_p(A_r)$ by hypothesis. Summing up, we get
$$
\ga_{\ty_S}^{-1}\left(\alpha/\pi_{r+1}(\t)^b+\p\right)_{\p\in\pp_S}=\rd_p(z_r)^c\tau_k\rd_p(c_{d-k+\ell}z_r^{\ell}+\cdots +c_d z_r^k) u_\beta.
$$
Since $\rd_p(z_r)$ is a unit, the absolute constant $\rd_p(z_r)^c$ (depending only on $S$ and $\delta$) may be dropped from all these elements. Denote
$$
\zeta_{k,\beta}=\tau_k\rd_p(c_{d-k+\ell}z_r^{\ell}+\cdots +c_d z_r^k) u_\beta.
$$

Since the pairs $(k,\beta)$ take $f_0f_1\cdots f_r=\dim_{\Z/p\Z}\rd_p(A)$ values, we need only to show that these elements $\zeta_{k,\beta}$ are linearly independent.

Suppose that for some family of elements $a_{k,\beta}\in\Z/p\Z$, we have
\begin{equation}\label{equalzero2}
\sum\nolimits_{k,\beta}a_{k,\beta}\,\zeta_{k,\beta}=0.
\end{equation}
Consider $w_0,\dots,w_{f_r-1}\in \rd_p(A_r)$ such that
$$0=\sum\nolimits_{k,\beta}a_{k,\beta}\,\zeta_{k,\beta}=w_0+w_1\rd_p(z_r)+\cdots +w_{f_r-1}\rd_p(z_r)^{f_r-1}.$$ Lemma \ref{augmentation} shows that all these coefficients are equal to zero. On the other hand, our explicit formulas show that
$$
0=w_{f_r-1}=\tau_{f_r-1}\rd_p(c_d)\sum\nolimits_\beta a_{f_r-1,\beta}\,u_\beta.
$$
Since $\tau_{f_r-1}$ and $\rd_p(c_d)$ are units, we deduce $0=\sum_\beta a_{f_r-1,\beta}\,u_\beta$, leading to $a_{f_r-1,\beta}=0$ for all $\beta\in \bb_{S,\delta_{f_r-1}}^{(r-1)}$, by our hypothesis. Hence, we obtain an identity like  (\ref{equalzero2}) for $0\le k\le f_r-2$. An iteration of this argument shows that $a_{k,\beta}=0$ for all $k,\beta$, so that our family $\zeta_{k,\beta}$ is linearly independent.
\end{proof}

\noindent{\bf Proof of Theorem \ref{pBasis}. }Let $p$ be a prime factor of $N$. By Lemmas \ref{reducedbasis} and \ref{w<1}, it suffices to show that $\bb$ is a $p$-reduced set. By Theorem \ref{criterion} we must prove that $\rd_\delta(\bb_\delta)$ is $(\Z/p\Z)$-linearly independent for all $\delta\in w(\bb)$.

Denote $\bb_{\op{trm},\delta}=\bigcup\nolimits_{S\in\st}\bb_{S,\delta}$.
Let us first discuss the case $\delta=0$. We saw along the proof of Lemma \ref{w<1} that $\bb^0\subset\bb_0$. Consider the splitting:
$$
V=V^0\times V_{\op{trm}},\quad V^0=\prod\nolimits_{\p\in\pp^0}\F_\p,\quad V_{\op{trm}}=\prod\nolimits_{\p\not\in\pp^0}\F_\p.
$$

By the proof of Theorem \ref{denquot} applied to the types of order zero given by the root nodes of $\tcal(f)$, we have
$$w_\p(q(\t))>0,\ \forall\,\p\not\in\pp^0;\qquad
w_\p(\alpha)>0,\ \forall\,\p\in\pp^0,\ \forall\,\alpha\in\bb_{\op{trm},0}.
$$Hence, $\rd_0(\bb^0)\subset V^0\times \{0\}$ and $\rd_0(\bb_{\op{trm},0})\subset \{0\}\times V_{\op{trm}}$.

It is obvious that $\rd_0(\bb^0)$ is linearly independent.
Thus, $\rd_0(\bb_0)$ is linearly independent if and only if
 $\rd_0(\bb_{\op{trm},0})$ is linearly independent.

Since for $\delta>0$ we have $\bb_\delta=\bb_{\op{trm},\delta}$, the proof of the theorem will be complete if we show that  $\rd_\delta(\bb_{\op{trm},\delta})$ is linearly independent for all $\delta\ge0$.

For any given $\delta\in w(\bb)$, let $\mathcal{S}_\delta:=\{S\in\st\mid \bb_{S,\delta}\ne\emptyset\}$. For any $S\in\mathcal{S}_\delta$ write $\rd_\delta(\bb_{S,\delta})=\{\zeta_{S,m}\mid 1\le m\le \#\bb_{S,\delta}\}\subset V$.

Suppose that for some family of elements $a_{S,m}\in\Z/p\Z$, we have
\begin{equation}\label{equalzero}
\sum\nolimits_{S,m}a_{S,m}\,\zeta_{S,m}=0,
\end{equation}
the sum running on $S\in\mathcal{S}_\delta$ and $1\le m\le \#\bb_{S,\delta}$. Take $T\in\mathcal{S}_\delta$ minimal with respect to the partial ordering of domination. By Lemma \ref{dom},
$$
w_\p(\alpha)>\delta,\quad \forall\,\p\in\pp_{T},\ \forall\,\alpha\in\bb_{S,\delta},\ \forall\,S\in\mathcal{S}_\delta, \ S\ne T.
$$
Hence, $\op{pr}_{T}(\zeta_{S,m})=0$ for all $S\in\mathcal{S}_\delta$, $S\ne T$, and all $m$. Thus, if we apply $\op{pr}_{T}$ to both sides of (\ref{equalzero}), we get
$$
\sum\nolimits_ma_{T,m}\op{pr}_{T}(\zeta_{T,m})=0.
$$
By Lemma \ref{bsreduced}, $a_{T,m}=0$, for all $m$. Hence, we get again an equation like (\ref{equalzero}) for $S$ running on $\mathcal{S}_\delta\setminus\{T\}$. An iteration of this argument shows that $a_{S,m}=0$ for all $S,m$. Thus,  $\rd_\delta(\bb_{\op{trm},\delta})$ is $(\Z/p\Z)$-linearly independent.\qed

\subsection{Computation of global integral bases}\label{subsecGlobalBases}

Let $P$ be the product of all prime factors $p$ of $\dsc(f)$ with $\ord_p(\dsc(f))>1$. It is well-known that a $P$-integral basis of $K$  is necessarily a global integral basis of $K$.

Our algorithm finds a splitting $P=N_1\cdots N_k$ for which we are able to compute $N_i$-integral bases $\bb_{N_1},\dots,\bb_{N_k}$ of $K$. This is sufficient for our purpose, because there are standard procedures to patch these bases into a $P$-integral basis.
Along the algorithm we use  the following subroutines:\medskip

\noindent{\tt CoprimeSplitting($d$,\,$N$)}

\noindent By an iterative application of $\op{gcd}$ routines, a proper divisor $d$ of $N$ determines a factorization $N=c_1^{a_1}\cdots c_k^{a_k}$ with pairwise coprime bases $c_1,\dots,c_k$. The routine expresses then each $c_i=d_i^{e_i}$ as a perfect power (with $e_i\ge1$) and outputs the list $[d_1,\dots,d_k]$.\medskip

\noindent{\tt SFD($N$)}

\noindent Computes the squarefree decomposition $N=d_1^{\ell_1}\cdots d_k^{\ell_k}$,  $\ell_1<\cdots<\ell_k$, and  outputs the list of coprime squarefree factors $[d_1,\dots,d_k]$.\bigskip

\noindent{\bf GLOBAL BASIS ALGORITHM}\vskip 1mm

\noindent INPUT:

$-$ A monic irreducible polynomial $f\in \Z[x]$ of degre $n>1$. \medskip

$-$ An integer $D$. \medskip


\sst{1}$ {\tt NBases}\leftarrow [\ ]$

\sst{2}FOR each prime number $p\le n$ such that $\ord_p(D)>1$ DO

\stst{}Apply the classical OM algorithm to compute a $p$-integral basis $\bb_p$

\stst{}Add the pair $[p,\bb_p]$ to {\tt NBases}

\stst{}$D\leftarrow D\,p^{-\ord_p(N)}$


\sst{3}{\tt Moduli} $\leftarrow\{D\}$.

\sst{}{\bf WHILE $\#${\tt Moduli}\;$>0$  DO}

\stst{4}Extract a modulus $N$ from {\tt Moduli} and call {\tt SF-OM}($N,f$)

\stst{5}IF a proper factor $d$ of $N$ is detected THEN
join the output of

\stst{}{\tt CoprimeSplitting}($d,N$) to {\tt Moduli} and
go to step {\bf 4}

\stst{6}IF the output tree $\tcal(f)$ has some ramified leaf THEN

\ststst{}{\tt SFD}($N$). Let $d_1,\dots,d_k$ be the squarefree factors

\ststst{}IF $k>1$ THEN  add  $d_1,\dots,d_k$ to {\tt Moduli} and
go to step {\bf 4}.

\stst{7}Compute the basis $\bb_N$ of Theorem \ref{pBasis} and add $[N,\bb_N]$ to {\tt NBases}

\sst{}{\bf END WHILE}\medskip

\noindent OUTPUT: A list of pairs  $[N,\bb_N]$, where $\bb_N$ is a reduced $N$-integral basis of $K$.

\bigskip

\noindent
{\bf Remarks}

\medskip
$\bullet$ The input integer $D$ will be in general the discriminant of $f$. However, since the computation of this discriminant may be unfeasible, we admit the possibility of working with a partial factorization of it.

\medskip

$\bullet$ The pairs $[N,\bb_N]$ of {\tt NBases} may be patched to obtain a  $D$- integral basis of $K$.   For instance, we can apply the triangulation algorithm given in \cite{mn} to all reduced $N$-integral bases, and then glue them together by means of the Chinese remainder theorem. When $D$ is the discriminant of $f$ or has the same prime factors, this leads to a global integral basis of $K$.

\medskip

$\bullet$  The routine {\tt SFD($N$)} is the bottleneck of the algorithm. However, in some cases the method is efficient because the successive splittings of the discriminant, caused by the SF-OM routine in step {\bf 4}, lead to factors of $N$ which are sufficiently small to admit a feasible performance of {\tt SFD($N$)}.

\section{Examples}

We illustrate the flow of the {\tt Global Basis} algorithm with some examples, which prove its power and efficiency. There are two obvious gains with respect to other existing algorithms:  the  factorization of the discriminant is not necessary, and  square-free decomposition of residual polynomials is used instead of the complete factorization. Thus, finding an SF-OM representation of a polynomial for a given integer $N$ is cheaper than finding its OM representations for all the primes dividing $N$. 

A single SF-type may encode many irreducible types. This compactness of SF-OM representations has a direct impact in the computation of integral bases, because it reduces drastically the amount of $\Z$-linear algebra needed to glue the local bases, which is a time and space expensive task.


The first example we present  exhibits the capability of the {\tt Global Basis} algorithm to detect some factors of the discriminant along the computation of an integral basis. We explain it to a certain detail for a better comprehension of the algorithm.  The second example makes apparent the benefits of allowing non-irreducible residual polynomials. Finally, we introduce an example that illustrates  the compactness of the SF-OM representations.

The three examples are parametrical, the parameters representing integer numbers. Most of the discussions are theoretical, but we include some particular cases which show a number of phenomena occurring along the flow of the algorithm. For these particular cases, we have used  our own  implementation  of the algorithm in {\em Magma}.

Our experimentation suggests that our program is  competitive to other routines for computation of integral bases, with the advantage that it can handle a broader range of number fields.




\subsection{Example 1}

Let us consider the polynomial
  $$
  f =x^4+2Nx^2+N^3(N-1)x+N^2,
  $$
   with $N>4$ an odd squarefree integer. Its discriminant is
   $$
   \dsc(f)=-N^9 (N-1)^2 \left(27 N^5-54 N^4+27 N^3-256\right).
   $$
We first apply the SF-OM algorithm to the pair $f$, $N$. We assume that $N$ is not divisible by 3 just for simplicity.

Since $\rdn(f)=y^4$, we take $g_1=x$, and obtain a  Newton polygon $N_1(f)$ with a unique side of slope -1/2. The residual polynomial of this side is $R_1(f)=y^2+2y+1=(y+1)^2$. We lift the irreducible factor of $R_1(f)$ to $g_2=x^2+N$. The second order Newton polygon $N_2(f)$ has a unique side of slope $-3/2$ joining the points  $(0,7)$ and $(2,4)$. The residual polynomial $R_2(f)=y+1$ is squarefree. Hence, the SF-OM-representation of $f$ with respect to $N$ has a unique type:
$$
\ty_N=(y; (x,-1/2, y+1);(x^2+N,-3/2,y+1)).
$$
In particular, $f$ is irreducible over $\Q_p$ for every prime $p\mid N$.

After theorem \ref{pBasis}, we know that
$$
{\mathcal B}_N=\left\{1, \t,   \frac{\t^2}{N},\frac{\t^3+N\t}{N^2}\right\}
$$
is an $N$-basis of the number field $K=\Q(\t)$ defined by $f$.

Thus, no matter how many prime factors the integer $N$ has, the SF-OM algorithm finds an $N$-basis in one hit. This is a significant  improvement with respect to the classical OM algorithm, which has to work with every prime factor of $N$, and patch all local prime bases to find an $N$-basis.\medskip


Let us now see what happens when we apply the {\tt Global Basis} algorithm to the pair $f$, $D:=\dsc(f)$. In a realistic situation, we cannot take profit of the factorization of $D$ given above.

We must work out the small primes 2 and 3 apart. Thus, we start by applying the SF-OM algorithm to the pair $f, D_1$ where $D_1=|D|/2^{v_2(D)}3^{v_3(D)}$.

In the very beginning, the square-free factorization of $f\pmod{ D_1}$ requires the computation of the GCD of $f$ and its derivative. Along this calculation, when we try to divide $f'$ out by $$f -\frac14xf'=Nx^2+\frac34 N^3(N-1)x+N^2,$$
we detect the factor $N$ of $D_1$, and then its coprime cofactor $D_2=D_1/N^9$.

We start the main loop again with \ {\tt Moduli}=$\{N, D_2 \}$. We discussed already how the algorithm computes an $N$-basis. We set \ {\tt Moduli}=$\{D_2 \}$ and apply the SF-OM algorithm to the pair $f,D_2$.

We compute the squarefree decomposition of $f\pmod{ D_2}$.
Thanks to the coefficient $N^3(N-1)$ in $f$, the computation of $\gcd(f,f')$ detects the factors  coming from $\gcd(N-1,D_2)$. We set \mbox{\tt Moduli}=$\{N_1, D_3\}$, where $N_1, D_3$ are the output of \mbox{\tt CoprimeSplitting}$(N-1,D_2)$.

For the modulus $N_1$ we have $f\equiv(x^2+1)^2\pmod{N_1}$, and
the SF-OM-representation of $f$ with respect to $N_1$ has a unique type of order 1:
$$
\ty_{N_1}=(y^2+1; (x^2+1,-1/2, y+z_0)), \quad z_0=y+(y^2+1)(\Z/N_1\Z)[y].
$$
The power basis $\{1,\t,\t^2,\t^3\}$ is an $N_1$-integral basis.

We have seen so far that, for general $N$, we will find three divisors of $\dsc(f)$, and for two of them the local bases have an specific form.
For the remaining factor $D_3$, a bunch of phenomena can occur, including the chance to find new factors.

For instance, for $N=10007\cdot10009$, we have $N_1=50080031=(N-1)/2$ and $D_3=68041943397686978810459285162708530849445$. When we apply the SF-OM algorithm to the pair $f, D_3$ the obvious 5 dividing $D_3$ is detected in a valuation computation. For both factors 5 and $D_3/5$  the power basis is a local basis. The classical Montes algorithm shows that the power basis is a 6-integral basis. Hence, in this case, ${\mathcal B}_N$  is a global integral basis.

\subsection{Example 2}

Let $p$ be a prime number and take positive integers $r,m$ such that $r<p/2$ and $r<m$.
Consider the polynomial
$$
f_{p,r,m}=  (x^2+p)(x^2+2p)\cdots(x^2+rp)+p^m=x^{2r}+b_2x^{2r-2}+\dots+b_{2r}.
$$
We assume that  $f:=f_{p,r,m}$ is irreducible.
The SF-OM-algorithm applied to $f$ and $p$ yields a unique type of order 1:
$$
\ty_N=(y; (x,-1/2, (y+1)(y+2)\cdots(y+r)),
$$
and the following  $p$-basis of the number $K=\Q(\t)$ field determined by $f$:
$$
\begin{array}{rl}
{\mathcal B}_p=&\left\{    1,\t,\frac{\t^2+b_2}{p},\frac{\t(\t^2+b_2)}{p},\dots,\right.\\ \\
   & \quad\left.\frac{\t^{2r-2}+b_2\t^{2r-4}+\dots+b_{2r-2}}{p^{r-1}},
\frac{\t(\t^{2r-2}+b_2\t^{2r-4}+\dots+b_{2r-2})}{p^{r-1}}\right\}.
\end{array}
$$

We see once more the gain of the SF-OM-algorithm with respect to the classical Montes algorithm, which in this case finds $r$ different types.

Eventually, we might find new factors of the discriminant in the triangulation process of the local bases. For instance, consider the case $f=f_{101,7,11}$. The discriminant   $\dsc(f)$ has 362 decimal figures. The initial square-free factorizations detect the factors $p=101$, and $D=\dsc(f)/(2^{14}   101^{91})$.
The {\tt Global Basis} routine finds $D_1=104065441$, $$D_2=1045681081654964908367435426796223594285104202271660887,$$ and no further splittings.

However, along the triangulation process of the reduced local $D_2$-basis introduced in \cite{mn}, certain valuation computations detect the factors 17 and 37. After all these splittings, in order to compute a global integral basis for $f$, we need to  compute local bases for the moduli:
$$
\begin{array}{l}
2, 17, 37, 101, 104065441, \\
1662450050325858359884635018753932582329259463071003.
\end{array}
$$
Checking the squarefreeness of these  moduli is almost immediate.
Observe that the last two moduli are not prime, but we  have no  need to factor them in order to compute a global integral basis.

\subsection{Example 3}

For an odd positive integer $r$, let
$$
(x-1)(x-2)\cdots(x-r)=x^r+a_{r-1}x^{r-1}+\dots+a_0.
$$

Take a squarefree integer $N>\mx\{|a_0|,\dots,|a_r|\}$. Consider the polynomials:
$$
\as{1.5}
\begin{array}{rl}
\phi&=x^4+N^2(N-1),\\
f&=\phi^{3r}+a_{r-1}N^6x^2\phi^{3(r-1)}+a_{r-2}N^{14}\phi^{3(r-2)}+a_{r-3}N^{20}x^2\phi^{3(r-3)}+\cdots\\
&\qquad\qquad\qquad\qquad\qquad\qquad\qquad\qquad\cdots+\,a_1N^{7(r-1)}\phi^3+ a_0N^{7r-1}x^2.
\end{array}
$$


The SF-OM algorithm applied to the pair $f, N$ generates a unique type of order 2:
$$
\ty_N=(y; (x,-1/2, (y+1)(y-1));(\phi,-2/3,(y-1)\cdots(y-r)).
$$
If $s$ is the number of primes dividing $N$, the type $\ty_N$ has $2rs$ irreducible types attached. We see again that the use of the SF-OM algorithm can be much faster than the classical Montes algorithm, and that it saves a lot of memory space.

\end{document}